\documentclass[aos,preprint]{imsart}
\arxiv{arXiv:1605.05051v5}

\startlocaldefs

\usepackage{xcolor}

\usepackage{amsthm,amsmath,amssymb,amsbsy,bbm,mathrsfs,supertabular,eurosym,graphicx,enumitem}
\usepackage[authoryear]{natbib}
\newcommand{\pa}[1]{\left({#1}\right)}
\newcommand{\norm}[1]{\left\|{#1}\right\|}
\newcommand{\cro}[1]{\left[{#1}\right]}
\newcommand{\ab}[1]{\left|{#1}\right|}
\newcommand{\ac}[1]{\left\{{#1}\right\}}

\def\gpen{\mathop{\rm \bf pen}\nolimits}
\def\pen{\mathop{\rm pen}\nolimits}

\newcommand{\dfleche}[1]{\,\displaystyle{\mathop{\longrightarrow}_{#1}}\,}

\RequirePackage[colorlinks,citecolor=blue,urlcolor=blue]{hyperref}

\def\E{{\mathbb{E}}}
\def\F{{\mathbb{F}}} 
\def\L{{\mathbb{L}}} 
\def\N{{\mathbb{N}}}
\def\P{{\mathbb{P}}}
\def\Q{{\mathbb{Q}}} 
\def\R{{\mathbb{R}}}

\def\sA{{\mathscr{A}}}
\def\sB{{\mathscr{B}}}
\def\sC{{\mathscr{C}}}
\def\sD{{\mathscr{D}}}

\def\sF{{\mathscr{F}}}
\def\sG{{\mathscr{G}}} 
\def\sH{{\mathscr{H}}}
\def\sI{{\mathscr{I}}}

\def\sM{{\mathscr{M}}}

\def\sP{{\mathscr{P}}}
\def\sQ{{\mathscr{Q}}}

\def\sW{{\mathscr{W}}}
\def\sX{{\mathscr{X}}}
\def\sY{{\mathscr{Y}}}

\DeclareMathAlphabet{\mathscrbf}{OMS}{mdugm}{b}{n}

\def\sbB{{\mathscrbf{B}}}

\def\sbE{{\mathscrbf{E}}}

\def\sbH{{\mathscrbf{H}}}

\def\sbM{{\mathscrbf{M}}}

\def\sbP{{\mathscrbf{P}}}
\def\sbQ{{\mathscrbf{Q}}} 
\def\sbR{{\mathscrbf{R}}}
 
\def\sbT{{\mathscrbf{T}}}

\def\sbV{{\mathscrbf{V}}}

\def\sbX{{\mathscrbf{X}}}

\def\cB{{\mathcal{B}}}

\def\cD{{\mathcal{D}}}

\def\cF{{\mathcal{F}}}
\def\cG{{\mathcal{G}}} 
\def\cH{{\mathcal{H}}}
\def\cI{{\mathcal{I}}}
\def\cJ{{\mathcal{J}}}
\def\cK{{\mathcal{K}}}
\def\cL{{\mathcal{L}}} 
\def\cM{{\mathcal{M}}}
\def\cN{{\mathcal{N}}}

\def\cP{{\mathcal{P}}}
\def\cQ{{\mathcal{Q}}} 
\def\cR{{\mathcal{R}}}
 
\def\cT{{\mathcal{T}}}
\def\cU{{\mathcal{U}}}

\def\gh{{\mathbf{h}}}

\def\gp{{\mathbf{p}}}
\def\gq{{\mathbf{q}}} 
\def\gr{{\mathbf{r}}}
\def\gs{{\mathbf{s}}} 
\def\gt{{\mathbf{t}}}
\def\gu{{\mathbf{u}}}

\def\gw{{\mathbf{w}}}
\def\gx{{\mathbf{x}}}

\def\gA{{\mathbf{A}}}

\def\gL{{\mathbf{L}}}

\def\gP{{\mathbf{P}}}
\def\gQ{{\mathbf{Q}}} 
\def\gR{{\mathbf{R}}}
\def\gS{{\mathbf{S}}} 
\def\gT{{\mathbf{T}}}

\def\gZ{{\mathbf{Z}}}

\newcommand{\bs}[1]{\boldsymbol{#1}}
\def\bsX{{\bs{X}}}

\def\gmu{{\bs{\mu}}}

\def\gnu{{\bs{\nu}}}

\def\gup{\bs{\Upsilon}}

\def\cbQ{{\bs{\cQ}}}
\def\cbT{{\bs{\cT}}}

\newtheorem{thm}{Theorem}
\newtheorem{lem}{Lemma}
\newtheorem{prop}{Proposition}
\newtheorem{cor}{Corollary}
\newtheorem{df}{Definition}
\newtheorem{ass}{Assumption}

\newlist{lista}{enumerate}{1}
\setlist[lista,1]{label=\alph*),ref=\alph*),leftmargin=*}

\newlist{listi}{enumerate}{1}
\setlist[listi,1]{label=\roman*),ref=\roman*),leftmargin=*}

\newcommand{\eref}[1]{(\ref{#1})}
\def\1{1\hskip-2.6pt{\rm l}}
\def\<{{\langle}}
\def\>{{\rangle}}

\newcommand{\etc}[1]{#1_1,\ldots,#1_n}
\newcommand{\Etc}[2]{#1_1,\ldots,#1_{#2}}
\newcommand{\str}[1]{\rule{0mm}{#1mm}}
\newcommand{\st}{\strut}

\newcommand{\dps}[1]{\displaystyle{#1}}

\def\eps{{\varepsilon}}
\def\ee{{\epsilon}}
\def\cbL{{\bs{\cL}}}
\def\cbU{{\bs{\cU}}}

\def\co{{c_{0}}}
\def\et{^{\star}}

\endlocaldefs

\begin{document}

\begin{frontmatter}

\title{Rho-estimators revisited: general theory and applications}
\runtitle{Rho-estimators revisited}
\author{\fnms{Yannick} \snm{Baraud}\corref{}\ead[label=e1]{baraud@unice.fr}\thanksref{m1}}
\address{Universit\'e C\^ote d'Azur\\
UMR CNRS 7351\\
Laboratoire Jean Alexandre Dieudonn\'e\\
Parc Valrose\\
06108 Nice cedex 02\\
France.\\ \printead{e1}}
\and
\author{\fnms{Lucien} \snm{Birg\'e}\ead[label=e2]{lucien.birge@upmc.fr}\thanksref{m2}}
\address{Sorbonne Universit\'es\\
 UPMC Univ.\ Paris 06\\
 CNRS UMR 7599\\
 LPMA - Case courrier 188\\
 75252 Paris Cedex 05\\
 France.\\ \printead{e2}}
\affiliation{Universit\'e C\^ote d'Azur \thanksmark{m1} and Sorbonne Universit\'es \thanksmark{m2}}

\maketitle

\begin{abstract}
Following Baraud {\em et al.}~\citeyearpar{MR3595933}, we pursue our attempt to design a robust universal estimator of the joint ditribution of $n$ independent (but not necessarily i.i.d.) observations for an Hellinger-type loss. Given such observations with an unknown joint distribution $\gP$ and a dominated model $\sbQ$ for $\gP$, we build an estimator $\widehat \gP$ based on $\sbQ$ (a $\rho$-estimator) and measure its risk by an Hellinger-type distance. When $\gP$ does belong to the model, this risk is bounded by some quantity which relies on the local complexity of the model in a vicinity of $\gP$. In most situations this bound corresponds to the minimax risk over the model (up to a possible logarithmic factor). When $\gP$ does not belong to the model, its risk involves an additional bias term proportional to the distance between $\gP$ and $\sbQ$, whatever the true distribution $\gP$. From this point of view, this new version of $\rho$-estimators improves upon the previous one described in Baraud {\em et al.}~\citeyearpar{MR3595933} which required that $\gP$ be absolutely continuous with respect to some known reference measure. Further additional improvements have been brought as compared to the former construction. In particular, it provides a very general treatment of the regression framework with random design as well as a computationally tractable procedure for aggregating estimators. We also give some conditions for the Maximum Likelihood Estimator to be a $\rho$-estimator.
Finally, we consider the situation where the Statistician has at his or her disposal many different models and we build a penalized version of the $\rho$-estimator for model selection and adaptation purposes. In the regression setting, this penalized estimator not only allows one to estimate the regression function but also the distribution of the errors.
\end{abstract}



\end{frontmatter}

%SECTION
\section{Introduction}\label{I}
In a previous paper, namely Baraud {\em et al.}~\citeyearpar{MR3595933},
we introduced a new class of estimators that we called $\rho$-estimators for estimating the distribution $\gP$ of a random variable $\bsX=(X_{1},\ldots,X_{n})$ with values in some measurable space $(\sbX,\sbB)$ under the assumption that the $X_{i}$ are independent but not necessarily i.i.d.
These estimators are based on density models, a {\em density model} being a family of densities $\gt$ with respect to some reference measure $\gmu$ on $\sbX$. We also assumed that $\gP$ was absolutely continuous with respect to $\gmu$ with density $\gs$ and, following Le Cam~\citeyearpar{MR0334381}, we measured the performance of an estimator $\widehat{\gs}$ of $\gs$ in terms of $\gh^{2}(\gs,\widehat{\gs})$, where $\gh$ is a Hellinger-type distance to be defined later. Originally, the motivations for this construction were to design an estimator $\widehat \gs$ of $\gs$ with the following properties.

--- Given a density model $\gS$, the estimator $\widehat \gs$ should be nearly optimal over $\gS$ from the minimax point of view, which means that it is possible to bound the risk of the estimator $\widehat \gs$ over $\gS$ from above by some quantity $CD_{n}(\gS)$ which is approximately of the order of the minimax risk over $\gS$.

--- Since in Statistics we typically have uncomplete information about the true distribution of the observations, when we assume that $\gs$ belongs to $\gS$ nothing ever warrants that this is true. We may more reasonably expect that $\gs$ is close to $\gS$ which means that the model $\gS$ is not exact but only approximate and that the quantity $\gh(\gs,\gS)=\inf_{\gt\in\gS}\gh(\gs,\gt)$ might therefore be positive. In this case we would like the risk of $\widehat \gs$ to be bounded by $C'\left[D_{n}(\gS)+\gh^{2}(\gs,\gS)\right]$ for some universal constant $C'$. In the case of $\rho$-estimators, the previous bound can actually be slightly refined and expressed in the following way. It is possible to define on $\gS$ a positive function $R$ such that the risk of the $\rho$-estimator is not larger than $R(\gs)$, with $R(\gs)\le CD_{n}(\gS)$ if $\gs$ belongs to the model $\gS$ and not larger than $C'\inf_{\overline \gs\in\gS}\left[R(\overline{\gs})+\gh^{2}(\gs,\overline{\gs})\right]$ when $\gs$ does not belong to $\gS$.

The weak sensibility of this risk bound to small deviations with respect to the Hellinger-type distance $\gh$ between $\gs$ and an element $\overline \gs$ of $\gS$ covers some classical notions of robustness among which robustness to a possible contamination of the data and robustness to outliers, as we shall see in Section~\ref{S2c}.

There are nevertheless some limitations to the properties of $\rho$-estimators as defined in Baraud {\em et al.}~\citeyearpar{MR3595933}. 
\begin{lista}
\item The study of random design regression required that either the distribution of the design be known or that the errors have a symmetric distribution. We want to relax these assumptions and consider the random design regression framework with greater generality.
\item We always worked with some reference measure $\gmu$ and assumed that all the probabilities we considered, including the true distribution $\gP$ of $\bsX$, were absolutely continuous with respect to $\gmu$. This is quite natural for the probabilities that belong to our models since the models are, by assumption, dominated and, typically, defined via a reference measure $\gmu$ and a family of densities with respect to $\gmu$. Nevertheless, the assumption that the true distribution $\gP$ of the observations be also dominated by $\gmu$ is questionable. We therefore would like to get rid of it and let the true distribution be completely arbitrary, relaxing thus the assumption that the density $\gs$ exists. Unexpectedly, such an extension leads to subtle complications as we shall see below and this generalization is actually far from being straightforward.
\item Our construction was necessarily restricted to countable models rather than the uncoutable ones currently used in Statistics.
\end{lista}
We want here to design a method based on ``probability models" rather than ``density models", which means working with dominated models $\cP$ consisting of probabilities rather than of densities as for $\gS$. Of course, the choice of a dominating measure $\gmu$ and a specific set $\gS$ of densities leads to a probability model $\cP$. This is by the way what is actually done in Statistics, but the converse is definitely not true and there exist many ways of representing a dominated probability model by a reference measure and a set of densities. It turns out --- see Section~\ref{ES1} --- that the performance of a very familiar estimator, namely the MLE (Maximum Likelihood Estimator), can be strongly affected by the choice of a specific version of the densities. Our purpose here is to design an estimator the performance of which only depends on the probability model $\cP$ and not on the choice of the reference measure and the densities that are used to represent it.

In order to get rid of the above-mentioned restrictions, we have to modify our original construction 
which leads to the new version that we present here. This new version retains all the nice properties that we proved in Baraud {\em et al.}~\citeyearpar{MR3595933} and the numerous illustrations we considered there remain valid for the new version. It additionally provides a general treatment of conditional density estimation and regression, allowing the Statistician to estimate both the regression function and the error distribution even when the distribution of the design is totally unknown and the errors admit no finite moments. From this point of view, our approach contrasts very much with that based on the classical least squares. An alternative point of view on the particular problem of estimating a conditional density can be found in Sart~\citeyearpar{Sart:2015kq}. 

A thorough study of the performance of the least squares estimator (or truncated versions of it) can be found in Gy\"orfi {\it et al.}~\citeyearpar{MR1920390} and we refer the reader to the references therein. A nice feature of these results lies in the fact that they hold without any assumption on the distribution of the design. While few moment conditions on the errors are necessary to bound the $\L_{2}$-integrated risk of their estimator, much stronger ones, typically boundedness of the errors, are necessary to obtain exponential deviation bounds. In contrast, in linear regression, Audibert and Catoni~\citeyearpar{MR2906886} established exponential deviation bounds for the risk of some robust versions of the ordinary least squares estimator. Their idea is to replace the sum of squares by the sum of their truncated version in view of designing a new criterion which is less sensitive to possible outliers than the original least squares. Their way of modifying the least squares criterion shares some similarity with our way of modifying the log-likelihood criterion, as we shall see below. However their results require some conditions on the distribution of the design as well as some (weak) moment condition on the errors while ours do not.

It is known, and we shall give an additional example below, that the MLE, which is often considered as a ``universal" estimator, does not possess, in general, the properties that we require and more specifically robustness. An illustration of the lack of robustness of the MLE with respect to Hellinger deviations is provided in Baraud and Birg\'e~\citeyearpar{MR3565484}. Some other weaknesses of the MLE have been described in Le~Cam~\citeyearpar{Lecam-MLE} and Birg\'e~\citeyearpar{MR2243881}, among other authors, and various alternatives aimed at designing some sorts of ``universal" estimators (for the problem we consider here) which would not suffer from the same weaknesses have been proposed in the past by Le Cam~\citeyearpar{MR0334381} and \citeyearpar{MR0395005} followed by Birg\'e~\citeyearpar{MR722129} and \citeyearpar{MR2243881}. The construction of $\rho$-estimators, as described in Baraud {\em et al.}~\citeyearpar{MR3595933} was in this line. In that paper, we actually introduced $\rho$-estimators via a testing argument as was the case for Le Cam and Birg\'e for their methods. This argument remains valid for the generalized version we consider here --- see Lemma~\ref{lem-test} in Appendix~\ref{Hpsi2} --- but $\rho$-estimators can also be viewed as a generalization, and in fact a robustified version, of the MLE. We shall even show, in Section~\ref{EX}, that in favorable situations (i.i.d.\ observations and a convex separable set of densities as a model for the true density) the MLE is actually a $\rho$-estimator and therefore shares their properties.

To explain the idea underlying the construction of $\rho$-estimators, let us assume that we observe an $n$-sample $\bsX=(\etc{X})$ with an unknown density $q$ belonging to a set $\overline{\cQ}$ of densities with respect to some reference measure $\mu$. We may write the log-likelihood of $q$ as $\sum_{i=1}^n\log \left(q(X_i)\st\right)$ and the log-likelihood ratios as
\[
\gL(\bsX,q,q')=\sum_{i=1}^n\log \left(\frac{q'(X_i)}{q(X_{i})}\right)=\sum_{i=1}^n\log \left(q'(X_i)\st\right)-
\sum_{i=1}^n\log \left(q(X_i)\st\right),
\]
so that maximizing the likelihood is equivalent to minimizing with respect to $q$
\[
\gL(\bsX,q)=\sup_{q'\in\overline \cQ}\sum_{i=1}^n\log \left(\frac{q'(X_i)}{q(X_{i})}\right)=\sup_{q'\in\overline{\cQ}}\gL(\bsX,q,q').
\]
This happens simply because of the magic property of the logarithm which says that $\log(a/b)=\log a-\log b$. However, the use of the unbounded log function in the definition of $\gL(\bsX,q)$ leads to various problems that are responsible for some weaknesses of the MLE. Replacing the log function by another function $\varphi$ amounts to replace $\gL(\bsX,q,q')$ by
%beg
\begin{equation}
\gT(\bsX,q,q')=\sum_{i=1}^{n}\varphi\pa{{q'(X_{i})}\over {q(X_{i})}}
\label{eq-Tphi}
\end{equation}
%end
which is different from $\sum_{i=1}^n\varphi\left(q'(X_i)\st\right)-\sum_{i=1}^n\varphi\left(q(X_i)\st\right)$ since $\varphi$ is not the log function. We may nevertheless define the analogue of $\gL(\bsX,q)$, namely
%beg
\begin{equation}
\gup(\bsX,q)=\sup_{q'\in\overline{\cQ}}\gT(\bsX,q,q')=\sup_{q'\in\overline{\cQ}}\sum_{i=1}^n\varphi\left(\frac{q'(X_i)}{q(X_{i})}\right)
\label{eq-gup1}
\end{equation}
%end
and define our estimator $\widehat{q}(\bsX)$ as a minimizer with respect to $q\in\overline{\cQ}$ of the quantity $\gup(\bsX,q)$. The resulting estimator is an alternative to the maximum likelihood estimator and we shall show that, for a suitable choice of a bounded function $\varphi$, it enjoys various properties, among which robustness, that are often not shared by the MLE.

To analyze the performance of this new estimator, we have to study the behaviour of the process $\gT(\bsX,q,q')$ when $q$ is fixed, $q\cdot\mu$ is close to the true distribution of the $X_{i}$ and $q'$ varies in $\overline{\cQ}$. Since the function $\varphi$ is bounded, the process is similar to those considered in learning theory for the purpose of studying empirical risk minimization. As a consequence, the tools we use are also similar to those described in great detail in Koltchinskii~\citeyearpar{MR2329442}.

It is well-known that working with a single model for estimating an unknown distribution 
is not very efficient unless one has very precise pieces of information about the true distribution, which is rarely the case. Working with many models simultaneously and performing model selection improves the situation drastically. Refining the previous construction of $\rho$-estimators by adding suitable penalty terms to the statistic $\gT(\bsX,q,q')$ allows one to work with a finite or countable family of probability models $\{\cP_{m},\,m\in\cM\}$ instead of a single one, each model ${\cP}_{m}$ leading to a risk bound of the form $C'\left[D_{n}(\cP_{m})+\gh^{2}\left(\cP_{m},\gP\right)\right]$, and to choose from the observations a model with approximately the best possible bound which results in a final estimator $\widehat{\gP}$ and a bound for $\gh^{2}(\widehat{\gP},\gP)$ of the form 
\[
C''\inf_{m\in\cM}\left[D_{n}(\cP_{m})+\gh^{2}\left(\cP_{m},\gP\right)+\Delta_{m}\right]
\] 
where the additional term $\Delta_{m}$ is connected to the complexity of the family of models we use. 

The paper is organised as follows. We shall first make our framework, which is based on dominated families of probabilities rather than families of densities with respect to a given dominating measure, precise in Section~\ref{F}. This section is devoted to the definition of models and of our new version of $\rho$-estimators, then to the assumptions that the function $\varphi$ we use to define the statistic $\gT$ in (\ref{eq-Tphi}) should satisfy. 
In Section~\ref{MRI}, we define the $\rho$-dimension function of a model, a quantity which measures the difficulty of estimation within the model using a $\rho$-estimator, and present the main results, namely the performance of these new $\rho$-estimators. Section~\ref{sect-ModGene} is devoted to the extension of the construction from countable to uncountable
statistical models (which are the ones currently used in Statistics) under suitable assumptions. We describe the robustness properties of $\rho$-estimators in Section~\ref{S2c}. In Section~\ref{EX} we 
investigate the relationship between $\rho$-estimators and the MLE when the model is a convex set of densities. Section~\ref{sect-DM} provides various methods that allow one to bound the $\rho$-dimension functions of different types of models and indicates how these bounds are to be used to bound the risk of $\rho$-estimators in typical situations with applications to the minimax risk over classical statistical models. We also provide a few examples of computations of bounds for the $\rho$-dimension function. Many applications of our results about $\rho$-estimators have already been given in Baraud {\em et al.}~\citeyearpar{MR3595933} and we deal here with a new one: estimation of conditional distributions in Section~\ref{CD}. In Section~\ref{RS} we apply this additional result to the special case of random design regression when the distribution of the design is completely unknown, a situation for which not many results are known. We provide here a complete treatment of this regression framework with simultaneous estimation of both the regression function and the density of the errors. Section~\ref{S-ES-A}  is devoted to estimator selection and aggregation: we show there how our procedure can be used either to select an element from a family of preliminary estimators or to aggregate them in a convex way. The Appendices (Supplementary material) contain the proofs as well as some additional facts.

%SECTION
\section{Our new framework and estimation strategy}\label{F}
As already mentioned, our method is based on statistical models which are sets of probability distributions, in opposition with more classical models which are sets of densities with respect to a given dominating measure. 

%SUBSECTION
\subsection{A probabilistic framework}\label{F2}
We observe a random variable $\bsX=(X_{1},\ldots,X_{n})$ defined on some probability space $(\Omega,\Xi,\P)$ with independent components $X_{i}$ and values in the measurable product space $(\sbX,\sbB)=(\prod_{i=1}^{n}\sX_{i},\bigotimes_{i=1}^{n}\sB_{i})$. We denote by $\sbP$ the set of all product probabilities on $(\sbX,\sbB)$ and by $\gP=\bigotimes_{i=1}^{n} P_{i}\in\sbP$ the true distribution of $\bsX$. We identify an element $\gQ=\bigotimes_{i=1}^{n} Q_{i}$ of $\sbP$ with the $n$-tuple $(Q_{1},\ldots,Q_{n})$ and extend this identification to the elements $\gmu=\bigotimes_{i=1}^{n}\mu_{i}$ of the set $\sbM$ of all $\sigma$-finite product measures on $(\sbX,\sbB)$. 

When $\gQ$ is absolutely continuous with respect to $\gmu\in\sbM$ ($\gQ\ll\gmu$) or, equivalently, $\gmu$ dominates $\gQ$, each $Q_{i}$, for $i=1,\ldots,n$, is absolutely continuous with respect to $\mu_{i}$ with density $q_{i}$ so that $Q_{i}=q_{i}\cdot\mu_{i}$. We denote by $\cL(\mu_{i})$ the set of all densities with respect to $\mu_{i}$, i.e.\ the set of measurable functions $q$ from $\sX_{i}$ to $\R_{+}$ such that $\int_{\sX_{i}}q(x)\,d\mu_{i}(x)=1$. We then write $\gQ=\gq\cdot \gmu$ where $\gq$ is the $n$-tuple $(q_{1},\ldots,q_{n})$ and we say that $\gq$ is a density for $\gQ$ with respect to $\gmu$. We denote by $\cbL(\gmu)=\prod_{i=1}^{n}\cL(\mu_{i})$ the set of such densities $\gq$ and by $\sbP^{\gmu}$ the set of  all those $\gP'\in\sbP$ which are absolutely continuous with respect to $\gmu$.

Our aim is to estimate the unknown distribution $\gP=(P_{1},\ldots,P_{n})$ from the observation of $\bsX$. In order to evaluate the performance of an estimator $\widehat \gP(\bsX)\in\sbP$ of $\gP$, we shall introduce, following Le~Cam~\citeyearpar{MR0395005}, an Hellinger-type distance $\gh$ on $\sbP$.
We recall that, given two probabilities $Q$ and $Q'$ on a measurable space $(\sX,\sB)$, the Hellinger distance and the Hellinger affinity between $Q$ and $Q'$ are respectively given by  
%beg
\begin{equation}
\left\{
\begin{array}{ll}  
h^{2}(Q,Q')&=\dps{{1\over 2}\int_{\sX}\pa{\sqrt{dQ\over d\mu}-\sqrt{dQ'\over d\mu}}^{2}d\mu},\\
\rho(Q,Q')&=\dps{\int_{\sX}\sqrt{{dQ\over d\mu}{dQ'\over d\mu}}d\mu=1-h^{2}(Q,Q')},
\end{array}\right.
\label{eq-hell}
\end{equation}
%end
where $\mu$ denotes any measure that dominates both $Q$ and $Q'$, the result being independent of the choice of $\mu$.
The Hellinger-type distance $\gh(\gQ,\gQ')$ and affinity $\bs{\rho}(\gQ,\gQ')$ between two elements $\gQ=(Q_{1},\ldots,Q_{n})$ and $\gQ'=(Q_{1}',\ldots,Q_{n}')$ of $\sbP$ are then given by the formulas
\[
\gh^{2}(\gQ,\gQ')=\sum_{i=1}^{n}h^{2}(Q_{i},Q_{i}')=\sum_{i=1}^{n}\cro{1-\rho(Q_{i},Q_{i}')}=n-\bs{\rho}(\gQ,\gQ').
\]
We shall denote by $\sbV$ the topology of the metric space $(\sbP,\gh)$.

%SUBSECTION
\subsection{Models and their representations}\label{ES0}
Let us start with this  definition:
%
% DEFINITION
\begin{df}\label{def-model}
We call {\em model} any dominated subset $\overline \sbQ$ of $\sbP$ and we call {\em representation} of (the model) $\overline\sbQ$ a pair $\cR(\overline\sbQ)=(\gmu,\overline\cbQ)$ where $\gmu=(\etc{\mu})$ is a $\sigma$-finite measure  which dominates $\overline\sbQ$ and $\overline\cbQ$ is a subset of $\cbL(\gmu)$ such that for any $\gQ$ in $\overline\sbQ$ there exists a unique density $\gq\in \overline\cbQ$ with $\gQ=\gq\cdot \gmu$.
\end{df}
This means that, given a representation $(\gmu,\overline\cbQ)$ of the model $\overline\sbQ$, we can associate to each probability $\gQ\in\overline\sbQ$ a density $\gq\in\overline\cbQ$ and vice-versa. Clearly a dominated subset $\overline\sbQ$ has different representations depending on the choice of the dominating measure $\gmu$ and the versions of the densities $q_{i}=dQ_{i}/d\mu_{i}$. 

Our estimation strategy is based on specific dominated subsets of $\sbP$ that we call {\em $\rho$-models}. 
%
% DEFINITION
\begin{df}\label{def-rhomodel}
A $\rho$-model is a {\em countable} (which in this paper always means either finite or infinite and countable) subset $\sbQ$ of $\sbP$.
\end{df}

A $\rho$-model $\sbQ$ being countable, it is necessarily dominated. One should think of it as a probability set to which the true distribution is believed to be close (with respect to the Hellinger-type distance $\gh$). 

%SUBSECTION   
\subsection{Construction of a $\rho$-estimator on a model $\sbQ$}\label{ES1}
Given the model $\sbQ$, our estimator is defined as a random element of ${\rm Cl}\!\left(\sbQ\right)$, where ${\rm Cl}\!\left(\st\sbR\right)$ denotes the closure of the subset $\sbR$ of $\sbP$ in the metric space $(\sbP,\gh)$, and its construction relies on a particular representation  $\cR(\sbQ)$ of the model $\sbQ$. It actually depends on three elements with specific properties to be made precise below:
%
%\begin{enumerate}[label=\roman*),ref=\alph*.,leftmargin=9mm]
\begin{listi}
\item A function $\psi$ (which will serve as a substitute for the logarithm to derive an alternative to the MLE) with the following properties: 
%
% ASSUMPTION
\begin{ass}\label{ass-psi}
The function $\psi$ is non-decreasing from $[0,+\infty]$ to $[-1,1]$, Lipschitz and satisfies
%beg
\begin{equation}
\psi(x)=-\psi(1/x)\ \ \mbox{for all}\ \ x\in [0,+\infty),\quad\mbox{hence }\psi(1)=0.
\label{H-phi}
\end{equation}
%end
\end{ass}
Throughout this paper we shall only consider, without further notice, functions $\psi$ satisfying Assumption~\ref{ass-psi}.
\item A model $\sbQ\subset \sbP$ (in most cases a $\rho$-model) with a representation $\cR(\sbQ)=(\gmu,\cbQ)$.
\item A {\it penalty function} ``$\gpen$" mapping $\sbQ$ to $\R$, the role of which will be explained later in Section~\ref{MRI}. We may, at first reading, assume that this penalty function is identically 0. 
%
%\end{enumerate}
\end{listi}
It is essential to note that the dominating measure $\gmu$ is chosen by the statistician and that there is no reason that the true distribution $\gP$ of $\bsX$ be absolutely continuous with respect to $\gmu$. On the contrary, all probabilities $\gP'$ on $\sbX$ belonging to ${\rm Cl}\!\left(\st{\sbQ}\right)$ are absolutely continuous with respect to $\gmu$.

Given the function $\psi$ and the representation $\cR(\sbQ)$, we define the real-valued function $\gT$ on $\sbX\times\cbQ\times\cbQ$ by
\begin{equation}\label{def-T}
\gT(\gx,\gq,\gq')=\sum_{i=1}^{n}\psi\pa{\sqrt{{q_{i}'(x_{i})}\over {q_{i}(x_{i})}}}\;\mbox{ for }
\gx=(x_{1},\ldots,x_{n})\in\sbX\;\mbox{ and }\;\gq,\gq'\in {\cbQ},
\end{equation}
with the conventions $0/0=1$ and $a/0=+\infty$ for all $a>0$.  We then set (with $\gQ=\gq\cdot\gmu$ and $\gQ'=\gq'\cdot\gmu$)
%beg
\begin{equation}
\gup(\bsX,\gq)=\sup_{\gq'\in\cbQ}\cro{\gT(\bsX,\gq,\gq')-\gpen(\gQ')}+\gpen(\gQ)\quad\mbox{for all }\gq\in\cbQ.
\label{def-gup}
\end{equation}
%end
%
% DEFINITION
\begin{df}[$\rho$-estimators]\label{def-rhoest}
Let $\sbE(\psi,\bsX)$ be the (nonvoid) set
%beg
\begin{equation}
\sbE(\psi,\bsX)=\ac{\gQ=\gq\cdot \gmu,\;\gq\in\cbQ\;, \; \gup(\bsX,\gq)< \inf_{\gq'\in\cbQ}\gup(\bsX,\gq')+\frac{\kappa}{25}}
\label{def-sE}
\end{equation}
%end
where the positive constant $\kappa$ is given by (\ref{cond-kappa}) below.
A {\em $\rho$-estimator} $\widehat \gP=\widehat \gP(\bsX)$ relative to $(\cR(\sbQ),\gpen)$ is any (measurable) element of ${\rm Cl}\!\left(\st\sbE(\psi,\bsX)\right)$.
\end{df}
Since $\widehat \gP$ belongs to  ${\rm Cl}\!\left(\st\sbE(\psi,\bsX)\right)$, the elements of which are dominated by $\gmu$, there exists a random density $\widehat \gp=(\widehat p_{1},\ldots,\widehat p_{n})$ with $\widehat p_{i}\in\cL(\mu_{i})$ for $i=1,\ldots,n$ such that $\widehat \gP=\widehat \gp\cdot \gmu$. Note that $\widehat \gP$ might not belong to ${\sbQ}$. 

As an immediate consequence of Assumption~\ref{ass-psi} and the convention $1/0=+\infty$,  $\psi(+\infty)=-\psi(0)$ and
%beg
\begin{equation}
\gT(\bsX,\gq,\gq')=-\gT(\bsX,\gq',\gq)\quad\mbox{for all }\gq,\gq'\in\cbQ.
\label{eq-antisym}
\end{equation}
%end
Moreover,
\[
\gup(\bsX,\gq)\ge \cro{\gT(\bsX,\gq,\gq)-\gpen(\gQ)}+\gpen(\gQ)=\gT(\bsX,\gq,\gq)=n\psi(1)=0
\]
for all $\gq\in\cbQ$, which implies that any element $\widehat \gP=\widehat \gp\cdot \gmu$ in ${\sbQ}$ such that $\gup(\bsX,\widehat \gp)<\kappa/25$ is a $\rho$-estimator. In particular, when $\gpen(\gQ)=0$ for all $\gQ\in{\sbQ}$ (which we shall write in the sequel $\gpen={\bf 0}$) and $\gup(\bsX,\widehat \gp)=0$, it follows from (\ref{def-gup}) that
\[
\gT(\bsX,\widehat \gp,\gq)\le\gup(\bsX,\widehat \gp)=0=\gT(\bsX,\widehat \gp,\widehat \gp)\le-\gT(\bsX,\widehat \gp,\gq)=\gT(\bsX,\gq,\widehat \gp)
\]
for all $\gq\in\cbQ$. This means that, in this case, $(\widehat \gp,\widehat \gp)$ is a saddle point of the map $(\gq,\gq')\mapsto \gT(\bsX,\gq,\gq')$.

A $\rho$-estimator $\widehat \gP$ depends on the chosen representation $\cR({\sbQ})$ of ${\sbQ}$ and there are different versions of the $\rho$-estimators associated to ${\sbQ}$, even though, most of the time, ${\sbQ}$ will directly be given by a specific representation, that is a family ${\cbQ}$ of densities with respect to some reference measure $\gmu$. Here is the important point, to be proven in Section~\ref{MRI}: when $\sbQ$ is {\em a $\rho$-model} the risk bounds we shall derive only depend on ${\sbQ}$ and the penalty function but not on the chosen representation of $\sbQ$, which allows us to choose the more convenient one for the construction. In contrast, the performances of many classical estimators are sensitive to the representation of the model ${\sbQ}$ and this is in particular the case of the MLE as shown by the following example.
%
% PROPOSITION
\begin{prop}\label{CEX-MLE}	
Let us consider a sequence of i.i.d.\ random variables $(X_{k})_{k\ge 1}$ defined on a measurable space $(\Omega,\sA,\P)$ with normal distribution $P_{\theta}=\cN(\theta,1)$ for some unknown $\theta\in\R$. We choose for reference measure $\mu=\cN(0,1)$ and for the version of $dP_{\theta}/d\mu$, $\theta\in\R$, the function
%beg
\begin{equation}
p_{\theta}(x)=\exp\cro{\theta x-\left(\theta^{2}/2\right)+(\theta^{2}/2)\exp\left(x^{2}\right)\1_{\theta}(x)\1_{(0,+\infty)}(\theta)}.
\label{eq-bad MLE}
\end{equation}
%end
Whatever the value of the true parameter $\theta$, on a set of probability tending to 1 when $n$ goes to infinity, the MLE is given by $X_{(n)}=\max\{X_{1},\ldots,X_{n}\}$ and is therefore inconsistent.
\end{prop}
The proof of Proposition~\ref{CEX-MLE} is given in Section~\ref{proof-EXgaussien} of the Appendix. Note that the usual choice for $p_{\theta}$: $x\mapsto\exp\left[-x\theta+(\theta^{2}/2)\right]$ for $dP_{\theta}/d\mu$ is purely conventional. Mathematically speaking our choice (\ref{eq-bad MLE}) is perfectly correct but leads to an inconsistent MLE. Also note that the usual tools that are used to prove consistency of the MLE, like bracketing entropy (see for instance Theorem~7.4 of van~de~Geer~\citeyearpar{MR1739079}) are not stable with respect to changes of versions of the densities in the family. The same is true for arguments based on VC-classes that we used in Baraud {\em et al.}~\citeyearpar{MR3595933}. 
Choosing a convenient set of densities  to work with is well-grounded as long as the reference measure $\gmu$ not only dominates the model but also the true distribution $\gP$. If not, sets of null measure with respect to $\gmu$ might have a positive probability  under $\gP$ and it becomes unclear how the choice of this set of densities influences the performance of the estimator. 

%SUBSECTION
\subsection{Notations and conventions}\label{ES3}
Throughout this paper, given a representation $\cR(\overline{\sbQ})=(\gmu,\overline{\cbQ})$ of a model $\overline{\sbQ}$, we shall use lower case letters $\gq,\gq',\ldots$ and $q_{i},q_{i}',\ldots$ for denoting the chosen densities of $\gQ,\gQ',\ldots$ and $Q_{i},Q_{i}',\ldots$ with respect to the reference measures $\gmu$ and $\mu_{i}$ respectively for all $i=1,\ldots,n$. 
We set $\log_{+}(x)=\max\{\log x,0\}$ for all $x>0$; $|A|$ denotes the cardinality of the set $A$; $\sbB(\gP,r)=\left\{\left.\gQ\in\sbP\,\right|\,\gh(\gP,\gQ)\le r\right\}$ is the closed Hellinger-type ball in $\sbP$ with center $\gQ$ and radius $r$. Given a set $E$, a non-negative function $\ell$ on $E\times E$,  $x\in E$ and $A\subset E$,  we set $\ell(x,A)=\inf_{y\in A}\ell(x,y)$. In particular, for $\sbR\subset\sbP$, $\gh(\gP,\sbR)=\inf_{\gR\in\sbR}\gh(\gP,\gR)$. We set $x\vee y$ and $x\wedge y$ for $\max\{x,y\}$ and $\min\{x,y\}$ respectively. By convention $\sup_{\varnothing}=0$, the ratio $u/0$ equals $+\infty$ for $u>0$, $-\infty$ for $u<0$ and 1 for $u=0$.

%SUBSECTION
\subsection{Our assumptions}\label{ES2}
Given the $\rho$-model ${\sbQ}$, let us now indicate what properties the function $\psi$ (satisfying Assumption~\ref{ass-psi}) are required in view of controlling the risk of the resulting $\rho$-estimators.
\begin{ass}\label{H-debase}\mbox{}
Let ${\sbQ}$ be the $\rho$-model to be used for the construction of $\rho$-estimators.
There exist three positive constants $a_{0}$, $a_{1}$, $a_{2}$ with $a_{0}\ge1\ge a_{1}$ and $a_{2}^{2}\ge 1\vee(6a_{1})$ such that, whatever the representation $\cR({\sbQ})=(\gmu,{\cbQ})$ of ${\sbQ}$, the densities $\gq,\gq'\in {\cbQ}$, the probability $\gR\in\sbP$ and $i\in\{1,\ldots,n\}$,
%beg
\begin{align}
\int_{\sX_{i}}\psi\pa{\sqrt{q'_{i}\over q_{i}}}\,dR_{i}&\le a_{0}h^{2}(R_{i},Q_{i})-a_{1}h^{2}(R_{i},Q'_{i})\label{eq-esp}\\
\int_{\sX_{i}}\psi^{2}\pa{\sqrt{q'_{i}\over q_{i}}}\,dR_i&\le a_{2}^{2}\cro{h^{2}(R_{i},Q_{i})+h^{2}(R_{i},Q'_{i})}.
\label{eq-var}
\end{align}
%end
\end{ass}
Note that the left-hand sides of~\eref{eq-esp} and~\eref{eq-var} depend on the choices of the reference measures $\mu_{i}$ and versions of the densities $q_{i}=dQ_i/d\mu_i$ and $q_{i}'=dQ'_i/d\mu_i$ while the corresponding right-hand sides do not. 

Given $\psi$ that satisfies Assumption~\ref{H-debase}, the values of $a_{0}$, $a_{1}$ and $a_{2}$ are clearly not uniquely defined but, in the sequel, when we shall say that Assumption~\ref{H-debase} holds, this will mean that the function $\psi$ satisfies \eref{eq-esp} and \eref{eq-var} with given values of these constants which will therefore be considered as fixed once $\psi$ has been chosen. When we shall say that some quantity depends on $\psi$ it will implicitely mean that it depends on these chosen values of $a_{0}$, $a_{1}$ and $a_{2}$. 

An important consequence of (\ref{eq-antisym}), \eref{eq-esp} and \eref{eq-var} is the fact that, for all $\gQ$, $\gQ'$ in ${\sbQ}$ and $\gP\in\sbP$,
%beg
\begin{align}
a_{1}\gh^{2}(\gP,\gQ)-a_{0}\gh^{2}(\gP,\gQ')&\le \E\cro{\gT(\bsX,\gq,\gq')}\label{eq-ET}\\
& \le a_{0}\gh^{2}(\gP,\gQ)-a_{1}\gh^{2}(\gP,\gQ').\nonumber 
\end{align}
%end
These inequalities follow by summing the inequalities (\ref{eq-esp}) with respect to $i$ with $\gR=\gP$, then exchanging the roles of $\gQ$ and $\gQ'$ and applying \eref{eq-antisym}. They imply that the sign of $\E\cro{\gT(\bsX,\gq,\gq')}$ tells us which of the two distributions $\gQ$ and $\gQ'$ is closer to the true one when the ratio between the distances $\gh(\gP,\gQ)$ and $\gh(\gP,\gQ')$ is far enough from one.

In view of checking that a given function $\psi$ satisfies Assumption~\ref{H-debase}, the next result to be proved in Section~\ref{sect-pprop-universel} of the Appendix is useful. 
%
% PROPOSITION
\begin{prop}\label{prop-universel}
If, for a particular representation $\cR({\sbQ})=(\gmu,{\cbQ})$ of the $\rho$-model ${\sbQ}$ and any probability $\gR\in\sbP^{\gmu}$, the function $\psi$ satisfies~\eref{eq-esp} and~\eref{eq-var} for positive constants $a_{0}>2$, $a_{1}\le [(a_{0}-2)/2]\wedge 1$ and $a_{2}^{2}\ge 1\vee(6a_{1})$, then it satisfies Assumption~\ref{H-debase} with the same constants $a_{0},a_{1}$ and $a_{2}$.
\end{prop}
This proposition means that, up to a possible adjustment of the constants $a_{0}$ and $a_{1}$, it is actually enough to check that~\eref{eq-esp} and~\eref{eq-var} hold true for a given representation $(\gmu,{\cbQ})$ of ${\sbQ}$ and all probabilities $\gR\ll \gmu$.

Let us now introduce two functions $\psi$ which do satisfy Assumption~\ref{H-debase}. 
%
% PROPOSITION
\begin{prop}\label{prop-expsi}
Let $\psi_{1}$ and $\psi_{2}$ be the functions taking the value 1 at $+\infty$ and defined for $x\in\R_{+}$ by
\[
\psi_{1}(x)={x-1\over \sqrt{x^{2}+1}}\qquad\mbox{and}\qquad\psi_{2}(x)={x-1\over x+1}.
\]
These two functions are continuously increasing from $[0,+\infty]$ to $[-1,1]$, Lipschitz (with respective Lipschitz constants 1.143 and 2) and satisfy Assumption~\ref{H-debase} for all $\rho$-models ${\sbQ}$ with $a_{0}=4.97$, $a_{1}=0.083$, $a_{2}^{2}=3+2\sqrt{2}$ for $\psi_{1}$ and $a_{0}=4$, $a_{1}=3/8$, $a_{2}^{2}=3\sqrt{2}$ for $\psi_{2}$.
\end{prop}
Both functions can therefore be used everywhere in the applications of the present paper. Nevertheless, we prefer $\psi_{2}$ because it leads to better constants in the risk bounds of the estimator. Proposition~\ref{prop-expsi} is proved in Appendix~\ref{pprop-expsi}. Some comments on  Assumption~\ref{H-debase} can be found in Appendix~\ref{Hpsi1}. When the $\rho$-model reduces to two elements, our selection procedure can be interpreted as a robust test between two simple hypotheses. Upper bounds on the errors of the first and second kinds are established in Appendix~\ref{Hpsi2}.

%SECTION
\section{The performance of $\rho$-estimators on $\rho$-models}\label{MRI}

%SUBSECTION
\subsection{The $\rho$-dimension function}\label{MRIa}
The deviation $\gh(\gP,\widehat \gP)$ between the true distribution $\gP$ and a $\rho$-estimator
$\widehat \gP$ built on the $\rho$-model ${\sbQ}$ is controlled by two terms which are the analogue of the classical bias and variance terms and we shall first introduce a function that replaces here the variance. 

Let $y>0$, $\gP,\overline \gP\in\sbP$ and $\sbP_{0}$ be an arbitrary subset of $\sbP$, we define
\[
\sbB^{\sbP_{0}}(\gP,\overline \gP,y)=\ac{\gQ\in\sbP_{0}\,\left|\,\gh^{2}(\gP,\overline\gP)+\gh^{2}(\gP,\gQ) < y^{2}\right.}
\]
and for measurable non-negative functions $\gq,\gq'$ on $(\sbX,\sbB)$, we set
%beg
\begin{equation}
\gZ(\bsX,\gq,\gq')=\gT(\bsX,\gq,\gq')-\E\cro{\gT(\bsX,\gq,\gq')}.
\label{eq-gZ}
\end{equation}
%end

Given a representation $\cR=(\gmu, \cbQ)$ of $\sbQ\cup\{\overline \gP\}$, we define 
%beg
\begin{equation}
w(\cR,\sbQ,\gP,\overline \gP,y)=\E\cro{\sup_{\gQ\in\sbB^{\sbQ}(\gP,\overline \gP,y)}\ab{\st\gZ(\bsX,\overline{\gp},\gq)}}
\label{eq-w1}
\end{equation}
%end
where, for $\gQ\in \sbB^{\sbQ}(\gP,\overline \gP,y)\subset \sbQ$, $\gq$ denotes the (unique) element of $\cbQ$ such as $\gQ=\gq\cdot \gmu$ and $\overline \gp$ denotes the element of $\cbQ$ such that $\overline \gP=\overline \gp\cdot \gmu$. We recall that we use the convention $\sup_{\varnothing}=0$. Since $\sbQ$ is countable, so is $\sbB^{\sbQ}(\gP,\overline \gP,y)\subset \sbQ$. Therefore the supremum of $\ab{\gZ(\bsX,\overline \gp,\cdot)\st}$ over $\sbB^{\sbQ}(\gP,\overline \gP,y)$ is measurable and the right-hand side of~\eref{eq-w1} is well-defined. Also note that, since $\gT(\bsX,\overline \gp,\overline \gp)=n\psi(1)=0$,
\[
\E\cro{\sup_{\gQ\in\sbB^{\sbQ}(\gP,\overline \gP,y)}\ab{\st\gZ(\bsX,\overline{\gp},\gq)}}=\E\cro{\sup_{\gQ\in\sbB^{\sbQ\cup\{\overline \gP\}}(\gP,\overline \gP,y)}\ab{\st\gZ(\bsX,\overline{\gp},\gq)}}.
\]
Hence $w(\cR,\sbQ,\gP,\overline \gP,y)=w(\cR,\sbQ\cup\{\overline \gP\},\gP,\overline \gP,y)$. 
%
% DEFINITION
\begin{df}[$\rho$-dimension function]\label{def-dimfunc}
Let $\sbQ$ be a $\rho$-model and $\psi$ some function satisfying  Assumption~\ref{H-debase} with constants $a_{0},a_{1}$ and $a_{2}$. The {\em $\rho$-dimen\-sion function} $D^{{\sbQ}}$ of $\sbQ $ is the mapping from $\sbP\times\sbP$ to $[1,+\infty)$ given by 
%beg
\begin{equation}
D^{{\sbQ}}(\gP,\overline \gP)=\cro{\beta^{2}\sup\ac{y^{2}\,\left|\,\gw^{{\sbQ}}\left(\gP,\overline \gP,y\right)>{a_{1}y^{2}\over 8}\right.}}\bigvee 1
\label{def-dim1}
\end{equation}
%end
with $\beta={a_{1}/(4a_{2})}$ and
\[
\gw^{\sbQ}(\gP,\overline \gP,y)=\inf_{\cR}w(\cR,\sbQ,\gP,\overline \gP,y)\quad\mbox{for all }y>0,
\]
where the infimum runs over all the representations $\cR=(\gmu,\cbQ)$ of $\sbQ\cup \{\overline \gP\}$. 
\end{df}
Note that the $\rho$-dimension function of $\sbQ$ depends on the choice of the function $\psi$ and not on the choice of the representations of $\sbQ\cup\{\overline \gP\}$. 
Since it measures the local fluctuations of the centred empirical process $\gZ(\bsX,\overline \gp,\gq)$ indexed by $\gq\in{\cbQ}$, it is quite similar to the local Rademacher complexity introduced in Koltchinskii~\citeyearpar{MR2329442} for the purpose of studying empirical risk minimization. Its importance comes from the following property.
%
% PROPOSITION
\begin{prop}\label{prop-gw}
Let $\sbQ$ be a $\rho$-model and $\cR=(\gmu,\cbQ)$ an arbitrary representation of $ \sbQ\cup\{\overline \gP\}$. Whatever $\gP,\overline\gP\in\sbP$,
%beg
\begin{equation}
w(\cR,\sbQ,\gP,\overline \gP,y)\le\gw^{\sbQ}(\gP,\overline \gP,y)+8\gh^{2}(\gP,\overline \gP)\;\mbox{ for all }y>0,
\label{eq-w/w}
\end{equation}
%end
hence,  for all $y>\beta^{-1}\sqrt{D^{{\sbQ}}(\gP,\overline \gP)}$
%beg
\begin{equation}
w(\cR,\sbQ,\gP,\overline \gP,y)\le \left(a_{1}y^{2}/8\right)+8\gh^{2}(\gP,\overline \gP).
\label{eq-w/D}
\end{equation}
%end
\end{prop}
The proof is provided in Section~\ref{sect-pprop-gw} of the Appendix.

%SUBSECTION
\subsection{Exponential deviation bounds}\label{MRIb}
Our first theorem, to be proven in Section~\ref{sect-main1}, deals with the situation of a null penalty function $\gpen={\bf 0}$.
%
% THEOREM
\begin{thm}\label{thm-main1}
Let $\gP$ be an arbitrary distribution in $\sbP$, ${\sbQ}$ a $\rho$-model and $\psi$ a function satisfying Assumption~\ref{H-debase}. Whatever the representation $\cR$ of $\sbQ$, a $\rho$-estimator $\widehat \gP$ relative to $(\cR,{\bf 0})$ as defined in Section~\ref{ES1} satisfies, for all $\overline{\gP}\in{\sbQ}$ and $\xi>0$,
%beg
\begin{equation}
\P\left[\gh^{2}(\gP,\widehat \gP)\le\gamma\gh^{2}(\gP,\overline \gP)+{4\kappa\over a_{1}}\pa{\frac{D^{{\sbQ}}(\gP,\overline\gP)}{4.7}+1.49+\xi}\right]\ge1-e^{-\xi},
\label{eq-risk1}
\end{equation}
%end
with
%beg
\begin{equation}
\gamma=\frac{4(a_{0}+8)}{a_{1}}+2+\frac{84}{a_{2}^{2}}\quad\mbox{and}\quad
\kappa=\frac{35a_{2}^{2}}{a_{1}}+74,\quad\mbox{hence}\quad\frac{\kappa}{25}\ge11.36.
\label{cond-kappa}
\end{equation}
%end
In particular, if the $\rho$-dimension function $D^{\sbQ}$ is bounded on $\sbP\times \sbQ$ by $D_{n}\ge 1$, then
%beg
\begin{equation}
\P\cro{C\gh^{2}(\gP,\widehat \gP)\le \gh^{2}(\gP,\sbQ)+D_{n}+\xi}\ge 1-e^{-\xi},\quad \mbox{for all $\xi>0$}
\label{eq-risk1b}
\end{equation}
and for some constant $C>0$ which only depends on the choice of $\psi$.
\end{thm}
None of the quantities involved in \eref{eq-risk1} depends on the chosen representation $\cR$ of $\sbQ$, which means that the performance of $\widehat \gP$ does not depend on $\cR$ although its construction depends on it. We shall therefore (abusively) refer to $\widehat \gP$ as {\em a $\rho$-estimator on $\sbQ$} omitting to mention what representation is used for its construction.

Introducing a non-trivial penalty function allows one to favour some probabilities as compared to others in $\sbQ$ and gives thus a Bayesian flavour to our estimation procedure. We shall mainly use it when we have at our disposal not only one single $\rho$-model for $\gP$ but rather a countable collection $\{{\sbQ}_{m},\ m\in \cM\}$ of candidate ones, in which case ${\sbQ}=\bigcup_{m\in\cM}{\sbQ}_{m}$ is still a $\rho$-model that we call the {\em reference $\rho$-model}. The penalty function may not only be used for estimating $\gP$ but also for performing model selection among the family $\{{\sbQ}_{m},\ m\in \cM\}$ by deciding that the procedure selects the $\rho$-model ${\sbQ}_{\widehat m}$ if the resulting estimator $\widehat \gP$ belongs to ${\sbQ}_{\widehat m}$. Since $\widehat \gP$ may belong to several $\rho$-models, this selection procedure may result in a (random) set of possible $\rho$-models for $\gP$ and a common way of selecting one is to choose that with the smallest {\it complexity} in a suitable sense. In the present paper,  the complexity of a $\rho$-model ${\sbQ}_{m}$ will be measured by means of a non-negative  weight function $\Delta$ mapping $\cM$ into $\R_{+}$ and which satisfies
%beg
\begin{equation}
\sum_{m\in\cM}e^{-\Delta(m)}\le 1,
\label{eq-delta}
\end{equation}
%end
where the number ``1" is chosen for convenience. When equality holds in~\eref{eq-delta}, $e^{-\Delta(\cdot)}$ can be viewed as a prior distribution on the family of $\rho$-models $\{{\sbQ}_{m},\ m\in \cM\}$.

In such a context, we shall describe how our penalty term should depend on this weight function $\Delta$ in view of selecting a suitable $\rho$-model for $\gP$. The next theorem is proved in Section~\ref{sect-main4}.
%
% THEOREM
\begin{thm}\label{thm-main4}
Let $\gP$ be an arbitrary distribution in $\sbP$, $\{{\sbQ}_{m},\ m\in\cM\}$ be a countable collection of $\rho$-models, $\Delta$ a weight function satisfying~\eref{eq-delta}, $\cR(\sbQ)$ a representation of $\sbQ=\bigcup_{m\in\cM}\sbQ_{m}$, $\psi$ a function satisfying Assumption~\ref{H-debase} and $\kappa$ be given by \eref{cond-kappa}. Assume that there exists a mapping $D_{n}:\cM\to \R_{+}$ and a number $K\ge 0$ such that, whatever $m\in\cM$,
%beg
\begin{equation}
D^{{\sbQ}_{m}}(\gP,\overline \gP)\le D_{n}(m)+ KD_{n}(m') \quad\mbox{for all }\,(\gP,\overline{\gP})\in\sbP\times {\sbQ}_{m'}.
\label{eq-L4}
\end{equation}
%end
Let the penalty function satisfies, for some constant $\kappa_{1}\in\R$,
%beg
\begin{equation}
\gpen(\gQ)=\kappa_{1}+\kappa\inf_{\{m\in\cM\,|\,{\sbQ}_{m}\ni\gQ\}}\left[\frac{D_{n}(m)}{4.7}+\Delta(m)\right]\quad\mbox{for all }\gQ\in{\sbQ}.
\label{eq-L3}
\end{equation}
%end
Then any $\rho$-estimator $\widehat \gP$ relative to $(\cR(\sbQ),\gpen)$ satisfies, for all $\xi>0$ with probability at least $1-e^{-\xi}$ and with $\gamma$ given by (\ref{cond-kappa}),
%beg
\begin{align}
\gh^{2}(\gP,\widehat \gP)\le  &\inf_{m\in\cM}\cro{\gamma\gh^{2}\left(\gP,{\sbQ}_{m}\right) +{4\kappa\over a_{1}}\pa{{K+1\over 4.7}D_{n}(m)+\Delta(m)}}
\label{eq-Th2-Prob}\\
& +{4\kappa\over a_{1}}(1.49+\xi),\nonumber 
\end{align}
%end
\end{thm}

%SUBSECTION
\subsection{The case of density estimation}\label{Dens}
Of special interest is the situation where the $X_{i}$ are {\em assumed} to be i.i.d.\ with values in a measurable set $(\sX,\sB)$ in which case $\sbX=\sX^{n}$, $\sbB=\sB^{\otimes n}$, $\sP$ and $\sM$ denote respectively the set of all probability distributions and all positive $\sigma$-finite measures on $(\sX,\sB)$ and $\gP$ is {\em expected} (although this is not necessarily true) to belong to $\sP^{n}=\left\{P^{\otimes n},\;P\in\sP\right\}$. Note that the Hellinger distance $h(\cdot,\cdot)$ on $\sP$ is related to the Hellinger-type distance $\gh(\cdot,\cdot)$ on $\sP^{n}$ in the following way: 
\[
\gh^{2}(\gQ,\gQ')=nh^{2}(Q,Q')\;\;\mbox { for all }\;Q,Q'\in\sP\,\mbox{ with }\,\gQ=Q^{\otimes n},\;\gQ'=(Q')^{\otimes n}.
\] 
If $\gP=P^{\otimes n}\in\sP^{n}$, estimating $\gP$ then amounts to estimating the marginal distribution $P$ and we model the probability $P$ rather than $\gP$. 
%
% DEFINITION
\begin{df}\label{def-densityrhomodel}
We call {\em density $\rho$-model} any countable subset $\sQ$ of $\sP$.
\end{df}
Given such a density $\rho$-model ${\sQ}$ for $P$ with representation $(\mu, \cQ)$ (which implies that the mapping $q\mapsto Q=q\cdot\mu$ is one to one), the corresponding $\rho$-model for $\gP$ is simply ${\sbQ}=\left\{\gQ=Q^{\otimes n},\ Q\in {\sQ}\right\}$ with representation $(\gmu,\cbQ)$, $\gmu=\mu^{\otimes n}$ and $\cbQ=\{\gq:(x_{1},\ldots,x_{n})\mapsto \left(q(x_{1})\ldots q(x_{n})\st\right),\; q\in \cQ\}$. In this case, for simplicity, we write $\gT(\bsX,q,q')$ and $\gup(\bsX,q)$ for $\gT(\bsX,\gq,\gq')$ and $\gup(\bsX,\gq)$ respectively. Examples involving density estimation will be considered in Sections~\ref{S2c}, \ref{EX}, \ref{CD} and \ref{RS} below.

We may also work with several density $\rho$-models $\{{\sQ}_{m},\; m\in\cM\}$ for $P$ simultaneously, in which case ${\sQ}=\bigcup_{m\in\cM}{\sQ}_{m}$ is also a density $\rho$-model. A penalty function $\pen$ on $\sQ$ leads to a penalty function $\gpen$ on $\sbQ=\bigcup_{m\in\cM}\sbQ_{m}$ defined by $\gpen(\gQ)=\gpen(Q^{\otimes n})=\pen(Q)$ for all $Q\in\sQ$. Any $\rho$-estimator $\widehat \gP$ relative to $((\gmu,\cbQ),\gpen)$ is of the form $\widehat \gP=\widehat P^{\otimes n}$ with $\widehat P\in{\rm Cl}(\sQ)$ and $\widehat P$ will be called a {\em density $\rho$-estimator} for $P$ relative to $((\mu,\cQ),\pen)$. 

We deduce that, under the assumptions of Theorem~\ref{thm-main1}, if $\gP$ is truly of the form $\gP=P^{\otimes n}$, for all $\overline P\in \sQ$,
\[
\P\left[Ch^{2}(P,\widehat P)\le h^{2}(P,\overline P)+\frac{D^{{\sbQ}}\left(\gP,\overline P^{\otimes n}\right)}{n}+{\xi\over n}\right]\ge1-e^{-\xi},\quad \mbox{for all $\xi>0$}.
\]
Under the assumptions of Theorem~\ref{thm-main4}, for all $\xi>0$ and a positive constant $C$ depending only on $\psi$,
\[
\P\cro{Ch^{2}(P,\widehat P)\le \inf_{m\in\cM}\cro{h^{2}\left(P,{\sQ}_{m}\right)+{D_{n}(m)+\Delta(m)\over n}}+{\xi\over n}}\ge 1-e^{-\xi}.
\]

%SECTION
\section{From $\rho$-models to uncountable statistical models}\label{sect-ModGene}
The previous results apply to statistical models $\sbQ$ that are countable, which is not the common case in statistics. The aim of this section is to explain how our general theory on $\rho$-models can be used to solve estimation problems on models that are possibly uncountable. Hereafter we shall denote by $\overline \sbQ$ a {\em general statistical model}, i.e.\ an arbitrary subset of $\sbP$.

%SUBSECTION
\subsection{Working with nets}

Let us first recall this classical definition.
%
% DEFINITION
\begin{df}\label{def-net}
Given $\eta\ge 0$, a subset $\sbQ$ of $\overline \sbQ$ such that $\gh(\gQ,\sbQ)\le \eta$ for all $\gQ\in\overline \sbQ$ is called an $\eta$-net of $\overline \sbQ$. The case $\eta=0$ corresponds to the situation where $\sbQ$ is $\sbV$-dense in $\overline \sbQ$.
\end{df}
If there exists a countable $\eta$-net $\sbQ$ for $\overline \sbQ$, it is a $\rho$-model. If its $\rho$-dimension $D^{\sbQ}$ is bounded by $D_{n}=D_{n}(\eta)\ge 1$ on $\sbP\times \sbQ$, we deduce from Theorem~\ref{thm-main1} and the inequality $\gh(\gP,\sbQ)\le \gh(\gP,\overline \sbQ)+\eta$ that any $\rho$-estimator on $\sbQ$ satisfies
%beg
\begin{equation}
\P\cro{C\gh^{2}(\gP,\widehat \gP)\le \gh^{2}(\gP,\overline \sbQ)+D_{n}(\eta)+\eta^{2}+\xi}\ge 1-e^{-\xi}\quad\mbox{for all }\xi>0
\label{eq-FRM-GM0}
\end{equation}
%end
hence, 
%beg
\begin{equation}
\E\cro{\gh^{2}(\gP,\widehat \gP)}\le C'\left[\gh^{2}(\gP,\overline \sbQ)+D_{n}(\eta)+\eta^{2}\right]
\label{eq-FRM-GM}
\end{equation}
%end
for some constants $C,C'>0$ depending on $\psi$ only.  Most of the statistical models $\overline \sbQ$ that are used in statistics possess $\eta$-nets for all values of $\eta\ge 0$. Since the $\rho$-dimension function $D^{\sbQ}$ can only increase with inclusion, choosing for each $\eta\ge 0$ an $\eta$-net with the smallest possible cardinality and then the value $\eta\et$ of $\eta$ that minimizes $D_{n}(\eta)+\eta^{2}$ leads to a $\rho$-estimator $\widehat \gP$ with the smallest possible risk bound in \eref{eq-FRM-GM}. This risk bound turns out to be minimax (up to possible extra logarithmic factors) in all cases we know --- see Section~\ref{DM1} ---.
 
\subsection{Models that are universally separable}\label{sect-UnivSep}
Following Pollard~\citeyearpar{MR762984}, we shall say that a class of densities $\overline  \cbQ\subset \cbL(\gmu)$ is {\em universally separable} if one can find a countable subset $\cbQ\subset \overline \cbQ$ such that, for each $\gq\in\overline \cbQ$, there exists a sequence $\left(\gq^{(j)}\right)_{j\ge1}$ in $\cbQ$ which converges toward $\gq$ pointwise, that is, 
%beg
\begin{equation}
q_{i}^{(j)}(x)\dfleche{j\rightarrow+\infty}q_{i}(x)\quad\mbox{for }1\le i\le n\mbox{ and all }
x\in\sX_{i}.
\label{eq-pCV}
\end{equation}
%end
We shall then say that $\cbQ$ is {\em $\sbT$-dense} in $\overline  \cbQ$. Note that if $\left(\gq^{(j)}\right)_{j\ge1}$ converges toward $\gq$ pointwise, by Scheff\'e's Lemma, the sequence of probabilities $\gQ_{j}=\gq^{(j)}\cdot \gmu$ converges in total variation, hence in Hellinger distance, toward $\gQ=\gq\cdot \gmu$. This implies that if $\cbQ$ is $\sbT$-dense in $\overline \cbQ$, the set of probabilities $\sbQ=\{\gq\cdot \gmu,\;\gq\in \cbQ\}$ is $\sbV$-dense in $\overline \sbQ=\{\gq\cdot \gmu,\;\gq\in \overline \cbQ\}$.

We shall work here within the following framework.
For some $\gmu\in\sbM$, let $\{\overline \cbQ_{m},\;m\in\cM\}$ be a countable family of universally separable subsets of $\cbL(\gmu)$ with $\cbQ_{m}\subset \overline \cbQ_{m}$ a countable and $\sbT$-dense subset of $\overline \cbQ_{m}$. We set $\overline \cbQ=\bigcup_{m\in\cM}\overline \cbQ_{m}$, $\cbQ=\bigcup_{m\in\cM}\cbQ_{m}$, $\overline \sbQ_{m}=\{\gq\cdot \gmu,\;\gq\in \overline \cbQ_{m}\}$ for all $m\in\cM$,  $\overline \sbQ=\{\gq\cdot \gmu,\;\gq\in \overline \cbQ\}$ and $\sbQ=\{\gq\cdot \gmu,\;\gq\in \cbQ\}$. Note that $\sbQ$ is a $\rho$-model since $\cbQ$ is countable and that $\sbQ$ is $\sbV$-dense in $\overline \sbQ$ since $\cbQ$ is $\sbT$-dense in $\overline \cbQ$. Let now $\overline \gpen$ be some penalty function on $\overline \sbQ$ with the following property. 
%
% ASSUMPTION
\begin{ass}\label{Hypo-pen}
There exists a function $p:\cM\to \R$ such that
%beg
\begin{equation}
\overline{\gpen}(\gQ)=\inf_{m\in \cM, \overline \sbQ_{m}\ni \gQ}p(m)\quad \mbox{for all}\quad \gQ\in\overline \sbQ
\label{eq-Hypo-pen}
\end{equation}
%end
and for any $\gQ\in\overline \sbQ$, there exists some $m_{\gQ}\in\cM$ such that $\gQ\in\overline \sbQ_{m_{\gQ}}$ and $\overline{\gpen}(\gQ)=p(m_{\gQ})$.
\end{ass}
Note that this assumption holds in particular in the case of a single model with $\overline{\gpen}={\bf 0}$. Within this framework, we can prove the following result.
%
% THEOREM
\begin{thm}\label{thm-Exten}
Let $\{\overline \cbQ_{m},; m\in\cM\}$ be a countable family of universally separable subsets of $\cbL(\gmu)$ and $\overline \gpen$ a penalty function on $\overline \sbQ$ that satisfies Assumption~\ref{Hypo-pen}. Any $\rho$-estimator $\widehat{\gP}$ on $\overline \sbQ$ relative to $((\gmu, \overline\cbQ),\overline \gpen)$ is also a $\rho$-estimator on the $\rho$-model $\sbQ$ relative to $((\gmu,\cbQ),\gpen)$ where $\gpen$ is the restriction of $\overline \gpen$ to $\sbQ$.
\end{thm}
The proof is postponed to Section~\ref{sect-pthm-Exten} of the Appendix. 

This result says that, provided that the penalty function satisfies \eref{eq-Hypo-pen}, which is consistent with \eref{eq-L3}, the construction of a $\rho$-estimator on the possibly uncountable set  $\overline \sbQ$ with representation $(\gmu, \overline\cbQ)$ actually results in a $\rho$-estimator based on the $\rho$-model  $\sbQ$. 

As soon as we can control the $\rho$-dimension function of $\sbQ$ by the some features of $\overline \sbQ$, in the case of a single model, or the $\rho$-dimension functions of the $\rho$-models $\sbQ_{m}=\{\gq\cdot \gmu,\;\gq\in\cbQ_{m}\}$ by the features of the models $\overline \sbQ_{m}$, in the general case, we are able to bound the risk of the $\rho$-estimator relative to $\left((\gmu,\overline \cbQ),\overline \gpen\right)$ using the results of Theorems~\ref{thm-main1} and~\ref{thm-main4}.

For illustration, let us mention a few examples of density sets that are universally separable: 
\begin{lista}
\item the set $\overline \cH_{D}$ of right-continuous histograms on $\R$ with at most $D\ge 1$ pieces;
\item for $L>0$ and $\alpha=r+\beta$ with $r\in\N$, $\beta\in (0,1]$, the set $\overline\cH_{\alpha}(L)$ of functions $f$ on $[0,1]$ that are $r$-times differentiable and satisfy 
\[
\ab{f^{(r)}(x)-f^{(r)}(y)}\le L\ab{x-y}^{\beta}\quad \mbox{for all }x,y\in [0,1];
\]
\item the set $\overline\cH_{\downarrow}$ of non-increasing and right-continuous densities on $(0,+\infty)$.
\end{lista}
The set $\overline\cH_{\alpha}(L)$ is universally separable because the larger set consisting of continuous functions on $[0,1]$ is separable for the topology induced by the norm of the uniform convergence, hence all its subsets are separable with respect to this topology which implies pointwise convergence. We prove that the sets $\overline\cH_{D}$ and $\overline\cH_{\downarrow}$ are universally separable in Section~\ref{Sect-UnivSep} of the Appendix. We shall see in Section~\ref{EX} that the MLE on the convex density sets $\overline\cH_{\alpha}(L)$ and $\overline\cH_{\downarrow}$ is actually a $\rho$-estimator.
 
%SECTION
\section{Why is a $\rho$-estimator robust?}\label{S2c}
The aim of this section is to analyse the robustness properties of $\rho$-estimators. 
For the sake of simplicity, we shall restrict ourselves to the particular case of density estimation as described in Section~\ref{Dens}. 

%SUBSECTION 
\subsection{Misspecification and contamination}\label{Mis&cont}
We assume here that we work with a single $\rho$-model ${\sbQ}$ (so that Theorem~\ref{thm-main1} applies) for which $D^{{\sbQ}}(\gP,\overline \gP)$ is bounded from above independently of $\gP\in \sbP$ and $\overline \gP\in {\sbQ}$ by some some number $D_{n}({\sQ})\ge 1$ depending on the marginal model ${\sQ}$ and the number $n$ of marginals. Examples of such situations will be provided in Section~\ref{sect-DM}. 

When $\gP=P^{\otimes n}$, that is when the data are truely i.i.d.\ with marginal distribution $P$, \eref{eq-risk1} becomes
%beg
\begin{equation}
\P\left[Ch^{2}(P,\widehat P)\le h^{2}\!\left(P,\sQ\right)+n^{-1}\!\left[D_{n}(\sQ)+\xi\right]\right]\ge1-e^{-\xi}\quad \forall\xi>0,
\label{eq-RRHOTYP}
\end{equation}
%end
where $C$ is a positive constant only depending on $\psi$. 

The bias term in~\eref{eq-RRHOTYP}, namely $h^{2}\!\left(P,\sQ\right)$, accounts for the robustness property of the $\rho$-estimator with respect to the Hellinger distance and measures the additional loss we get as compared to the case when $P$ belongs to $\sQ$. If this quantity is small, the performance of the $\rho$-estimator will not deteriorate too much as compared to the ideal situation where $P$ does belong to $\sQ$. In fact, if there exists some probability $\overline P\in\sQ$ such that $h^{2}\!\left(P,\sQ\right)=h^{2}(P,\overline P)$ is small as compared to $D_{n}(\sQ)/n$, everything is almost as if the $\rho$-estimator $\widehat P$ were built from an i.i.d.\ sample with distribution $\overline{P}$. The $\rho$-estimators under $P$ and $\overline P$ would therefore look the same. This includes the following situations: 

%PARAGRAPH
\paragraph{Misspecification}
The true distribution $P$ of the observations does not belong to $\sQ$ but is close to $\sQ$. For example, let $\sQ$ be countable and $\sbV$-dense in the set of all Gaussian distributions on $\R^{k}$ with identity covariance matrix and mean vector belonging to a linear subspace $\overline S\subset\R^{k}$. Assume that the true distribution $P$ has the same form except for the fact that its mean does not belong to $\overline S$ but is at Euclidean distance $\eps>0$ from $\overline S$. Then, it follows from classical formulas that
\[
h^{2}\!\left(P,\sQ\right)=1-e^{-\eps^{2}/8}\le\eps^{2}/8.
\]

%PARAGRAPH
\paragraph{Contamination}
The true distribution $P$ is of the form $(1-\eps)\overline P+\eps R$ with $\overline P\in\sQ$ and $R\ne\overline P$ but otherwise arbitrary. This situation arises when a proportion $\eps\in (0,1)$ of the sample $X_{1},\ldots,X_{n}$ is contaminated by another sample. If follows from the convexity property of the Hellinger distance that
\[
h^{2}\!\left(P,\sQ\right)\le h^{2}(P,\overline P)\le \eps h^{2}(R,\overline P)\le \eps,
\]
and this bound holds whatever the contaminating distribution $R$. From a more practical point of view, one can see the contaminated case as follows: for each $i$ one decides between no contamination with a probability $1-\varepsilon$ and contamination with a probability $\varepsilon$ and draws $X_{i}$ accordingly with distribution either $\overline{P}$ or $R$. 
If it were possible to extract from the sample $X_{1},\ldots,X_{n}$ these $N$ data, with $N\sim\cB(n,1-\varepsilon)$, which are really distributed according to the distribution $\overline P\in\sQ$, we would build a $\rho$-estimator $\widetilde P$ on these data. The robustness property ensures that the $\rho$-estimator $\widehat P$ based on the whole data set remains close to $\widetilde P$. Everything works almost as if the $\rho$-estimator $\widehat P$ only considered the non-contamined subsample and ignored the other data, at least when $\eps$ is small enough.
%

%SUBSECTION
\subsection{More robustness}\label{Morerob}
There is an additional aspect of robustness that is not apparent in~\eref{eq-RRHOTYP}. 
Our general result about the performance of $\rho$-estimators, as stated in \eref{eq-risk1}, actually allows that our observations be independent but not necessarily i.i.d., in which case the joint distribution $\gP$ of $(X_{1},\ldots,X_{n})$ is actually of the form $\bigotimes_{i=1}^{n} P_{i}$ but not necessarily of the form $P^{\otimes n}$. Of course we do not know whether $\gP$ is of the first form or the second and, proceding as if $X_{1},\ldots,X_{n}$ were i.i.d., we build a $\rho$-estimator $\widehat P\in{\rm Cl}\!\left(\sQ\right)$ of the presumed common density $P$ and make a mistake which is no longer $h^{2}(P,\widehat P)$ but 
\[
{1\over n}\gh^{2}(\gP,\widehat \gP)\quad\mbox{with}\quad\widehat{\gP}=\widehat{P}^{\otimes n}\qquad\mbox{and}\qquad\gh^{2}(\gP,\widehat \gP)=\sum_{i=1}^{n}h^{2}(P_{i},\widehat P),
\]
which is consistent with the i.i.d.\ case $P_{i}=P$ for all $i$. In this context, we actually get the following analogue of~\eref{eq-RRHOTYP}: for all $\xi>0$,
%beg
\begin{equation}
\P\left[{C\over n}\gh^{2}(\gP,\widehat \gP)\le \inf_{Q\in\sQ}\left({1\over n}\sum_{i=1}^{n}h^{2}(P_{i},Q)\right)+{D_{n}(\sQ)+\xi\over n}\right]\ge 1-e^{-\xi}.
\label{eq-rho9}
\end{equation}
%end
This allows many more possibilities of deviations between $\gP$ and the statistical model $\left\{Q^{\otimes n},\;Q\in \sQ\right\}$. For instance, we may have $h(P_{i},\overline{P})\le\varepsilon$ for some $\overline{P}\in \sQ$ and all $i$, $P_{i}\neq P_{i'}$ for all $i\neq i'$, and nevertheless 
\[
\inf_{Q\in\sQ}\left({1\over n}\sum_{i=1}^{n}h^{2}(P_{i},Q)\right)\le \eps^{2}.
\]
An alternative situation corresponds to a small number of ``outliers", namely,  $P_{i}=P$  except on a subset $J\subset\{1,\ldots,n\}$ of indices of small cardinality and, for $i\in J$, $P_{i}$ is completely arbitrary, for instance a Dirac measure. In such a case, for any probability $Q$,
\[
\pa{1-{|J|\over n}}h^{2}(P,Q)\le {1\over n}\sum_{i=1}^{n}h^{2}(P_{i},Q)\le \pa{1-{|J|\over n}}h^{2}(P,Q)+{|J|\over n},
\]
and we deduce from (\ref{eq-rho9}) that, on a set of probability at least $1-e^{-\xi}$, 
\begin{align*}
\frac{C(n-|J|)}{n}h^{2}(P,\widehat P)&\le C{\gh^{2}(\gP,\widehat \gP)\over n}\\
&\le\left[\pa{{n-|J|\over n}}h^{2}\!\left(P,\sQ\right)+{|J|\over n}\right]+{D_{n}(\sQ)+\xi\over n}.  
\end{align*}
Finally, 
\[
\P\cro{Ch^{2}(P,\widehat P)\le h^{2}\!\left(P,\sQ\right)+{|J|+D_{n}(\sQ)+\xi\over n-|J|}}\ge1-e^{-\xi}\quad\mbox{for all }\xi>0.
\]
When $|J|/n$ is small enough, this bound appears to be a slight modification of what we would get from~\eref{eq-RRHOTYP} if $\gP$ were of the form $P^{\otimes n}$. This means that the $\rho$-estimator $\widehat P$ is also robust with respect to a possible departure from the assumption that the $X_{i}$ are i.i.d.

%SECTION
\section{The $\rho$-estimators and the MLE}\label{EX}
As mentioned in the introduction, there are some deep connexions between the MLE and $\rho$-estimators which are mostly due to the similarities in the neighbourhood of 1 between the logarithm and the functions $\psi$ of Proposition~\ref{prop-expsi}. A nice result in this direction was communicated to the authors by Weijie Su on October 2016. It concerns the case of density estimation, as described in Section~\ref{Dens} with a single density model $\overline{\sQ}=\{q\cdot\mu,\,q\in\overline{\cQ}\}$ where $\overline{\cQ}$ is universally separable as defined in Section~\ref{sect-UnivSep}. 
%
% ASSUMPTION
\begin{ass}\label{ass-conv}
The function $x\mapsto\varphi(x)=\psi(\sqrt{x})$, where $\psi$ is the function used to define the statistic $\gT$ in (\ref{def-T}), satisfies $\varphi(1)=0$, is concave and admits a positive derivative at 1.
\end{ass}
%
% PROPOSITION
\begin{prop}[Weijie Su 2016]\label{prop-SU}
Let Assumption~\ref{ass-conv} hold, $\overline \cQ$ be a convex set of densities on the measured space $(\sX,\sB,\mu)$ and the likelihood be not identically equal to 0 on $\overline \cQ$. The maximum likelihood estimator $\widehat Q=\widehat q\cdot \mu$ on the density model $\overline \sQ=\left\{q\cdot \mu,\ q\in\overline\cQ\right\}$, when it exists, satisfies 
\[
\gup(\bsX,\widehat{q})=\sup_{q'\in \overline \cQ}\gT(\bsX,\widehat q,q')=0=\inf_{q\in \overline \cQ}\,\sup_{q'\in \overline \cQ}\gT(\bsX,q,q')=\inf_{q\in \overline \cQ}\gup(\bsX,q)
\]
and is therefore a $\rho$-estimator relative to $(\overline \cQ,0)$. 
\end{prop}
\begin{proof}
Given the data $X_{1},\ldots,X_{n}$, if the maximum likelihood $\widehat q$ exists, it is unique since the logarithm is strictly concave. Moreover $\widehat q(X_{i})>0$ for all $i\in\{1,\ldots,n\}$. Since $\gup(\bsX,\widehat{q})\ge\gT(\bsX,\widehat q,\widehat{q})=0$, it suffices to prove that
\[
L(q)=\gT(\gx,\widehat{q},q)=\sum_{i=1}^{n}\varphi\pa{q(X_{i})\over \widehat q(X_{i})}\le 0\quad \mbox{for all $q\in \overline \cQ$.}
\]
For $q\in\overline \cQ$ and $\eps\in [0,1]$, $(1-\eps)\widehat q+\eps q\in\overline \cQ$ and, when $\varepsilon\rightarrow0$,
%beg
\begin{align}
L\left(\st(1-\eps)\widehat q+\eps q\right)& = n\varphi(1)+\eps\left[\varphi'(1)\sum_{i=1}^{n}{q(X_{i})\over \widehat q(X_{i})}+o(1)\right]\label{eq-dlphi}\\
&=\eps\left[\varphi'(1)\sum_{i=1}^{n}{q(X_{i})\over \widehat q(X_{i})}+o(1)\right]
\nonumber 
\end{align}
%end
since $\varphi(1)=0$. When $\varphi$ is the logarithm and $\varepsilon>0$, the right-hand side of (\ref{eq-dlphi}) is negative since $\widehat q$ is the unique MLE. Letting $\eps$ go to 0 we derive that
%beg
\begin{equation}
\sum_{i=1}^{n}{q(X_{i})\over \widehat q(X_{i})}\le 0\quad \mbox{for all $q\in\overline \cQ$.}
\label{eq-Weijie}
\end{equation}
%end
Moreover, the concavity of $\varphi$ implies that for all $\eps\in[0,1]$
\[
\varphi\pa{(1-\eps)\widehat q(X_{i})+\eps q(X_{i})\over \widehat q(X_{i})}\le\varphi(1)+\varepsilon
{q(X_{i})\over \widehat q(X_{i})}\varphi'(1)=\varepsilon\varphi'(1){q(X_{i})\over \widehat q(X_{i})}
\]
so that, for all $q\in\overline \cQ$, $L\left(\st(1-\eps)\widehat q+\eps q\right)\le\eps\varphi'(1)\sum_{i=1}^{n}{q(X_{i})/\widehat q(X_{i})}$ and 
\[
\gT(\gx,\widehat{q},q)=L(q)\le \varphi'(1)\sum_{i=1}^{n}{q(X_{i})\over \widehat q(X_{i})}\le 0
\]
by (\ref{eq-Weijie}), which concludes the proof.
\end{proof}
Note that both functions $\psi_{1}$ and $\psi_{2}$ of Proposition~\ref{prop-expsi} satisfy Assumption~\ref{ass-conv}. 

We may now derive the following relationship between the MLE and $\rho$-estimators, the proof of which immediately follows from Theorem~\ref{thm-Exten} and Su's Proposition.
%
% COROLLARY
\begin{cor}\label{cor-rho/mle}
Let $\overline \cQ$ be a convex set of densities on the measured space $(\sX,\sB,\mu)$ which is universally separable on $\sX$ with countable and $\sbT$-dense subset $\cQ$ and $\psi$ satisfy Assumptions~\ref{ass-psi} and \ref{ass-conv}. The maximum likelihood estimator $\widehat Q=\widehat q\cdot \mu$ on the density model $\overline \sQ=\{q\cdot \mu,\ q\in\overline\cQ\}$, when it exists, is a $\rho$-estimator on the $\rho$-density model $\sQ=\{q\cdot \mu,\ q\in\cQ\}$.
\end{cor}
For illustration, the set $\overline \cQ=\overline \cH_{\cI}$ of right-continuous histograms based on a fixed partition $\cI$ of $[0,1)$ into $D\ge 1$ intervals is convex and obviously universally separable. The usual histogram $\widehat p$ based on $\cI$, which corresponds to the MLE on $\overline \cH_{\cI}$, can be viewed as a $\rho$-estimator on a countable subset of $\overline \cH_{\cI}$. Taking back some of the examples of convex and universally separable density sets given in Section~\ref{sect-UnivSep}, we deduce that the MLE on $\overline \cH_{\downarrow}$, i.e.\ the Grenander estimator, or on the set $\overline \cH_{\alpha}(L)$ are also $\rho$-estimators.

%SECTION
\section{Bounding the $\rho$-dimension function of a $\rho$-model with applications to the risk of $\rho$-estimators}\label{sect-DM}
It clearly follows from the results of Section~\ref{MRI} that bounding the risk of $\rho$-estimators amounts to bounding the $\rho$-dimension of $\rho$-models which we shall now do under various assumptions. Throughout this section we fix the function $\psi$ satisfying Assumption~\ref{H-debase} (typically $\psi_{1}$ or $\psi_{2}$) and when we shall say that some quantity depends on $\psi$, this will mean that it actually depends on $a_{1}$ and $a_{2}$.

In view of~\eref{eq-risk1b}, of special interest is the situation where the $\rho$-dimension function $(\gP,\overline \gP)\mapsto D^{\sbQ}(\gP,\overline\gP)$ of the $\rho$-model $\sbQ$ is uniformly bounded from above on $\sbP\times  \sbQ$ by some constant $D_{n}\ge 1$. Let us begin by a few elementary considerations. If one can find a representation $\cR=(\gmu,{\cbQ})$ of $\sbQ\cup\{\overline \gP\}$ such that $w\left(\cR,\sbQ,\gP,\overline \gP,y\right)\le a_{1}y^{2}/8$ for all $y\ge\beta^{-1}\sqrt{D}$, we immediately derive from the definition of $D^{\sbQ}$ that
%beg
\begin{equation}
D^{\sbQ}(\gP,\overline \gP)\le D\vee 1\quad \mbox{for }(\gP,\overline \gP)\in \sbP^{2}.
\label{eq-bound-D}
\end{equation}
%end
In particular, since $|\psi|\le1$, the expectation in (\ref{eq-w1}) is never larger than $2n$ so that $w\left(\cR,\sbQ,\gP,\overline \gP,y\right)\le a_{1}y^{2}/8$ for $y\ge4\sqrt{(n/a_{1})}$ and (\ref{eq-bound-D}) always holds with 
\[
\sqrt{D}=4\beta\sqrt{(n/a_{1})}=\sqrt{na_{1}}/a_{2}\quad\mbox{ or equivalently }\quad
D=na_{1}/a_{2}^{2}\le n/6.
\]
Finally, whatever the choices of $\sbQ$ and $\psi$,
%beg
\begin{equation}
D^{\sbQ}(\gP,\overline \gP)\le n/6\quad\mbox{for all }(\gP,\overline\gP)\in\sbP^{2},
\label{eq-RB}
\end{equation}
%end

More precise bounds will now be given that depend on some specific features of $\sbQ$. 
\subsection{The finite case}\label{DM1}
Given a finite subset $\sbQ\subset \sbP$, let us set
%beg
\begin{equation}
\sH\left(\sbQ,y\right)=\sup_{\gP\in\sbP}\log_{+}\!\pa{2\left|\sbQ\cap \sbB(\gP,y)\strut\right|}\ \ \mbox{for all}\ y> 0
\label{eq-reseau}
\end{equation}
and for $x_{0}=\sqrt{2}\cro{{\sqrt{1+(\beta/a_{2})}+1}}$
%beg
\begin{equation}\label{def-eta}
\overline \eta=\sup\ac{z>0, \,\sqrt{\sbH\left(\sbQ,z/\beta\right)}>z/x_{0}}.
\end{equation}
Since $\sbQ$ is finite, the function $y\mapsto\sH\left(\sbQ,y\right)$ is bounded by $\log\pa{2\left|\sbQ\right|}$ and since $\beta/a_{2}=a_{1}/(4a_{2}^{2})\le 1/24$,  
%beg
\begin{equation}
\overline \eta\le x_{0}\sqrt{\log\pa{2\left|\sbQ\right|}}<3\sqrt{\log\pa{2\left|\sbQ\right|}}.
\label{Eq-eta}
\end{equation}
%end
%
% PROPOSITION
\begin{prop}\label{cas-fini}
If  $\sbQ$ is a finite subset of $\sbP$ and $\overline \eta$ is defined by~\eref{def-eta}, 
\[
D^{\sbQ}(\gP,\overline \gP)\le D_{n}\left(\sbQ\right)=\overline \eta^{2}\vee 1<9\log\pa{2\left|\sbQ\right|}\quad\mbox{for all }(\gP,\overline\gP)\in\sbP^{2}.
\]
\end{prop}
The proof of this result is given in Section~\ref{sect-pcas-fini} of the Appendix. The first upper bound $\overline \eta^{2}\vee 1$ for $D^{\sbQ}(\gP,\overline \gP)$ neither depends on $\gP$ nor on $\overline \gP$ but might depend on $\beta$. The second bound only depends on the cardinality of $\sbQ$ and therefore holds whatever $\psi$. 

If a model $\overline{\sbQ}$ is a totally bounded subset of the metric space $\left(\sbP,\gh\right)$ and $\eta>0$, one can cover $\overline{\sbQ}$ by a finite number of closed balls of radius $\eta$ and the set $\sbQ[\eta]$ of their centers is an $\eta$-net for $\overline{\sbQ}$ (see Definition~\ref{def-net}), which means that $\gh(\gQ,\sbQ[\eta])\le \eta$ for all $\gQ\in \overline{\sbQ}$. The function $y\mapsto\sH(\sbQ[\eta],y)$ measures in a sense the massiveness of $\sbQ[\eta]$ and turns out to be a useful tool to measure that of $\overline{\sbQ}$. We shall in particular use the following classical notions of dimension based on the metric structure of $\overline{\sbQ}$.
%
% DEFINITION
\begin{df}\label{def-metdim}
A  set $\overline{\sbQ}\subset \sbP$ is said to have a metric dimension bounded by $\widetilde{D}$, where $\widetilde D$ is a right-continuous function  from $(0,+\infty)$ to $[1/2,+\infty]$, if, for any positive $\eta$, there exists an $\eta$-net $\sbQ[\eta]$ for $\overline{\sbQ}$ which satisfies
%beg
\begin{equation}\label{def-birge}
\sbH(\sbQ[\eta],y)\le(y/\eta)^{2}\widetilde D_{n}(\eta)\quad\mbox{for all } y\ge2\eta.
\end{equation}
We shall say that $\overline{\sbQ}$ has an entropy dimension bounded by $V\ge0$ if, for any $\eta>0$, there exists an $\eta$-net $\sbQ[\eta]$ of $\overline{\sbQ}$ such that
\begin{equation}\label{def-entropy}
\sbH(\sbQ[\eta],y)\le V\log\pa{y/\eta}\quad\mbox{for all }\ \ y\ge2\eta.
\end{equation} 
\end{df}
For the sake of convenience, we have slightly modified the original definition of the metric dimension due to Birg\'e~\citeyearpar{MR2219712} (Definition~6 p.~293) which is actually obtained by replacing the left-hand side of~\eref{def-birge} by~$\sbH(\sbQ[\eta],y)-\log 2$.  Since in both definitions the metric dimension is not smaller than $1/2$, it is easy to check that, if $\overline \sbQ$ has a metric dimension bounded by $D_{M}$ in Birg\'e's sense, it has a metric dimension bounded by $\widetilde D=(1+(\log 2)/2)D_{M}$ in our sense and, conversely, if $\overline \sbQ$ has a metric dimension bounded by $\widetilde D$ in our sense, it also has a metric dimension bounded by $\widetilde D$ in Birg\'e's sense. Hence, changing $\widetilde D$ into $D_{M}$ only changes the numerical constants.

The logarithm being a slowly varying function, 
it is not difficult to see that the notion of metric dimension is more general than the 
entropy one in the sense that if $\overline{\sbQ}$ has an entropy dimension bounded 
by $V$, then it also has a metric dimension bounded by $\widetilde D_{n}(\cdot)$ with
\begin{equation}\label{comp}
\widetilde D_{n}(\eta)\le(1/2)\vee[V(\log2)/4]\quad\mbox{for all }\eta>0.
\end{equation}

If $\overline \sbQ$ has a metric dimension bounded by $\widetilde D$ and if $\eta$ is a positive number satisfying 
\begin{equation}\label{eq-Deta}
\widetilde D_{n}(\eta)\le(\beta\eta/x_{0})^{2},
\end{equation}
with $x_{0}$ given by~\eref{def-eta}, we deduce from~\eref{def-birge} that there exists an $\eta$-net $\sbQ[\eta]$ for $\overline \sbQ$ for which
\[
\sqrt{\sbH(\sbQ[\eta],z/\beta)}\le z/x_{0}\quad\mbox{for all }z\ge2\beta\eta.
\]
It then follows that $\overline\eta$, as defined in (\ref{def-eta}), satisfies $\overline\eta\le 2\beta\eta$ and we deduce from Proposition~\ref{cas-fini} that the $\rho$-dimension function $D^{\sbQ}$ of $\sbQ=\sbQ[\eta]\subset \overline \sbQ$ satisfies
\begin{equation}\label{eq-BDM}
D^{\sbQ}(\gP,\overline \gP)\le D_{n}(\sbQ)=(2\beta\eta)^{2}\vee 1\quad\mbox{for all } (\gP,\overline \gP)\in\sbP^{2}.
\end{equation}
If, in particular, $\overline{\sbQ}$ has an entropy dimension bounded by $V\ge0$ we deduce from~\eref{comp} that~\eref{eq-Deta} holds for 
\begin{equation}\label{eq-defetaED}
\eta^{2}={x_{0}^{2}\over 2\beta^{2}}\pa{1\bigvee {V\log 2\over 2}}<{9\over 2\beta^{2}}\pa{1\bigvee {V\log 2\over 2}}
\end{equation}
and we derive from~\eref{eq-BDM} that 
%beg
\begin{equation}
D^{\sbQ}(\gP,\overline \gP)\le D_{n}(\sbQ)=18\pa{1\bigvee {V\log 2\over 2}}\quad\mbox{for all }(\gP,\overline \gP)\in\sbP^{2}.
\label{eq-Bentro}
\end{equation}
%end
Since in both cases $\gh(\gP,\sbQ)\le \gh(\gP,\overline \sbQ)+\eta$ for all $\gP\in\sbP$ because $\sbQ$ is an $\eta$-net for $\overline{\sbQ}$, we obtain from (\ref{eq-BDM}), (\ref{eq-Bentro}) and \eref{eq-FRM-GM0} the following result.
%
% COROLLARY
\begin{cor}\label{cor-casFD}
Let $\psi$ be a function satisfying Assumption~\ref{H-debase}.
\begin{listi}
\item If $\overline \sbQ$ has a metric dimension bounded by $\widetilde{D}$ and $\eta$ satisfies~\eref{eq-Deta}, any $\rho$-estimator $\widehat \gP$ based on a suitable $\eta$-net $\sbQ$ for $\overline \sbQ$ satisfies for all $\gP\in\sbP$ and $\xi>0$
%beg
\begin{equation}
\P\left[C\gh^{2}(\gP,\widehat \gP)\le \gh^{2}(\gP,\overline \sbQ)+\left(\eta^{2}\vee1\right)+\xi\right]\ge1-e^{-\xi}.
\label{bornerisk-DM}
\end{equation}
%end
\item If $\overline \sbQ$ has an entropy dimension bounded by $V$ and $\eta$ satisfies~\eref{eq-defetaED}, any $\rho$-estimator $\widehat \gP$ based on a suitable $\eta$-net $\sbQ$ for $\overline \sbQ$ satisfies sfor all $\gP\in\sbP$ and $\xi>0$
%beg
\begin{equation}
\P\left[C\gh^{2}(\gP,\widehat \gP)\le \gh^{2}\left(\gP,\overline{\sbQ}\right)+(V\vee1)+\xi\right]\ge1-e^{-\xi}.
\label{eq-bornecasFD}
\end{equation}
%end
\end{listi}
In both cases, $C$ is a constant depending only on the choice of $\psi$.
\end{cor}
%

%SUBSECTION
\subsection{Bounds based on the VC-index}\label{DM2}
In this section we investigate the case of a model $\overline \sbQ$ given by a specific  representation $(\gmu, \overline \cbQ)$ where the density set $\overline \cbQ$ is possibly uncountable but satisfies some property to be described below. 

A density $\gq=(q_{1},\ldots,q_{n})\in\cbL(\gmu)$ can be viewed as a function on $\overline \sX=\bigcup_{i=1}^{n}\left(\{i\}\times\sX_{i}\right)$ defined, for $\overline x=(i,x)$ with $x\in\sX_{i}$, by $\gq(i,x)=q_{i}(x)$ so that a subset $\overline{\cbQ}\subset \cbL(\gmu)$ is now viewed as a class of real-valued functions on $\overline \sX$.  A common notion of dimension for the class $\overline{\cbQ}$ is the following one. 
%
% DEFINITION
\begin{df}
A class $\sF$ of functions from a set $\sX$ with values in $(-\infty,+\infty]$ is 
VC-subgraph with index $\overline V$ (or equivalently with dimension $\overline V-1\ge 0$) if the class of subgraphs $\{(x,u)\in\sX\!\times\R\,|\,f(x)>u\}$ as $f$ varies in $\sF$ is a VC-class of sets in
$\sX\!\times\R$ with index $\overline V$ (or dimension $\overline V-1$).
\end{df}
We recall that, by definition, the index $\overline{V}$ of a VC-class is a positive integer, hence its dimension $\overline{V}-1\in\N$. For additional information about VC-classes and related notions, we refer to van der Vaart and Wellner~\citeyearpar{MR1385671} and Baraud {\em et al.}~\citeyearpar[Section~8]{MR3595933}. 

%
% PROPOSITION
\begin{prop}\label{cas-VC}
Let $\psi$ satisfy Assumption~\ref{H-debase} and $\overline{\cbQ}\subset \cbL(\gmu)$ be a VC-subgraph class of densities on $\overline \sX$ with index not larger than $\overline V$. For any $\rho$-model $\sbQ\subset \overline \sbQ=\{\gq\cdot \gmu,\; \gq\in\overline \cbQ\}$, for all $(\gP,\overline\gP)\in\sbP\times \sbP^{\gmu}$
\[
D^{\sbQ}(\gP,\overline \gP)\le D_{n}\left(\overline{\sbQ}\right)=C_{1}\left(\overline V \wedge n\right)\cro{1+\log_{+}\pa{n/\overline V}},
\]
where $C_{1}$ is a universal constant.
\end{prop}
The proof is given in Section~\ref{sect-pcas-VC} of the Appendix. A nice feature of this bound lies in the fact that it neither depends on the choices of $\psi$ nor on the cardinality of $\sbQ$ which can therefore be arbitrarily large. In particular, when $\sbQ$ is $\sbV$-dense in $\overline \sbQ$ we deduce the following result from Proposition~\ref{cas-VC},  \eref{eq-FRM-GM0} (with $\eta=0$) and Theorem~\ref{thm-Exten}.
%
% COROLLARY
\begin{cor}\label{cor-VC}
Let $\psi$ be a function satisfying Assumption~\ref{H-debase} and $\overline{\cbQ}\subset \cbL(\gmu)$ a VC-subgraph class of densities on $\overline \sX$ with index $\overline V$. Any $\rho$-estimator $\widehat \gP$ built on a countable and $\sbV$-dense subset $\sbQ$ of $\overline \sbQ=\left\{\gq\cdot \gmu,\; \gq\in\overline \cbQ\right\}$ satisfies, for all $\gP\in\sbP$ and $\xi>0$,
%beg
\begin{equation}
\:\P\left[C\gh^{2}(\gP,\widehat \gP)\le \gh^{2}\left(\gP,\overline \sbQ\right)+\left(\overline V \wedge n\right)\cro{1+\log_{+}\!\pa{{n/\overline V }}}+\xi\right]
\ge1-e^{-\xi}
\label{eq-bornecasVC}
\end{equation}
%end
where the constant $C$ only depends on the choice of $\psi$. If, moreover, $\overline \cbQ$ is universally separable, then \eref{eq-bornecasVC} holds for any $\rho$-estimator relative to $\left((\gmu, \overline \cbQ), \bf{0}\right)$.
\end{cor}
In the particular case of density estimation the following result is useful in view of applying Proposition~\ref{cas-VC}.
%
% PROPOSITION
\begin{prop}\label{cas-VCd}
If $\overline{\cQ}\subset \cL(\mu)$ is VC-subgraph on $\sX$ with index $\overline V$, then the set $\overline{\cbQ}=\left\{\gq=(q,\ldots,q),\ q\in\overline{\cQ}\right\}\subset \cbL(\gmu)$ is VC-subgraph on $\overline \sX$ with index not larger than $\overline V$.
\end{prop}
% PROOF
\begin{proof}
If the class of subgraphs $\left\{\left.\!(\overline x,u)\in\overline \sX\!\times\R \,\right|\gq(\overline x)>u\right\}$, with $\gq$ running in $\overline{\cbQ}$, shatters the subset $\left\{(\overline x_{1},u_{1}),\ldots,(\overline x_{k},u_{k})\right\}$ of $\overline \sX\times \R$, then, whatever $J\subset \{1,\ldots,k\}$, there exists $\gq\in\overline{\cbQ}$ such that $j\in J$ is equivalent to $\gq(\overline x_{j})=q(x_{j})>u_{j}$.
Hence, the class of subgraphs $\{(x,u)\in\sX\!\times\R\,|\,q(x)>u\}$ with $q$ running in $\overline{\cQ}$ shatters the subset $\left\{(x_{1},u_{1}),\ldots,(x_{k},u_{k})\right\}$ of $\sX\times \R$ and therefore $k+1\le \overline V$.
\end{proof}
%
% SUBSECTION
\subsection{More bounds}\label{DM3}
As we have observed in the previous sections --- see (\ref{bornerisk-DM}), (\ref{eq-bornecasFD}) and (\ref{eq-bornecasVC}) ---, there are various situations for which, given a model $\overline \sbQ\subset \sbP$, it is possible to build a $\rho$-estimator $\widehat \gP$ with values in $\overline \sbQ$ satisfying for all $\gP\in\sbP$ and $\xi>0$,
\[
\P\left[C\gh^{2}(\gP,\widehat \gP)\le \gh^{2}\left(\gP,\overline{\sbQ}\right)+D_{n}(\overline \sbQ)+\xi\right]\ge1-e^{-\xi}
\]
for some quantity $D_{n}\left(\overline{\sbQ}\right)\ge 1$ only depending on the specific features of $\overline \sbQ$ and some constant $C>0$ depending on the choice of $\psi$. Such an inequality leads to a risk bound of the following form (with $C'>0$):
\begin{equation}\label{eq-borne-risk}
\E\cro{\gh^{2}(\gP,\widehat \gP)}\le C'\left[\gh^{2}\left(\gP,\overline{\sbQ}\right)+D_{n}(\overline \sbQ)\right]\quad \mbox{for all }\gP\in\sP
\end{equation}
and allows us to bound from above the minimax risk over $\overline \sbQ$ by $C'D_{n}\left(\overline \sbQ\right)$.

However, not all statistical models admit a finite minimax risk and for such models there is consequently no hope to bound from above the $\rho$-dimension function uniformly as we did in the previous sections. One such example is the set of probabilities on $(0,+\infty)$ with non-increasing densities with respect to the Lebesgue measure. More examples can also be found in Baraud and Birg\'e~\citeyearpar{MR3565484}. For some of these models it is possible to build a $\rho$-estimator the risk of which does not degenerate, a typical example being the Grenander estimator which is, as already seen, a $\rho$-estimator.

Following Baraud~\citeyearpar{Bar2016}, we introduce this definition:
%
% DEFINITION
\begin{df}\label{weakVC}
A class of functions $\sF$ defined on a set $\sX$ and with values in $[-\infty,+\infty]$ is said to be weak VC-major with dimension not larger than $k\in\N$ if, for all $u\in\R$, the class of subsets
\[
\sC_{u}(\sF)=\ac{\st\{x\in\sX\,|\,f(x)>u\},\ f\in\sF}
\]
is a VC-class with dimension not larger than $k$ (index not larger than $k+1$). The weak dimension of $\sF$ is the smallest of such integers $k$.
\end{df}
%
% DEFINITION
\begin{df}
Let $\sF$ be a class of real-valued functions on $\sX$. We shall say that an element  $\overline f\in \sF$ is extremal in $\sF$ with degree $d\in\N$ if the class of functions 
\[
\left(\sF/\overline f\right)=\left\{f/\overline f,\ f\in\sF\right\}
\]
is weak VC-major with dimension $d$.
\end{df}

For $\gmu\in\sbM$, we consider a density set $\overline \cbQ\subset \cbL(\gmu)$ which is viewed as a class of real-valued functions on $\overline \sX=\bigcup_{i=1}^{n}\left(\{i\}\times\sX_{i}\right)$ as we did in Section~\ref{DM2}. The corresponding model $\overline \sbQ$ is $\left\{\gq\cdot \gmu,\; \gq\in \overline \cbQ\right\}$ and, for $d\in\{1,\ldots,n\}$, we denote by $\overline \cbQ_{d}$ the subset of $\overline \cbQ$ of those densities $\gq$ which are extremal in $\overline \cbQ$ with degree $d$. We set $\overline \sbQ_{d}=\left\{\gq\cdot \gmu,\; \gq\in \overline \cbQ_{d}\right\}$ and let $\cD$ be the subset of $\{1,\ldots,n\}$ consisting of  those $d$ such that $\overline \sbQ_{d}\neq \varnothing$. 
%
% PROPOSITION
\begin{prop}\label{cas-extreme}
Let $\psi$ satisfy Assumption~\ref{H-debase}. For all $\rho$-models $\sbQ\subset \overline \sbQ$ and all $d\in \cD$,  
\[
D^{{\sbQ}}(\gP,\overline \gP)\le 33d\left[\log\!\left(e^{2}n/d\right)\right]^{3}\quad\mbox{for all }(\gP,\overline \gP)\in\sbP\times \overline \sbQ_{d}.
\]
\end{prop}
The proof is given in Section~\ref{sect-pcas-extreme} of the Appendix.
This upper bound, although depending on the specific features of $\overline \sbQ$, is free from the choices of $\psi$. We immediately derive from Proposition~\ref{cas-extreme} and  Theorem~\ref{thm-main1} with a suitable choice of $\overline \gP$ the following result.
%
% COROLLARY
\begin{cor}\label{cor-Extrem}
Let $\psi$ satisfy Assumption~\ref{H-debase} and assume that $\cD\ne \varnothing$ and that $\overline \sbQ_{d}$ contains a countable and $\sbV$-dense subset $\sbQ_{d}$ for all $d\in\cD$. Any $\rho$-estimator $\widehat \gP$ on a $\rho$-model $\sbQ\subset \overline \sbQ$ containing $\bigcup_{d\in\cD}\sbQ_{d}$ satisfies, for all $\gP\in\sbP$ and $\xi>0$,
\begin{equation}\label{eq-cor-Extrem}
\P\cro{C\gh^{2}(\gP,\widehat \gP)\le \inf_{d\in\cD}\cro{\gh^{2}(\gP,\overline \sbQ_{d})+d\left[\log\!\left(e^{2}n/d\right)\right]^{3}}+\xi}\ge 1-e^{-\xi},
\end{equation}
for some constant $C$ depending on $\psi$ only. If moreover, $\overline \cbQ$ is universally separable, any $\rho$-estimator relative to $\left((\gmu, \overline \cbQ),\bf{0}\right)$ also satisfies \eref{eq-cor-Extrem}.
\end{cor}
%
% PROOF
\begin{proof}
Proposition~\ref{cas-extreme} and  Theorem~\ref{thm-main1} lead to \eref{eq-cor-Extrem}. When $\overline \cbQ$ is universally separable there exists a countable and $\sbT$-dense subset $\cbQ'\subset \overline \cbQ$. The countable set $\cbQ=\cbQ'\bigcup\left(\bigcup_{d\in\cD}\cbQ_{d}\right)$ is still countable and $\sbT$-dense in $\overline \cbQ$ and the corresponding $\rho$-model $\sbQ$ is $\sbV$-dense in $\overline \sbQ$ and contains $\bigcup_{d\in\cD}\sbQ_{d}$. By Theorem~\ref{thm-Exten} any $\rho$-estimator relative to $\left((\gmu, \overline \cbQ),\bf{0}\right)$ is also a $\rho$-estimator relative to $((\gmu,\sbQ),\bf{0})$ and therefore satisfies \eref{eq-cor-Extrem}.
\end{proof}
Note that the bound depends on the initial representation $\left(\gmu, \overline \cbQ\right)$ because the sets $\overline{\cbQ}_{d}$, hence the sets $\overline \sbQ_{d}$, depend on $\overline \cbQ$. This result looks like a model selection result among the sets $\left\{\overline \sbQ_{d},\; d\in \cD\right\}$ although the $\rho$-estimator $\widehat \gP$ is built on a single $\rho$-model $\sbQ\subset\overline \sbQ$ with a nul penalty function. It implies that the minimax risk over each set $\overline \sbQ_{d}$ is necessarily finite, while that on $\overline \sbQ$ might not be. 

In the particular case of density estimation, the following result turns to be useful in view of applying Proposition~\ref{cas-extreme}.
%
% PROPOSITION
\begin{prop}
Let $\overline \cQ$ be a subset of $\cL(\mu)$ viewed as a class of functions on $\sX$. If $\overline p$ is extremal in $\overline{\cQ}$ with degree $d$, $\overline \gp=(\overline p,\ldots,\overline p)$ is extremal in $\overline{\cbQ}=\left\{\gq=(q,\ldots,q),\ q\in\overline{\cQ}\right\}$, viewed as a class of functions on $\overline \sX$, with degree not larger than $d$.
\end{prop}
%
% PROOF
\begin{proof}
Let $u\in\R$. If $\sC_{u}\!\left(\st(\overline{\cbQ}/\overline \gp)\right)$ shatters $\left\{\overline x_{1},\ldots,\overline x_{k}\right\}\subset\overline \sX$, for all $J\subset \{1,\ldots,k\}$ there exists $\gq\in \overline{\cbQ}$ such that $j\in J$ if and only if $\left(\gq/\overline \gp\right)(\overline x_{j})=\left(q/\overline p\right)(x_{j})>u$.
Hence $\sC_{u}\!\left(\st(\overline{\cQ}/\overline p)\right)$ shatters $\left\{x_{1},\ldots, x_{k}\right\}\subset\sX$ which is only possible for $k\le d$.
\end{proof}
%
%SUBSECTION
\subsection{Some examples of statistical models}\label{MRII}
Let us restrict ourselves here to the density framework where $X_{1},\ldots,X_{n}$ are assumed to be i.i.d.\ with common distribution $P$ on $\sX$ and we have at hand a set of candidate probabilities $Q=q\cdot\mu$ with $q\in\overline{\cQ}\subset\cL(\mu)$ for $P$. We shall provide here some examples of density sets $\overline \cQ$ to which Proposition~\ref{cas-VC} or~\ref{cas-extreme} applies. 

\paragraph{\bf Piecewise constant functions}
Let $k$ be a positive integer and $\sX$ an arbitrary interval of $\R$ (possibly $\sX=\R$). We define $\sF_{k}$ as the class of functions $f$ on $\sX$ such that there exists a partition $\sI(f)$ of $\sX$ into at most $k$ intervals (of positive lengths) with $f$ constant on each of these intervals. Note that $\sI(f)$ depends on $f$. The following result is to be proved in Section~\ref{pr9} of the Appendix.
%
% PROPOSITION
\begin{prop}\label{prop-CPM}
The set $\sF_{k}$ is VC-subgraph with dimension bounded by $2k$.
\end{prop}
Let us apply this to histogram estimation on $\sX=\R$. For a positive integer $D$ we denote by $\overline \cQ_{D}$ the subset of $\sF_{D+2}$ of right-continuous densities with respect to the Lebesgue measure $\mu$, that is the set of right-continuous and piecewise constant densities on $\R$ with at most $D$ pieces and by $\overline \sQ_{D}=\left\{q\cdot \mu,\, q\in \overline \cQ_{D}\right\}$ the corresponding model for $P$.
We derive from Propositions~\ref{cas-VC} and~\ref{prop-CPM} that, for some universal constant $C>0$ and all $\rho$-models $\sbQ\subset \overline \sbQ_{D}$,
\[
D^{{\sbQ}}(\gP,\overline \gP)\le C(D\wedge n)\cro{1+\log_{+}(n/D)}\quad\mbox{for all }\left(\gP,\overline \gP\right)\in \sbP\times \sbP^{\gmu}.
\]
Hence, by Corollary~\ref{cor-VC}, for all $\rho$-estimators $\widehat \gP$ on some countable and $\sbV$-dense subset $\sbQ_{D}$ of $\overline \sbQ_{D}$,
%beg
\begin{equation}
C\E\cro{\gh^{2}(\gP,\widehat \gP)}\le \gh^{2}(\gP,\overline{\sbQ}_{D})+(D\wedge n)\cro{1+\log_{+}(n/D)}.
\label{eq-RiskHisto}
\end{equation}
%end
Since $\overline \cQ_{D}$ is universally separable (see Section~\ref{Sect-UnivSep} of the Appendix), we deduce from Theorem~\ref{thm-Exten} that \eref{eq-RiskHisto} remains true for any $\rho$-estimator $\widehat \gP$ on the non-countable model $\overline \sbQ_{D}$ relative to the representation $(\gmu, \overline\cbQ_{D})$ (with a null penalty function).

The logarithmic factor in this bound turns out to be necessary. The argument is as follows. When $\gP\in\overline \sbQ_{D}$, it follows from (\ref{eq-RiskHisto}) that $\E\cro{\gh^{2}(\gP,\widehat \gP)}\le C'(D\wedge n)\cro{1+\log_{+}(n/D)}$
for some universal constant $C'>0$. This inequality appears to be optimal (up to the numerical constant $C'$) in view of the lower bound established in Proposition~2 of Birg\'e and Massart~\citeyearpar{MR1653272}. This also shows that the logarithmic factor involved in the bound of the $\rho$-dimension function established in Proposition~\ref{cas-VC} is necessary, at least for some VC-subgraph classes.

\paragraph{\bf Piecewise exponential families}
Using similar arguments based on Cor\-ollaries~\ref{cor-VC} and~\ref{cor-Extrem} as we did above, we may establish risk bounds of the same flavour as \eref{eq-RiskHisto} with the following density sets.

\begin{df}\label{def-pef}
Let $g_{1},\ldots,g_{J}$ be $J\ge 1$ real-valued functions on a set $\sX$.
We shall say that a class $\sF$ of positive functions on $\sX$ is an exponential family based on $g_{1},\ldots,g_{J}$ if the elements $f$ of $\sF$ are of the form 
\begin{equation}\label{formexp}
f=\exp\cro{\sum_{j=1}^{J}\beta_{j}g_{j}}\ \ \mbox{for}\ \ \beta_{1},\ldots,\beta_{J}\in\R.
\end{equation}
If $\sX$ is a nontrivial interval of $\R$ and $k$ a positive integer, we shall say that $\sF$ is a $k$-piecewise exponential family based on $g_{1},\ldots,g_{J}$ if for any $f\in \sF$ there exists a partition $\sI(f)$ of $\sX$ into at most $k$ intervals such that for all $I\in \sI(f)$, the restriction $f_{I}$ of $f$ to $I$ is of the form~\eref{formexp} with coefficients $\beta_{j}$ depending on $I$.
\end{df}
The properties of exponential and piecewise exponential families are described by the following proposition to be proven in Section~\ref{sect-ppp-dimexp} of the Appendix.
%
% PROPOSITION
\begin{prop}\label{pp-dimexp}
Let $\overline{\cQ}$ be a class of functions on $\sX$.
%
%\begin{enumerate}[label=\roman*),ref={\it \roman*},leftmargin=7mm]
%
\begin{listi}
\item\label{pp-eq1} If $\overline{\cQ}$ is an exponential family based on $J\ge 1$ functions, $\overline{\cQ}$ is VC-subgraph with index not larger than $J+2$. 
\item\label{pp-eq1b} Let $\sI$ be a partition of $\sX$ with cardinality not larger than $k\ge 1$. If for all $I\in\sI$ the family $\overline{\cQ}_{I}$  consisting of the restrictions of the functions $q$ in $\overline{\cQ}$ to the set $I$ is an exponential family on $I$ based on $J\ge 1$ functions, $\overline{\cQ}$ is VC-subgraph with index not larger than $k(J+2)$.
\item\label{pp-eq2} If $\sX$ is a non-trivial interval of \,$\R$ and $\overline{\cQ}$ is a $k$-piecewise exponential family based on $J$ functions, all densities $\overline p\in\overline{\cQ}$ are extremal in $\overline{\cQ}$ with degree $d$ not larger than $\lceil9.4k(J+2)\rceil=\inf\{j\in\N,\,j\ge9.4k(J+2)\}$.
%\end{enumerate}
\end{listi}
\end{prop}
%

%SECTION
\section{Estimating a conditional distribution}\label{CD}
%SUBSECTION
\subsection{Description of the framework}\label{CD1}
Let us now apply our result to the estimation of a conditional distribution. We consider i.i.d.\ pairs $X_{i}=(W_{i},Y_{i})$, $i=1,\ldots,n$ of random variables with values in the product space $(\sW\times \sY,\sB(\sW)\otimes\sB(\sY))$ and common distribution $P$, assuming that truely $\gP=P^{\otimes n}$. We denote by $P_{W}$ the marginal distribution of $W$ and assume the existence of a conditional distribution $P_{w}$ of $Y$ when $W=w$, which means that for all bounded measurable functions $f$ on $\sY$,
\[
\E[f(Y)\,|\,W=w]=\int_{\sY} f(y)\,dP_{w}(y)\quad P_{W}\mbox{-a.s.}
\]
and for all bounded measurable functions $g$ on $\sW\times \sY$,
\[
\E[g(W,Y)]=\int_{\sW} \cro{\int_{\sY} g(w,y)\,dP_{w}(y)}\,dP_{W}(w).
\]
Our purpose is to estimate the conditional distribution $P_{w}$ without the knowledge of $P_{W}$ which may therefore be completely arbitrary. 
To do so, we consider a reference measure $\lambda$ on $(\sY,\sB(\sY))$ and the set $\cL_{c}(\sW,\lambda)$ of conditional densities with respect to $\lambda$, that is the set of measurable functions $t$  from $(\sW\times\sY, \sB(\sW)\otimes\sB(\sY))$ to $\R_{+}$ such that for all $w\in \sW$, the function $t_{w}:y\mapsto t(w,y)\in\cL(\lambda)$. Then, to each element $t\in\cL_{c}(\sW,\lambda)$ is associated a conditional distribution $t_{w}\cdot \lambda$ for $Y$. In order to build our estimators we first introduce a countable family $\{S_{m},\ m\in\cM\}$ of countable subsets $S_{m}$ of $\cL_{c}(\sW,\lambda)$, and a non-negative weight function $\Delta$ on $\cM$ satisfying (\ref{eq-delta}).
To each $S_{m}$, we associate the $\rho$-density model ${\sQ}_{m}=\{Q_{t},\,t\in S_{m}\}$ for $P$, where the probability $Q_{t}$ on $\sW\times\sY$ is given by
\[
Q_{t}(A\times B)=\int_{A}\cro{\int_{B}t_{w}(y)\,d\lambda(y)}dP_{W}(w)\;\mbox{ i.e. }\;\frac{dQ_{t}}{dP_{W}\otimes d\lambda}(w,y)=t(w,y).
\]
This means that $Q_{t}$ has a marginal distribution $P_{W}$ on $\sW$ and a conditional distribution given $W=w$ with density $t_{w}$ with respect to $\lambda$. Note that the $\rho$-models ${\sQ}_{m}$ depend on the unknown distribution $P_{W}$ but the densities with respect to the dominating measure $P_{W}\otimes \lambda$ do not. This leads to a family of $\rho$-models ${\sbQ}_{m}$ for $\gP$ and a reference $\rho$-model ${\sbQ}=\bigcup_{m\in\cM}{\sbQ}_{m}$. If we introduce a suitable penalty $\gpen$ on ${\sbQ}$, we may build a $\rho$-estimator of $\gP$ from our sample $\etc{X}$ according to the recipe of Section~\ref{ES1} since its values only depend on the family of densities in ${\cQ}=\bigcup_{m\in\cM}S_{m}$.
As a consequence, our estimation strategy neither needs to know $P_{W}$ nor to estimate it. Such a $\rho$-estimator will be of the form $Q_{\widehat{s}}^{\otimes n}$ with $Q_{\widehat{s}}=\widehat{s}\cdot(P_{W}\otimes \lambda)$ and will provide an estimator $\widehat{s}_{w}\cdot \lambda$ of the conditional probability $P_{w}$.

Within this framework, the Hellinger distance between the probabilities at hand writes, for any measure $\nu$ that dominates both $P$ and $P_{W}\otimes \lambda$,
%beg
\begin{align*}
&{1\over 2}\int_{\sW\times\sY}\pa{\sqrt{\frac{dP_{W}(w)dP_{w}(y)}{d\nu}}
-\sqrt{t(w,y)\frac{dP_{W}(w)d\lambda(y)}{d\nu}}}^{2}d\nu(w,y)\qquad\\
&\hspace{50mm}= h^{2}(P,Q_{t})=\int_{\sW}h^{2}(P_{w},t_{w}\cdot \lambda)\,dP_{W}(w).
\end{align*}
Therefore
%beg
\begin{equation}
h^{2}(P,{\sQ}_{m})=\inf_{t\in S_{m}}\int_{\sW}h^{2}(P_{w},t_{w}\cdot \lambda)\,dP_{W}(w).
\label{eq-loss}
\end{equation}
%end
Note that $h^{2}(P,Q_{t})$ can actually be viewed as a loss function for the conditional distributions, of the form $\ell(P_{w},t_{w}\cdot \lambda)$ since it actually only depends on  $P_{w}$ and $t_{w}$.

%SUBSECTION
\subsection{Assumptions and results}\label{CD2}
Let us assume the following:
\begin{ass}\label{Ass-cd}
For all $m\in \cM$, $S_{m}$ is VC-subgraph with index not larger $\overline V_{m}$.
\end{ass}
We may then deduce from Theorem~\ref{thm-main} the following result.
%
% COROLLARY
\begin{cor}\label{cor-conddens}
Let $\{S_{m},\ m\in\cM\}$ be a family of countable subsets of $\cL_{c}(\sW,\lambda)$ satisfying Assumption~\ref{Ass-cd},  $\Delta$ be a weight function on $\cM$ which satisfies~\eref{eq-delta}, $\psi$ a function satisfying Assumption~\ref{H-debase}, ${\sQ}=\bigcup_{m\in\cM}\{Q_{t},\,t\in S_{m}\}$ and $\pen: {\sQ}\to \R_{+}$ given, for all $Q\in{\sQ}$, by 
\[
\pen(Q)=\kappa\inf_{\{m\in\cM,\,Q=Q_{t}\;{\rm with}\;t\in  S_{m}\}}\!\cro{{C_{1}\over4.7}(\overline V_{m}\wedge n)\!\cro{1+\log_{+}\left(\frac{n}{\overline V_{m}}\right)}+\Delta(m)}
\]
where $\kappa$ is given by \eref{cond-kappa} and $C_{1}$ is the constant appearing in Proposition~\ref{cas-VC}. Then any density $\rho$-estimator $Q_{\widehat s}$ relative to $((\mu,\cQ),\pen)$ satisfies, for some constant $C'>0$ depending on the choice of $\psi$ only,
%beg
\begin{align*}
\lefteqn{\E\cro{h^{2}(P,Q_{\widehat{s}})}}\hspace{10mm}\\&\le C'\inf_{m\in\cM}\cro{h^{2}(P,{\sQ}_{m})+{\overline V_{m}\wedge n\over n}\pa{1+\log_{+}\left(\frac{n}{\overline V_{m}}\right)}+{\Delta(m)\over n}}.
\end{align*}
%end
\end{cor}
Note that this result does not require any information or assumption on the distribution of $W$. If, in particular, the conditional probability $P_{w}$ is absolutely continuous with respect to $\lambda$ for almost all $w$ with density $dP_{w}/d\lambda=s_{w}$, $P_{W}$-a.s., one can write 
\[
\left\{
\begin{array}{lll}  
h^{2}(P,Q_{\widehat{s}})&=&\dps{\int_{\sW}h^{2}(s_{w}\cdot \lambda,\widehat{s}_{w}\cdot \lambda)\,dP_{W}(w)},\\
h^{2}(P,{\sQ}_{m})&=&\dps{\inf_{t\in S_{m}}\int_{\sW}h^{2}(s_{w}\cdot \lambda,t_{w}\cdot \lambda)\,dP_{W}(w).}\end{array}\right.
\]
\begin{proof}[Proof of Corollary~\ref{cor-conddens}]
Applying Propositions~\ref{cas-VC} and~\ref{cas-VCd} to each $\rho$-model ${\sbQ}_{m}$ with $m\in\cM$, we obtain under Assumption~\ref{Ass-cd} the existence of a universal constant $C_{1}>0$ such that for all $\left(\gP,\overline \gP\right)\in\sbP\times \sbP^{\gmu}$
\[
D^{{\sbQ}_{m}}(\gP,\overline \gP)\le D_{n}(m)=C_{1}(\overline V_{m}\wedge n)\cro{1+\log_{+}(n/\overline V_{m})}.
\]
Inequality \eref{eq-L4} is fulfilled with $K=0$ and the penalty function therefore satisfies~\eref{eq-L3} with $\kappa_{1}=0$ for all $m\in\cM$. The result follows from Theorem~\ref{thm-main4}, then an integration of \eref{eq-Th2-Prob} with respect to $\xi>0$.
\end{proof}
\section{Regression with a random design}\label{RS}
In this section we assume that the observations $X_{i}=(W_{i},Y_{i})$, $1\le i\le n$ are i.i.d.\ copies of a random pair
\begin{equation}\label{mod-reg}
X=(W,Y)\qquad\mbox{with}\qquad Y=f(W)+\ee,
\end{equation}
where $W$ is a random variable with distribution $P_{W}$ on a measurable space $(\sW,\sB(\sW))$, $f$ is an unknown regression function  mapping $\sW$ into $\R$ and $\ee$ is  a real-valued random variable with distribution $P_{\ee}$, which is independent of $W$. Both distributions $P_{W}$ and $P_{\ee}$ are assumed to be unknown. We shall use the specific notations introduced in Section~\ref{Dens} when the data are i.i.d.\ and denote by $\mu$ the product measure $P_{W}\otimes \lambda$ where $\lambda$ is the Lebesgue mesure on $\R$. Note that $\mu$ is unknown since it depends on the distribution $P_{W}$ of the design $W$. 

If $\ee$ had a density $s$ with respect to $\lambda$, the distribution $P$ of $X=(W,Y)$ would be absolutely continuous with respect to $\mu$ with density $p$ given by 
\begin{equation}\label{def-s}
p(w,y)=s(y-f(w))\quad\mbox{for }(w,y)\in\sX,
\end{equation}
depending thus on two parameters: the density $s$ of the errors and the regression function $f$. 

Denoting by $\sD$ the set of all densities on $(\R,\sB(\R),\lambda)$ and $\sF$ the set of all measurable functions mapping $\sW$ into $\R$, 
our aim is to estimate $P$ assuming that it is close to some distribution of the form $p\cdot \mu$ with $p$ given by~\eref{def-s} for some $s\in\sD$ and $f\in \sF$. Besides, when $P$ is truly of this form we shall also derive estimators for both $s$ and $f$.
\subsection{The main result}\label{RS1}
For $r\in\sD$ and $g\in\sF$, we set
\[
Q_{r,g}=q_{r,g}\cdot \mu\quad\mbox{ with }\quad q_{r,g}(w,y)=r(y-g(w)),
\]
which means that $Q_{r,g}$ is the distribution of $X$ in~\eref{mod-reg} when $f=g$ and $\ee$ is distributed according to $R=r\cdot \lambda$. Given a density $r\in\sD$ and a countable subset  $F$ of $\sF$, we define the $\rho$-model
\[
{\sQ}_{m}=\{Q_{r,g},\ g\in F\}\quad\mbox{for }m=(r,F).
\]
Given a countable subset $\cD$ of $\sD$ and a countable family $\F$ of countable subsets  $F$ of $\sF$, we estimate $P$ on the basis of the collection of $\rho$-models $\{{\sQ}_{m},\ m\in\cM\}$ with $\cM\subset \cD\times \F$. We endow the family $\{{\sbQ}_{m},\ m\in\cM\}$ with a weight function $\Delta$ satisfying~\eref{eq-delta} and assume the following.
%
% ASSUMPTION
\begin{ass}\label{ass-1}\mbox{}\vspace{-2mm}
\begin{listi}
\item\label{i-ass1} The densities $r\in \cD$ are unimodal.
\item\label{i-ass2} Each $F$ in $\F$ is VC-subgraph with index $\overline V(F)$.
\item\label{i-ass3} The function $\psi$ satisfies Assumption~\ref{H-debase} with ${\sbQ}=\{Q^{\otimes n},\;Q\in{\sQ}=\bigcup_{m\in\cM}{\sQ}_{m}\}$.
\end{listi}
\end{ass}
Under Assumptions~\ref{ass-1}-$\ref{i-ass1})$ and~\ref{ass-1}-$\ref{i-ass2})$, the family of densities ${\cQ}_{m}$ is VC-subgraph on $\sX$ with index not larger than
%beg
\begin{equation}
\overline V_{m}=9.41\overline V(F)\quad\mbox{for all }m=(r,F)\in\cM.
\label{eq-barVm}
\end{equation}
%end
This result derives from Baraud {\em et al.}~\citeyearpar[Proposition~42]{MR3595933}. Besides, under Assumption~\ref{ass-1}-$\ref{i-ass3})$, Proposition~\ref{cas-VC} applies and implies that, for some universal constant $C_{1}>0$, all $m\in\cM$, $\gP\in\sbP$ and $\overline \gP\in {\sbQ}$, $D^{{\sbQ}_{m}}(\gP,\overline \gP)\le D_{n}(m)$ with 
%beg
\begin{equation}
D_{n}(m)=C_{1}(\overline V_{m}\wedge n)\cro{1+\log_{+}\!\pa{n/\overline V_{m}}}\quad\mbox{for all }m=(r,F)\in\cM,
\label{def-Dm}
\end{equation}
%end
so that \eref{eq-L4} holds with $K=0$. Setting
\[
\gpen(\gQ)=\kappa\inf_{\{m\in\cM\,|\,{\sQ}_{m}\ni Q\}}\ac{{ D_{n}(m)\over4.7}+\Delta(m)},
\]
we may apply Theorem~\ref{thm-main4} with $\kappa_{1}=0$, which leads, in this particular case, to the following analogue of (\ref{eq-Th2-Prob}).
\begin{thm}\label{thm-main2}
Assume that Assumption~\ref{ass-1} holds. For any distribution $P\in\sP$ and $\gP=P^{\otimes n}$, any $\rho$-estimator $\widehat \gP=(\widehat P,\ldots,\widehat P)$ satisfies, for all $\xi>0$, with probability at least $1-e^{-\xi}$,
\begin{align}
\lefteqn{Ch^{2}(P,\widehat P)}\hspace{10mm}\nonumber \\
&\le  \inf_{m\in\cM} \cro{h^{2}(P,{\sQ}_{m})+{\overline V_{m}\wedge n\over n}\cro{1+\log_{+}\!\pa{\frac{n}{\overline V_{m}}}}+{\Delta(m)\over n}}+{\xi\over n},\label{eq-risk3r}
\end{align}
for some constant $C>0$ only depending on the choice of $\psi$.
\end{thm}
At this stage, some comments are in order.
\begin{enumerate}[label=\alph*),ref=\roman*,leftmargin=\parindent]
\item This result holds without any assumption on the distribution $P_W$ of the design.  
\item The result is true even if the regression framework~\eref{mod-reg} is not exact as long as the $X_{i}$ are i.i.d. In particular, the distribution $P$ needs not have a density with respect to $\mu=P_{W}\otimes \lambda$.
\item If $r$ admits $k$ modes with $k> 1$ and $F$ is VC-subgraph with index not larger than $\overline V$, ${\sQ}_{(r,F)}$ remains VC-subgraph and its index is still bounded by $C(k)\overline V$ for some constant $C(k)$ that now depends on $k$. Consequently the above result generalizes to families $\cD$ of densities admitting more than a single mode in which case $\overline V(F)$ should be replaced by $c(r)\overline V(F)$ where $c(r)$ is a positive number depending on the number of modes of the density $r$.
\item With Theorem~\ref{thm-main2} at hand we could obtain in the present random design context an analogue  of Corollary~39 in Baraud {\em et al.}~\citeyearpar{MR3595933} which was established when the $W_{i}$ were deterministic (fixed design regression).
\end{enumerate}

\subsection{Estimation of \texorpdfstring{$s$}{Lg} and \texorpdfstring{$f$}{Lg}}
Let us now consider the situation where the regression framework~\eref{mod-reg} is exact and $\ee$ has an unknown density $s$ with respect to the Lebesgue measure $\lambda$. Then $P=Q_{s,f}$ admits a density $q_{s,f}$ with respect to $\mu$ which is given by~\eref{def-s} with $s$ belonging to 
$\sD$ and $f$ to $\sF$ but not necessarily to our $\rho$-models $\cD$ and $\cF=\bigcup_{F\in\F}F$ respectively. Since we may choose our $\rho$-estimator of the form 
\[
\widehat P= q_{\widehat s,\widehat f}\cdot \mu\quad\mbox{with}\quad(\widehat s,\widehat f)\in\cD\times\cF,
\] 
our procedure results in estimators $\widehat s$ and $\widehat f$ for $s$ and $f$ respectively and our aim in this section is to establish risk bounds for these two estimators. 

Since the map $(r,g)\mapsto Q_{r,g}$ is not necessarily one-to-one from $\sD\times \sF$ to $\sP$ an identifiability condition is required on our $\rho$-model ${\sQ}$ so that the equality $Q_{r,g}=Q_{r',g'}$ with $r,r'\in\cD$ and $g,g'\in\cF$ implies that $r=r'$ $\lambda$-a.e.\ and $g=g'$ $P_{W}$-a.s. In order to state this identifiability condition, let us introduce the following notation. 
For $r\in\sD$ and $a\in\R$, we shall denote by $R_{a}$ the probability on $(\R,\sB(\R),\lambda)$ with density $r_{a}(\cdot)=r(\cdot-a)$. When $\ee$ has density $r$ and $a=g(w)$ for some $w\in\sW$, $R_{a}$ can be viewed as the conditional distribution of $Y=g(W)+\ee$ given $W=w$. Given $r,r'\in\sD$,  $g,g'\in\sF$ and $w\in\sW$, the Hellinger distance between the probabilities $R_{g(w)}$, and $R'_{g'(w)}$ is given by
\[
h^{2}\!\left(R_{g(w)},R'_{g'(w)}\right)={1\over 2}\int_{\R}
\cro{\sqrt{r\pa{y-g(w)}}-\sqrt{r'\pa{y-g'(w)}}}^{2}d\lambda(y)
\]
and the Hellinger distance between the corresponding probabilities $Q_{r,g}$ and $Q_{r',g'}$ on $(\sX,\sB)$ writes 
\begin{equation}\label{eq-HQ}
h^{2}\!\left(Q_{r,g},Q_{r',g'}\right)=\int_{\sW}h^{2}\!\pa{R_{g(w)},R'_{g'(w)}}dP_{W}(w).
\end{equation}
We recall that the Hellinger distance is translation invariant which means that for all densities $r,r' \in\sD$, $a,a'\in\R$,
\begin{equation}\label{eq-hel}
h^{2}(R_{a},R'_{a'})=h^{2}(R_{a-a'},R').
\end{equation}
In particular, taking $a=g(w)$ and $a'=g'(w)$ for $g,g'\in\sF$ and $w\in\sW$ and integrating~\eref{eq-hel} with respect to $P_{W}$ we get for all $(g,g')\in\sF^{2}$ and $(r,r')\in\sD^{2}$ 
\begin{equation}\label{eq-hel2}
h^{2}\!\left(Q_{r,g},Q_{r',g'}\right)=h^{2}\!\left(Q_{r,g-g'},Q_{r',0}\right).
\end{equation}
In order to warrant identifiability, we assume the following.
\begin{ass}\label{ass-2}
There exists a positive constant $A$ such that, for all $r,r'\in\cD$, $R=r\cdot \lambda$ and $R'=r'\cdot \lambda$,
\[
h(R,R')\le A\inf_{a\in\R}h(R_{a},R').
\]
\end{ass}
When $Q_{r,g}=Q_{r',g'}$,~\eref{eq-HQ} asserts that $h\!\left(R_{g(w)},R'_{g'(w)}\right)=0$ for $P_{W}$-almost all $w\in\sW$ and~\eref{eq-hel} implies that $h\!\left(R_{g(w)-g'(w)},R'\right)=0$ for all such $w\in\sW$.  Applying Assumption~\ref{ass-2} with $a=g(w)-g'(w)$ leads to $R=R'$ and $g(w)=g'(w)$ which solves our identifiability problem.

In order to evaluate the risk of our estimator $\widehat f$ of $f$, we endow $\sF$ with the loss function $d_{s}$ defined on $\sF\times\sF$ by 
\[
d_{s}^{2}(g,g')={1\over 2}\int_{\sW\times \R}\pa{\sqrt{s_{g(w)}(y)}-\sqrt{s_{g'(w)}(y)}}^{2}dP_{W}(w)\,dy\quad\mbox{for }g,g'\in \sF.
\]
This loss function depends on the true density $s$ of the errors $\ee$ and on the distribution $P_{W}$ of the design, hence on $P$. We have seen in Section~6.3 of Baraud {\em et al.}~\citeyearpar{MR3595933} that, if the density $s$ is of order $\alpha$ with $\alpha\in (-1,1]$ (see Definition~26 of that paper for the order of a function),
the restriction of $d_{s}$ to the $\L_{\infty}(P_{W})$-ball $\sB_{\infty}(b)$ centred at 0 with radius $b$ is equivalent (up to factors depending on $b$ and $s$) to 
\[
\norm{g-g'}_{1+\alpha,P_{W}}^{(1+\alpha)/2}\quad\mbox{ with }\quad\norm{g-g'}_{1+\alpha,P_{W}}=\cro{\int_{\sW}\ab{g-g'}^{1+\alpha}dP_{W}}^{1/(1+\alpha)}.
\]
In particular, if $\cF\subset \sB_{\infty}(b)$ and the true regression function $f$ also belongs to $\sB_{\infty}(b)$,
\[
c(s,b)\norm{f-g}_{1+\alpha,P_{W}}^{(1+\alpha)/2}\le d_{s}(f,g)\le C(s,b)\norm{f-g}_{1+\alpha,P_{W}}^{(1+\alpha)/2}\quad\mbox{for all }g\in\cF
\]
and suitable positive numbers $c(s,b)$ and $C(s,b)$. Of special interest is the case of $\alpha=1$ for which $d_{s}(f,g)$ is of the order of the $\L_{2}(P_{W})$-distance between $f$ and $g$ for all $g\in\cF$. This situation is met when the translation $\rho$-model associated to $s$ is regular which is the case when $s$ is Cauchy, Gaussian, Laplace, etc. When $s$ is uniform or exponential, $d_{s}^{2}(\cdot,\cdot)$ is then equivalent to the $\L_{1}(P_{W})$-norm. Furthermore, when $\cF$ is a subset of a 
$k$-dimensional linear space $\overline \cF$ generated by $\phi_{1},\ldots,\phi_{k}$, these norms can in turn be translated into a norm on $\R^{k}$ between the coefficients relative to this basis. More precisely, if $f$ belongs to $\overline {\cF}$ and writes as $\sum_{j=1}^{k}\beta_{j}\phi_{j}$ and $\overline f=\sum_{j=1}^{k}\overline \beta_{j}\phi_{j}$ is an element of $\overline {\cF}$, there exist two positive constants $c'(s,b,k)$ and $C'(s,b,k)$ such that
%beg
\begin{align*}
\lefteqn{c'(s,b,k) \cro{\max_{j=1,\ldots,k}\ab{\beta_{j}-\overline \beta_{j}}}^{(1+\alpha)/2}}\hspace{35mm}\\&\le d_{s}(f,\overline f) \le C'(s,b,k) \cro{\max_{j=1,\ldots,k}\ab{\beta_{j}-\overline \beta_{j}}}^{(1+\alpha)/2}.
\end{align*}
%end
In particular, if $\overline f=\overline f(n)$ converges toward $f$ at a rate $v_{n}$ with respect to the distance $d_{s}(\cdot,\cdot)$, the coefficients of $\overline f$ converge in sup-norm toward those of $f$ at rate $v_{n}^{2/(1+\alpha)}$, this latter rate being faster than $v_{n}$ when $\alpha<1$.

In view of evaluating the risk of our estimator $\widehat s$ of the density $s$ we shall consider the loss between two densities $r,r'\in\sD$ induced by the Hellinger distance between the two corresponding measures $r\cdot \lambda$ and $r'\cdot \lambda$ and we shall write this loss $h(r,r')$ so that
\[
h(r,r')=h(r\cdot \lambda,r'\cdot \lambda)=h\!\left(Q_{r,0},Q_{r',0}\right)\quad\mbox{for all }r,r'\in \sD. 
\]
We deduce from Theorem~\ref{thm-main2} the following result.
%
% COROLLARY
\begin{cor}\label{cor-reg}
Assume that the $X_{i}$ are i.i.d.\ with density $p$ given by~\eref{def-s} and that 
Assumptions~\ref{ass-1} and~\ref{ass-2} are satisfied. For all $\xi>0$ and all $\rho$-estimators $Q_{\widehat s,\widehat f}$, with $\widehat s\in\cD$ and $\widehat f\in \cF$, based on the family of $\rho$-models $\sQ_{m}$ defined in Section~\ref{RS1}, with probability at least $1-e^{-\xi}$,
%beg
\begin{align*}
\lefteqn{C\max\ac{d_{s}^{2}(f,\widehat f),h^{2}(s,\widehat s)}}\hspace{5mm}\\
&\le\inf\ac{d_{s}^{2}(f,F)+h^{2}(s,r)+{\overline V_{m}\wedge n\over n}\cro{1+\log_{+}\!\pa{\frac{n}{\overline V_{m}}}}+{\Delta(m)\over n}}+{\xi\over n},
\end{align*}
%end
where the infimum runs among all $m=(r,F)\in\cM$, $C$ is a positive constant depending on $A$ and the choice of $\psi$ and $\overline{V}_{m}$ is given by (\ref{eq-barVm}).
\end{cor}
The risk bound is the same for the two estimators and depends on the approximation properties of $\cF$ and $\cD$ with respect to $f$ and $s$ respectively. The proof of this corollary is given in Section~\ref{sect-pcor-reg} of the Appendix.

%SECTION 
\section{Estimator selection and aggregation}\label{S-ES-A}
In the case of density estimation, $\rho$-estimators can also be used to perform selection or aggregation of preliminary estimators. In this case, we assume that we have at hand a set $\bsX_{1}=(\etc X)$ of $n$ independent random variables with an unknown joint distribution $\gP$ to be estimated. We also have at hand a finite family ${\sbQ}=\left\{\gP_{j},\,j\in\cJ\right\}$ of probabilities that can be considered as candidate estimators for $\gP$. These are completely arbitrary but, in a typical situation, it is assumed (although this may not be true) that the observations $X_{i}$ are i.i.d.\ and the $\gP_{j}$ are preliminary estimators of the form $\gP_{j}=P_{j}^{\otimes n}(\bsX_{2})$, where $\bsX_{2}$ is a second sample independent from $\bsX_{1}$, and the ${P}_{j}=P_{j}(\bsX_{2})$ are estimators that derive from various  procedures applied to the sample $\bsX_{2}$.

%SUBSECTION
\subsection{Estimator selection}\label{S1}
Taking $\cM=\cJ$, we view each probability $\gP_{j}$ as a $\rho$-model ${\sbQ}_{j}=\ac{\gP_{j}}$ with a single element. As a consequence, it follows from Proposition~\ref{cas-fini} that $D^{{\sbQ}_{j}}(\gP,\gP_{j})\le9\log2< 6.3$ so that (\ref{eq-L4}) holds with $D_{n}(j)=6.3$ for all $j$ and $K=0$. Then we choose the weights $\Delta(j)$ satisfying (\ref{eq-delta}). We may choose $\Delta(j)=\log |\cJ|$ for all $j\in\cJ$ but other more Bayesian choices are possible, or choices based on the confidence we have in the various procedures used to build the preliminary estimators. To compute the penalized $\rho$-estimator $\widehat{\gP}$ of $\gP$, we may use the penalty function $\gpen\left(\gP_{j}\right)=\kappa \Delta(j)$ for all $j\in\cJ$ which is of the form~\eref{eq-L3} with $\kappa_{1}=-\kappa(6.3/4.7)$. Finally (\ref{eq-Th2-Prob}) shows that, for all $\xi>0$ and $\gP\in\sbP$, 
\[
\P\left[\gh^{2}(\gP,\widehat \gP)\le C\inf_{j\in\cJ}\left[\gh^{2}(\gP,\gP_{j})+\Delta(j)+1+\xi\right]\right]\ge1-e^{-\xi},
\]
where $C$ denotes a suitable constant depending on $\psi$ only.

%SUPPLEMENT (A CITER QQUE PART)

%SUBSECTION
\subsection{Convex estimator aggregation}\label{S2}
In this case we set $\cJ=\{1,\ldots,N\}$, $N\ge2$ and $\sC$ for the $N$-dimensional simplex:
\[
\sC=\left\{(\Etc{\alpha}{N})\mbox{ such that }\alpha_{j}\ge0\mbox{ for }1\le j\le N\mbox{ and }\sum_{j=1}^{N}\alpha_{j}=1\right\}.
\]
We select a dominating measure $\mu$, densities $p_{j}=dP_{j}/d\mu$ and we then consider a single density model 
\[
\overline{\cQ}=\left\{\sum_{j=1}^{N}\alpha_{j}p_{j}\quad\mbox{for}\quad(\Etc{\alpha}{N})\in\sC\right\}.
\]
The following result then holds.
%
% PROPOSITION
\begin{prop}\label{prop-agreg}
Let $\psi$ satisfy Assumption~\ref{H-debase}. Any $\rho$-estimator $\widehat \gP$ on $\overline \sbQ=\{\gq\cdot \gmu,\; \gq\in \overline \cbQ\}$ relative to $((\gmu, \overline \cbQ),\bf{0})$ satisfies
\[
\P\cro{Ch^{2}(\gP,\widehat \gP)\le h^{2}(\gP,\overline \sbQ)+{N\log n}+{\xi}}\ge 1-e^{-\xi},\quad \mbox{for all $\xi>0$}
\]
and some constant $C>0$ depending on $\psi$ only.
\end{prop}
%
% PROOF
\begin{proof}
The map $(\alpha_{1},\ldots,\alpha_{n})\mapsto \sum_{j=1}^{N}\alpha_{j}p_{j}$ from $\sC$ to $\overline{\cQ}$ is continuous if we equip $\sC$ with the usual Euclidean distance and $\overline{\cQ}$ with the topology of pointwise convergence. Since $\sC$ is separable, 
$\overline{\cQ}$ is universally separable and contains thus a countable subset $\cQ$ which is $\sbT$-dense in $\overline \cQ$. The set $\overline \cQ$ being furthermore a subset of an $N$-dimensional linear space, it is VC-subgraph with
index $\overline{V}$ not larger than $N+2$ and it follows from Proposition~\ref{cas-VC} (and Proposition~\ref{cas-VCd}) that, for all $\rho$-models $\sbQ\subset \overline \sbQ$ 
\[
D^{{\sbQ}}(\gP,\overline \gP)\le D_{n}=CN\log n\quad\mbox{for all }\left(\gP,\overline \gP\right)\in\sbP\times\sbP^{\gmu}.
\]
We may apply Theorem~\ref{thm-Exten} with $p\equiv 0$ to the single model $\overline \sbQ$. A $\rho$-estimator $\widehat \gP$ on $\overline \sbQ$ relative to $((\gmu,\overline \cbQ),{\bf 0})$ is also a $\rho$-estimator on the $\rho$-model $\sbQ=\{\gq\cdot \gmu,\; \gq\in \cbQ\}$ and we deduce from Theorem~\ref{thm-main1}, more precisely from \eref{eq-risk1b}, that 
\[
\P\cro{C\gh^{2}(\gP,\widehat \gP)\le \gh^{2}(\gP,\sbQ)+N\log n+\xi}\ge 1-e^{-\xi},\quad \mbox{for all $\xi>0$},
\]
some constant $C$ depending on $\psi$ only and all $\gP\in \sbP$.
We conclude using the fact that $\sbQ$ is $\sbV$-dense in $\overline \sbQ$ since $\cQ$ is $\sbT$-dense in $\overline \cQ$. 
\end{proof}
It should be noted that there is no $\L_{2}$-type argument here, the densities $p_{j}$ can be absolutely anything and the true distribution $\gP$ should be a product measure but not necessarily of the form $P^{\otimes n}$. 

\paragraph{{\bf Practical implementation}.} Since the set $\overline{\cQ}$ is convex, Proposition~\ref{prop-SU} applies with $\psi=\psi_{1}$ or $\psi=\psi_{2}$ of Proposition~\ref{prop-expsi} and the MLE on $\overline{\cQ}$ is a $\rho$-estimator. It is obtained by maximizing over the convex set $\sC$ the concave map 
\[
(\alpha_{1},\ldots,\alpha_{N})\mapsto \sum_{i=1}^{n}\log\pa{\sum_{j=1}^{N}\alpha_{j}p_{j}(X_{i})}.
\]

\section*{Acknowledgements}
The authors are grateful to Weijie Su for letting them know about the nice connection between the MLE and $\rho$-estimators in the case of a convex parameter set and for allowing them to include his result (Proposition~\ref{prop-SU}) in this paper.

%%%%% Bibliographie A-MonLatex alias Library/texmf/bibtex/bib/
%\bibliographystyle{apalike}
%\bibliography{biblio}

%\end{document}

\pagebreak
\appendix

%SECTION
\section{Proofs of the main results}

%SUBSECTION
\subsection{A general result}\label{P1}
Let $\{\sbQ_{m},\ m\in\cM\}$ be a countable collection of $\rho$-models, $\Delta$ a weight function satisfying~\eref{eq-delta}, $\psi$ a function satisfying Assumption~\ref{H-debase} with $\sbQ=\bigcup_{m\in\cM}\sbQ_{m}$. We fix a representation $\cR=(\gmu,\cbQ)$ of $\sbQ$ leading to a (unique) representation $\cR_{m}$ 
of $\sbQ_{m}\cup\{\overline \gP\}\subset \sbQ$ for each $m\in\cM$ and $\overline \gP\in\sbQ$.
%
% THEOREM
\begin{thm}\label{thm-main}
Assume that there exists a real-valued function $G$ on $\sbP\times \sbQ$ and real non-negative numbers $\{D_{n}(m),\; m\in\cM\}$ such that, for all $(\gP,\overline{\gP})\in \sbP\times \sbQ$ and all $m\in\cM$,
%beg
\begin{equation}
w(\cR_{m},\sbQ_{m},\gP,\overline \gP,y)\le\frac{a_{1}y^{2}}{8}+8\gh^{2}(\gP,\overline \gP)\quad \mbox{for all}\quad  y>y_{m}
 \label{eq-L1}
\end{equation}
%end
where 
\[
y_{m}=\beta^{-1}\sqrt{{4.7G(\gP,\overline\gP)\over \kappa}+D_{n}(m)}\quad\mbox{for all}\quad m\in\cM
\]
and $\kappa$ is given by \eref{cond-kappa}. Let the penalty function $\gpen$ from $\sbQ$ to $\R_+$ satisfy
%beg
\begin{equation}
\gpen(\gQ)\ge\kappa\inf_{\{m\in\cM\,|\,\sbQ_{m}\ni\gQ\}}\left[\frac{D_{n}(m)}{4.7}+\Delta(m)\right]\;\mbox{ for all }\gQ\in\sbQ.
\label{eq-L2}
\end{equation}
%end
Then any $\rho$-estimator $\widehat \gP$ on $\sbQ$ satisfies, for all $\overline{\gP}\in\sbQ$ and $\xi>0$, with probability at least $1-e^{-\xi}$,
%beg
\begin{align}
 \gh^{2}(\gP,\widehat \gP) \le &\,\gamma  \gh^{2}(\gP,\overline \gP)-\gh^{2}(\gP,\sbQ)\label{eq-L0} \\
&\,+(4/a_{1})\cro{G(\gP,\overline\gP)+\gpen(\overline \gP)+\kappa(1.49+\xi)},\nonumber
\end{align}
%end
where $\gamma$ given by (\ref{cond-kappa}).
\end{thm}
This theorem is based on an auxiliary result to be proved in Section~\ref{sect-pprop-dev} of the Appendix.
%
% PROPOSITION
\begin{prop}\label{prop-dev}
Let $\co>0$, $K'>0$, $B=1+a_{2}^{2}/(16\co)$ and numbers $\alpha>\delta>1$ and $\vartheta>1$ be such that
%beg
\begin{equation}
2\exp[-\vartheta]+\sum_{j\ge1}\exp\left[-\vartheta\delta^{j}\right]\le1.
\label{eq-sum1}
\end{equation}
%end
If, for all $m\in\cM$,
%beg
\begin{equation}
w(\cR_{m},\sbQ_{m},\gP,\overline \gP,y)\le \co y^{2}+K'\gh^{2}\left(\gP,\overline{\gP}\right)\;\mbox{ for all }y>\sqrt{\widetilde D_{m}(\gP,\overline \gP)},
\label{eq-dev}%
\end{equation}
%end
then, for any $\xi>0$ and for all $m\in\cM$ simultaneously, with probability at least $1-e^{-\xi}$, 
%beg
\begin{align}
&\hspace{-5mm}\sup_{\gQ\in\sbQ_{m}}\cro{\str{3.7}\ab{\st\gZ(\bsX,\overline{\gp},\gq)}- \co\alpha\cro{\gh^{2}(\gP,\overline \gP)+\gh^{2}(\gP,\gQ)}}\label{Eq-fonda}\\
\hspace{10mm}\le&\, K\left[\co
\widetilde D_{m}(\gP,\overline \gP)\bigvee\tau\left[\st\Delta(m)+\vartheta+\xi\right]\right]+K'\left(1+\frac{4\sqrt{2}}{\sqrt{B\tau}}\right)\gh^{2}\!\left(\gP,\overline{\gP}\right)\nonumber,
\end{align}
%end
with
%beg
\begin{equation}
\tau\ge\tau_{0}=\frac{1}{2\delta B}\left[\sqrt{1+\frac{\alpha-\delta}{8\delta B}}-1\right]^{-2}\quad\mbox{ and }\quad K=1+8\sqrt{{2B\over\tau}}+\frac{4}{\tau}.
\label{Eq-tau}
\end{equation}
\end{prop}
Let us now turn to the proof of Theorem~\ref{thm-main}. It follows from (\ref{eq-L1})  that \eref{eq-dev} is satisfied with $\widetilde D_{m}(\gP,\overline \gP)=\beta^{-2}\left[4.7\kappa^{-1}G(\gP,\overline\gP)+D_{n}(m)\right]$, $\co={a_{1}/8}$ and $K'=8$. We may therefore apply Proposition~\ref{prop-dev} with $\alpha=4$, $\tau=\tau_{0}$, $\delta=1.175$ and $\vartheta=1.47$, which implies that (\ref{eq-sum1}) is satisfied. Let us then set 
\[
a'=a_{1}/2,\;\; c_{1}=1+8\sqrt{{2B\over\tau_{0}}}+\frac{4}{\tau_{0}},\;\;c_{2}=
8\left(1+\frac{4\sqrt{2}}{\sqrt{B\tau_{0}}}\right)\;\;\mbox{and}\;\;\xi'=\kappa(\xi+\vartheta).
\]
We deduce from (\ref{Eq-fonda}) and the chosen values of the various constants involved that, on a set $\Omega_{\xi}$ the probability of which is at least $1-e^{-\xi}$, for all $\gQ\in\sbQ$ and all $\rho$-models $\sbQ_{m}$ containing $\gQ$,
%beg
\begin{align}
\gT(\bsX,\overline \gp,\gq)\le&\; \E\cro{\gT(\bsX,\overline \gp,\gq)}+a'\cro{\gh^{2}(\gP,\overline \gP)+\gh^{2}(\gP,\gQ)}\nonumber\\&\;
+c_{1}\left[\co\widetilde{D}_{m}(\gP,\overline \gP)+\tau_{0}\left(\Delta(m)+\kappa^{-1}\xi'\right)\right]+c_{2}\gh^{2}(\gP,\overline \gP)\nonumber\\
=&\;\E\cro{\gT(\bsX,\overline \gp,\gq)}+a'\cro{\gh^{2}(\gP,\overline \gP)+\gh^{2}(\gP,\gQ)}+c_{2}\gh^{2}(\gP,\overline \gP)\nonumber\\
&\;+\tau_{0}c_{1}\left[\frac{\co}{\tau_{0}\beta^{2}}\left[4.7\kappa^{-1}G(\gP,\overline\gP)+D_{n}(m)\right]+\Delta(m)+\kappa^{-1}\xi'\right].
\label{eq-boundT}
\end{align}
%end
Now observe that $B=1+a_{2}^{2}/(2a_{1})$, hence $B\ge4$ since $a_{2}^{2}\ge 6a_{1}$ and $(\alpha-\delta)/(8\delta B)<0.07514$. Since
\[
0.4909\,x<\sqrt{1+x}-1<0.5\,x\quad\mbox{for } 0<x<0.07514,
\]
it follows that
\[
\frac{0.1475}{B}<\sqrt{1+\frac{\alpha-\delta}{8\delta B}}-1<\frac{\alpha-\delta}{16\delta B}<\frac{0.1503}{B},
\]
hence
%beg
\begin{equation}
18.837B<\tau_{0}=\frac{1}{2\delta B}\left[\sqrt{1+\frac{\alpha-\delta}{8\delta B}}-1\right]^{-2}<19.56B
\label{eq-Btau}
\end{equation}
%end
and
\[
\tau_{0}c_{1}=\tau_{0}+8\sqrt{2B\tau_{0}}+4<69.6B+4\le\frac{34.8a_{2}^{2}}{a_{1}}+73.6<\kappa.
\]
Moreover, (\ref{eq-Btau}) and the inequality $B> a_{2}^{2}/(2a_{1})$ imply that
\[
{\co\over \tau_{0}\beta^{2}}=\frac{a_{1}}{8}\frac{16a_{2}^{2}}{a_{1}^{2}\tau_{0}}=\frac{2a_{2}^{2}}{a_{1}\tau_{0}}<\frac{2a_{2}^{2}}{18.837Ba_{1}}<\frac{2a_{2}^{2}}{18.837\left[a_{2}^{2}/(2a_{1})\right]a_{1}}<\frac{1}{4.7}
\]
and
%beg
\begin{equation}
\frac{c_{2}}{8}-1<\frac{4\sqrt{2}}{B\sqrt{18.837}}<\frac{8\sqrt{2}a_{1}}{\sqrt{18.837}a_{2}^{2}}<2.6068\frac{a_{1}}{a_{2}^{2}}.
\label{eq-tau0}
\end{equation}
%end
Then (\ref{eq-boundT}) becomes
%beg
\begin{align*}
\gT(\bsX,\overline \gp,\gq)\le&\;\E\cro{\gT(\bsX,\overline \gp,\gq)}+a'\cro{\gh^{2}(\gP,\overline \gP)+\gh^{2}(\gP,\gQ)}\\&\;+\kappa\left[\frac{D_{n}(m)}{4.7}+\Delta(m)\right]+G(\gP,\overline\gP)+\xi'+c_{2}\gh^{2}(\gP,\overline \gP).
\end{align*}
%end
Since the last inequality is true for all $\sbQ_{m}$ containing $\gQ$, we derive from (\ref{eq-L2}) that
%beg
\begin{align*}
\gT(\bsX,\overline \gp,\gq)\le&\;\E\cro{\gT(\bsX,\overline \gp,\gq)}+a'\cro{\gh^{2}(\gP,\overline \gP)+\gh^{2}(\gP,\gQ)}\\&\;+G(\gP,\overline \gP)+\gpen(\gQ)+\xi'+c_{2}\gh^{2}(\gP,\overline \gP),
\end{align*}
%end
which, together with~\eref{eq-ET}, leads to the following inequality which holds on $\Omega_{\xi}$ for all $\gQ\in \sbQ$:
\[
\gT(\bsX,\overline \gp,\gq)\le A\gh^{2}(\gP,\overline \gP)-a'\gh^{2}(\gP,\gQ)+G(\gP,\overline \gP)+\gpen(\gQ)+\xi'
\]
with $A=a_{0}+a'+c_{2}$. We deduce from this inequality that, on $\Omega_{\xi}$, for any (random) element $\widehat \gP\in\sbQ$,
%beg
\begin{equation}
\gT(\bsX,\overline \gp,\widehat \gp)\le A\gh^{2}(\gP,\overline \gP)-a'\gh^{2}(\gP,\widehat \gP)+G(\gP,\overline \gP)+\gpen(\widehat \gP)+\xi'
\label{eq-pp}
\end{equation}
%end
and that
%beg
\begin{align}
\gup(\bsX,\overline \gP)&=\sup_{\gQ\in\sbQ}\cro{\gT(\bsX,\overline \gp,\gq)-\gpen(\gQ)}+\gpen(\overline\gP)\nonumber\\&\le A\gh^{2}(\gP,\overline \gP)
-a'\gh^{2}(\gP,\sbQ)+G(\gP,\overline \gP)+\gpen(\overline{\gP})+\xi'.
\label{maj-gup}
\end{align}
%end
Since $\gT(\bsX,\widehat \gp,\overline \gp)=-\gT(\bsX,\overline \gp,\widehat \gp,)$, (\ref{eq-pp}) leads to
%beg
\begin{align}
a'\gh^{2}(\gP,\widehat \gP)\le&\; A\gh^{2}(\gP,\overline \gP)-\gT(\bsX,\overline \gp,\widehat \gp)+G(\gP,\overline \gP)+\gpen(\widehat \gP)+\xi'\nonumber\\
=&\; A\gh^{2}(\gP,\overline \gP)+\cro{\gT(\bsX,\widehat \gp,\overline \gp)-\gpen(\overline\gP)}+\gpen(\widehat \gP)\nonumber \\
&\;+G(\gP,\overline \gP)+\gpen(\overline\gP)+\xi'\nonumber\\
\le&\; A\gh^{2}(\gP,\overline \gP)+\gup(\bsX,\widehat \gP)+G(\gP,\overline \gP)+\gpen(\overline\gP)+\xi'. \label{maj-risk0}
\end{align}
%end
If $\widehat \gP$ belongs to the set $\sbE(\psi,\bsX)$, then $\gup(\bsX,\widehat \gP)<\gup(\bsX,\overline \gP)+(\kappa/25)$ and, by~\eref{maj-gup},
\[
\gup(\bsX,\widehat \gP)<A\gh^{2}(\gP,\overline \gP)-a'\gh^{2}(\gP,\sbQ)
+G(\gP,\overline \gP)+\gpen(\overline\gP)+\xi'+(\kappa/25),
\]
which, together with~\eref{maj-risk0}, shows that on the set $\Omega_{\xi}$ and for all $\widehat\gP\in\sbE(\psi,\bsX)$,
\begin{align*}
\lefteqn{a'\gh^{2}(\gP,\widehat \gP)}\hspace{2mm}\\
&\le 2A\gh^{2}(\gP,\overline \gP)-a'\gh^{2}(\gP,\sbQ)
+2\left[G(\gP,\overline \gP)+\gpen(\overline\gP)+\kappa(\xi+\vartheta+0.02)\right].
\end{align*}
This inequality, which extends straightforwardly to any element $\widehat \gP$ belonging to ${\rm Cl}\!\left(\sbE(\psi,\bsX)\right)$, 
writes
\begin{align*}
&\gh^{2}(\gP,\widehat \gP)\\
&\;\;\,\le \frac{4A}{a_{1}}\gh^{2}(\gP,\overline \gP)-\gh^{2}(\gP,\sbQ)
+\frac{4}{a_{1}}\left[G(\gP,\overline \gP)+\gpen(\overline\gP)+\kappa(\xi+\vartheta+0.02)\right]
\end{align*}
and the conclusion follows since, by (\ref{eq-tau0}), $A<a_{0}+(a_{1}/2)+8+21(a_{1}/a_{2}^{2})$.

%SUBSECTION
\subsection{Proof of Theorem~\ref{thm-main1}}\label{sect-main1}
Let us take $\cM=\{0\}$, $\sbQ_{0}=\sbQ$, $D_{n}(0)=\Delta(0)=0$ so that the choice $\gpen(\gQ)=0$ for all $\gQ\in\sbQ$ satisfies \eref{eq-L2}. It follows from Proposition~\ref{prop-gw}, more precisely from \eref{eq-w/D}, that \eref{eq-L1} is satisfied with $G(\gP,\overline\gP)=(\kappa/4.7)D^{\overline \sbQ}(\gP,\overline \gP)$. Consequently, Theorem~\ref{thm-main} applies and  (\ref{eq-L0}) leads to (\ref{eq-risk1}). The lower bound for $\kappa/25$ in (\ref{cond-kappa}) follows from our assumption that $a_{2}^{2}\ge 6a_{1}$.

%SUBSECTION
\subsection{Proof of Theorem~\ref{thm-main4}}\label{sect-main4}
Replacing $\gpen(\gQ)$ by $\gpen(\gQ)+a$ for all $\gQ\in{\sbQ}$ and $a\in\R$ does not change the value of the function $\gup$, hence the set $\sbE(\psi,\bsX)$ and the family of $\rho$-estimators. We may therefore assume that $\kappa_{1}=0$. 

Let $m'\in\cM$ and $\overline \gP\in\sbQ_{m'}$. Within the framework of Theorem~\ref{thm-main}, let us take $G(\gP,\overline \gP)=\kappa K D_{n}(m')/4.7$ for all $(\gP,\overline \gP)\in\sbP\times \sbQ$ so that, under \eref{eq-L4}, \eref{eq-L1} is satisfied by Proposition~\ref{prop-gw}. Moreover (\ref{eq-L3}) with $\kappa_{1}=0$ implies that (\ref{eq-L2}) also holds and Theorem~\ref{thm-main} therefore applies, leading in this case to
\[
\gh^{2}(\gP,\widehat \gP)\le\gamma\gh^{2}(\gP,\overline \gP)+{4\kappa\over a_{1}}\cro{{K+1\over 4.7}D_{n}(m')+\Delta(m')+1.49+\xi}
\]
with probability $1-e^{-\xi}$ whatever $\xi>0$. The conclusion follows since the choices of $m'\in\cM$ and $\overline \gP\in \sbQ_{m'}$ are arbitrary.

\subsection{Proof of Proposition~\ref{prop-dev}}
\label{sect-pprop-dev}

%\begin{supplement}
%  \sname{Supplement A}\label{PR1}
%  \stitle{Proofs of Propositions~\ref{prop-dev} and Theorem~\ref{thm-Exten}}
%  \slink[doi]{??}
%\sdescription{
%\subsection{Proof of Proposition~\ref{prop-dev}}\label{sect-pprop-dev}
{The proof of Proposition~\ref{prop-dev} relies on the following result --- see Proposition~45 of Baraud {\em et al.}~\citeyearpar{MR3595933} --- which presents an extension of a version of Talagrand's Theorem on the suprema of empirical processes that is proved in Massart~\citeyearpar{MR2319879}. 
\begin{prop}\label{talagrand}
Let $T$ be some finite or countable set, $U_{1},\ldots,U_{n}$ be independent centered random vectors with values in $\R^{T}$ and let $Z=\sup_{t\in T}\ab{\sum_{i=1}^{n}U_{i,t}}$. If for some positive numbers $b$ and $v$, 
\[
\max_{i=1,\ldots,n}\ab{U_{i,t}}\le b\qquad\mbox{and}\qquad
\sum_{i=1}^{n}\E\left[U^2_{i,t}\right]\le v^{2}\ \quad\mbox{for all }t\in T,
\] 
then, for all positive numbers $c$ and $x$,
\begin{equation}
\P\left[Z\le(1+c)\E(Z)+(8b)^{-1}cv^2+2\left(1+8c^{-1}\right)bx\right]\ge1-e^{-x}.
\label{massart3}
\end{equation}
\end{prop}
Let $\xi>0$, $\alpha>\delta>1$, $\overline{\gP}\in\sbQ$ and $m$ in $\cM$ be fixed. For $j\in\N$ we define
\[
x_0(m)=\left(\st\Delta(m)+\vartheta+\xi\right)\bigvee\left(\tau^{-1}\co \widetilde{D}_{m}(\gP,\overline{\gP})\right),
\]
\[
x_j(m)=\delta^jx_0(m),\qquad y_j^2=\co^{-1}\tau x_j(m),
\]
\[
B^{\sbQ_{m}}_{j}(\gP,\overline{\gP})=\ac{\gQ\in\sbQ_{m}\mbox{ such that }
y_{j}^{2}\le\gh^{2}(\gP,\gQ)+\gh^{2}(\gP,\overline{\gP})<y_{j+1}^{2}}
\]
and
\[
Z_{j}^{\sbQ_{m}}(\bsX,\overline{\gp})=\sup_{\gQ\in B^{\sbQ_{m}}_{j}(\gP,\overline{\gP})}
\ab{\gZ(\bsX,\overline{\gp},\gq)}.
\]
Let us drop, for a while, the dependency of the quantities $x_{j}(m)$ with respect to $m$ for $j\ge 0$. Since $B^{\sbQ_{m}}_{j}(\gP,\overline{\gP})\subset\sbB^{\sbQ_{m}}(\gP,\overline{\gP},y_{j+1})$ and $y_{j+1}^{2}>y_{0}^{2}\ge\widetilde{D}_{m}(\gP,\overline{\gP})$, it follows from (\ref{eq-w1}) and (\ref{eq-dev}) that
%beg
\begin{equation}
\E\left[Z_{j}^{\sbQ_{m}}(\bsX,\overline{\gp})\right]\le\co y_{j+1}^{2}+K'\gh^{2}\left(\gP,\overline{\gP}\right).
\label{eq-Zj}
\end{equation}
%end
For each $j\ge 0$, we may apply Proposition~\ref{talagrand} to the supremum 
$Z_{j}^{\sbQ_{m}}(\bsX,\overline{\gp})$ by taking  $T=B^{\sbQ_{m}}_{j}(\gP,\overline{\gP})$ (which is 
countable as a subset of $\sbQ_{m}$) and 
\begin{equation}\label{def-U}
U_{i,\gq}=\psi\pa{\sqrt{q_{i}\over \overline p_{i}}(X_{i})}-
\E\cro{\psi\pa{\sqrt{q_{i}\over \overline p_{i}}(X_{i})}}\quad\mbox{for all }i=1,\ldots,n.
\end{equation}
For such a choice, the assumptions of Proposition~\ref{talagrand} are met with $b=2$ (since $\psi$ is bounded by $1$) and $v^{2}=a_{2}^{2}y_{j+1}^{2}$ (by (\ref{eq-var}) and the definition of $B_{j}^{\sbQ_{m}}(\gP,\overline{\gP})$). It therefore follows from \eref{massart3} that, for all $c>0$, with probability at least $1-e^{-x_{j}}$ and for all $\gQ\in B^{\sbQ_{m}}_{j}(\gP,\overline{\gP})$,
%beg
\begin{align}
|\gZ(\bsX,\overline{\gp},\gq)|\nonumber
&\le Z_{j}^{\sbQ_{m}}(\bsX,\overline{\gp})\\&\le(1+c)\E\left[Z_{j}^{\sbQ_{m}}(\bsX,\overline{\gp})\right]+(ca_{2}^{2}y_{j+1}^{2}/16)+4\left(1+8c^{-1}\right)x_{j}.\label{Eq-pr1}
\end{align}
%end
Since $y_{j+1}^{2}=\delta y_{j}^{2}$ and $x_j=\co y_{j}^{2}/\tau$, (\ref{Eq-pr1}) becomes, using (\ref{eq-Zj}),
%beg
\begin{align*}
\lefteqn{|\gZ(\bsX,\overline{\gp},\gq)|}\hspace{15mm}\\&\le y_{j+1}^{2}
\left[\co(1+c)+\frac{ca_{2}^{2}}{16}\right]+4\left(1+8c^{-1}\right)x_j+K'(1+c)\gh^{2}\left(\gP,\overline{\gP}\right)\\&=\co  y_{j}^{2}\left[\delta cB+\frac{4\left(1+8c^{-1}\right)}{\tau}+\delta\right]+K'(1+c)\gh^{2}\left(\gP,\overline{\gP}\right).
\end{align*}
%end
Minimizing the bracketed term with respect to $c$ leads to $c=4\sqrt{2/(\delta B\tau)}$ 
and
\[
|\gZ(\bsX,\overline{\gp},\gq)|\le\co y_{j}^{2}\left[\frac{4}{\tau}+8\sqrt{\frac{2\delta B}{\tau}}+\delta\right]+K'\left(1+\frac{4\sqrt{2}}{\sqrt{\delta B\tau}}\right)\gh^{2}\left(\gP,\overline{\gP}\right).
\]
Since $\gh^{2}(\gP,\gQ)+\gh^{2}(\gP,\overline{\gP})>y_j^{2}$ on $B^{\sbQ_{m}}_{j}(\gP,\overline{\gP})$, we derive that
%beg
\begin{align*}
\lefteqn{|\gZ(\bsX,\overline{\gp},\gq)|-\co\alpha\cro{\gh^{2}(\gP,\gQ)+\gh^{2}(\gP,\overline{\gP})}
-K'\left(1+\frac{4\sqrt{2}}{\sqrt{\delta B\tau}}\right)\gh^{2}\left(\gP,\overline{\gP}\right)}\hspace{70mm}\\
&<\co y_{j}^{2}\left[\frac{4}{\tau}+8\sqrt{\frac{2\delta B}{\tau}}-(\alpha-\delta)\right]
\end{align*}
%end
and the bracketed factor is non-positive provided that
\[
\frac{1}{\tau}\le2\delta B\left[\sqrt{1+\frac{\alpha-\delta}{8\delta B}}-1\right]^2,
\]
which is our condition (\ref{Eq-tau}). Finally, since $\delta>1$, with probability at least $1-e^{-x_{j}}$ and for all 
$\gQ\in B^{\sbQ_{m}}_{j}(\gP,\overline{\gP})$,
\[
|\gZ(\bsX,\overline{\gp},\gq)|-\co\alpha\cro{\gh^{2}(\gP,\gQ)+\gh^{2}(\gP,\overline{\gP})}
-K'\left(1+\frac{4\sqrt{2}}{\sqrt{B\tau}}\right)\gh^{2}\left(\gP,\overline{\gP}\right)<0.
\]

Let us now define
\[
Z^{\sbQ_{m}}(\bsX,\overline\gp)=\sup_{\gQ\in\sbB^{\sbQ_{m}}(\gP,\overline\gP,y_{0})}\ab{\gZ(\bsX,\overline\gp,\gq)}
\]
and apply Proposition~\ref{talagrand} in a similar way to $Z^{\sbQ_{m}}(\bsX,\overline{\gp})$ with $x=x_0$. We then deduce analogously that, for all $c>0$, with probability at least $1-e^{-x_0}$ and for all $\gQ\in \sbB^{\sbQ_{m}}(\gP,\overline{\gP},y_{0})$, 
\begin{align*}
|\gZ(\bsX,\overline{\gp},\gq)|&\le Z^{\sbQ_{m}}(\bsX,\overline{\gp})\\
&\le \co  y_{0}^{2}\left[cB+\frac{4\left(1+8c^{-1}\right)}{\tau}+1\right]+K'(1+c)\gh^{2}\left(\gP,\overline{\gP}\right).
\end{align*}
With $c=4\sqrt{2/(B\tau)}$, we get, for all $\gQ\in\sbB^{\sbQ_{m}}(\gP,\overline\gP,y_{0})$ and with probability at least $1-e^{-x_0}$,
%beg
\begin{align*}
|\gZ(\bsX,\overline{\gp},\gq)|&\le\co y_{0}^{2}\left(\frac{4}{\tau}+8\sqrt{\frac{2B}{\tau}}+1\right)
+K'\left(1+\frac{4\sqrt{2}}{\sqrt{B\tau}}\right)\gh^{2}\left(\gP,\overline{\gP}\right)\\&=\left(\frac{4}{\tau}+8\sqrt{\frac{2B}{\tau}}+1\right)\tau x_0+
K'\left(1+\frac{4\sqrt{2}}{\sqrt{B\tau}}\right)\gh^{2}\left(\gP,\overline{\gP}\right).
\end{align*}
%end
Combining all these bounds and reintroducing the dependency of the $x_{j}$ with respect to $m$, we derive that, for all $\gQ\in\sbQ_{m}$ simultaneously,
%beg
\begin{align*}
\lefteqn{|\gZ(\bsX,\overline{\gp},\gq)|-\co\alpha\cro{\gh^{2}(\gP,\gQ)+\gh^{2}(\gP,\overline{\gP})}}
\hspace{25mm}\\&\le\left(\frac{4}{\tau}+8\sqrt{\frac{2B}{\tau}}+1\right)\tau x_0(m)
+K'\left(1+\frac{4\sqrt{2}}{\sqrt{B\tau}}\right)\gh^{2}\left(\gP,\overline{\gP}\right),
\end{align*}
%end
with probability at least
\[
1-\eta_{m}\quad\mbox{with}\quad\eta_{m}=2\exp\left[-x_0(m)\right]+\sum_{j\ge1}\exp\left[-x_j(m)\right].
\]
In order to bound $\eta_{m}$, we observe that $x_{j}(m)\ge\Delta(m)+\vartheta\delta^{j}+\xi$ for all $j\in\N$, hence by (\ref{eq-sum1}),
\[
\eta_{m}\le\exp\left[-\xi-\Delta(m)\right]\pa{2\exp[-\vartheta]+\sum_{j\ge1}\exp\left[-\vartheta\delta^{j}\right]}\le\exp\left[-\xi-\Delta(m)\right].
\]
The result finally extends to all $\gQ\in\sbQ$ by summing these bounds over $m\in\cM$ and using~(\ref{eq-delta}). 

%SUBSECTION
\subsection{Proof of Theorem~\ref{thm-Exten}}\label{sect-pthm-Exten}
We shall denote by $\sbQ_{m}$ the $\rho$-models $\{\gq\cdot \gmu,;\ \gq\in\cbQ_{m}\}$ for all $m\in\cM$. Let $\overline \gup(\bsX,\gq)$, $\overline \sbE(\psi,\bsX)$ be the function of $\gq$ and the set defined by \eref{def-gup} and \eref{def-sE} respectively, with $\cbQ$ replaced by $\overline{\cbQ}$ and $\gpen$ by $\overline \gpen$ in \eref{def-gup} and \eref{def-sE}.

Since $\widehat{\gP}$ belongs, by definition, to ${\rm Cl}\!\left(\st\overline \sbE(\psi,\bsX)\right)$, it is enough to show that $\overline \sbE(\psi,\bsX)\subset{\rm Cl}\!\left(\st\sbE(\psi,\bsX)\right)$. The proof will be divided into several steps.

{\bf Step 1.} It follows from Assumption~\ref{Hypo-pen} that for all $\gQ	\in\overline \sbQ$ there exists some $m_{\gQ}\in\cM$ such that $\overline \sbQ_{m_{\gQ}}\ni \gQ$ and $\overline\gpen(\gQ)=p(m_{\gQ})$. Since $\sbQ_{m_{\gQ}}$ is a $\rho$-model with representation $(\gmu,\cbQ_{m_{\gQ}})$ and $\cbQ_{m_{\gQ}}$ is $\sbT$-dense in $\overline{\cbQ}_{m_{\gQ}}$, we can find a sequence $\left(\gq^{(k)}\right)_{k\ge1}$ in $\cbQ_{m_{\gQ}}\subset \overline{\cbQ}_{m_{\gQ}}$ which converges to $\gq$ with respect to the topology $\sbT$ and, by Assumption~\ref{Hypo-pen} again, $\overline \gpen\left(\gQ^{(k)}\right)\le p(m_{\gQ})=\overline\gpen(\gQ)$ for all $k\ge1$.\vspace{2mm}

{\bf Step 2.} Let us now show that $\overline \gup(\bsX,\gq)=\gup(\bsX,\gq)$ for all $\gq\in\cbQ$. Since clearly $\overline \gup(\bsX,\gq)\ge\gup(\bsX,\gq)$, it is enough to show that $\gup(\bsX,\gq)\ge \overline \gup(\bsX,\gq)$. Setting 
\[
\left\{
\begin{array}{lll}  
\Lambda(\bsX,\gq)&=&\sup_{\gr\in\cbQ}\cro{\gT(\bsX,\gq,\gr)-\overline\gpen(\gR)},\\
\overline\Lambda(\bsX,\gq)&=&\sup_{\gr\in\overline{\cbQ}}\cro{\gT(\bsX,\gq,\gr)-\overline\gpen(\gR)},
\end{array}\right.
\]
it suffices to show that $\Lambda(\bsX,\gq)\ge \overline \Lambda(\bsX,\gq)$ for all $\gq\in\cbQ$. Let us fix $\gq\in\cbQ$, $\eps>0$ and some $\gr\in\overline{\cbQ}$ such that $\gT(\bsX,\gq,\gr)-\gpen(\gR)\ge \overline \Lambda(\bsX,\gq)-\varepsilon$. By Step~1, $\gR$ belongs to $\overline{\sbQ}_{m_{\gR}}$ and one can find a sequence $\left(\gr^{(k)}\right)_{k\ge1}$ in $\cbQ_{m_{\gR}}$ which converges to $\gr$ with respect to the topology $\sbT$ so that $r^{(k)}_{i}(X_{i})\dfleche{k\rightarrow+\infty}r_{i}(X_{i})$ for $1\le i\le n$. Moreover, by Step~1 again, $\overline \gpen(\gR^{(k)})\le \overline\gpen(\gR)$.\vspace{2mm}\\
--- If $q_{i}(X_{i})>0$, then, by continuity,
\[
\psi\pa{\sqrt{{r^{(k)}_{i}(X_{i})}\over {q_{i}(X_{i})}}}\dfleche{k\rightarrow+\infty}\psi\pa{\sqrt{{r_{i}(X_{i})}\over {q_{i}(X_{i})}}}.
\]
--- If $q_{i}(X_{i})=0$ and $r_{i}(X_{i})>0$, then, for $k$ large enough, $r^{(k)}_{i}(X_{i})>0$ and
\[
\psi\pa{\sqrt{{r^{(k)}_{i}(X_{i})}\over {q_{i}(X_{i})}}}=\psi(+\infty)=1=\psi\pa{\sqrt{{r_{i}(X_{i})}\over {q_{i}(X_{i})}}}.
\]
--- Finally, if $q_{i}(X_{i})=0=r_{i}(X_{i})$, then 
\[
\psi\pa{\sqrt{{r_{i}(X_{i})}\over {q_{i}(X_{i})}}}=\psi(1)=0
\]
while 
\[
\psi\pa{\sqrt{{r^{(k)}_{i}(X_{i})}\over {q_{i}(X_{i})}}}=\left\{
\begin{array}{ll}  
0&\;\mbox{ if }r^{(k)}_{i}(X_{i})=0,\\1&\;\mbox{ if }r^{(k)}_{i}(X_{i})>0,
\end{array}\right.
\]
so that in both cases it is not smaller than $\psi\pa{\sqrt{r_{i}(X_{i})/q_{i}(X_{i})}}$. It follows that, for $k$ large enough, 
%beg
\begin{align*}
\Lambda(\bsX,\gq)&\ge\gT\left(\bsX,\gq,\gr^{(k)}\right)-\overline\gpen\left(\gR^{(k)}\right)\ge\gT(\bsX,\gq,\gr)-\varepsilon-\overline\gpen\left(\gR^{(k)}\right)\\
&\ge\gT(\bsX,\gq,\gr)-\varepsilon-\overline\gpen(\gR)\ge \overline \Lambda(\bsX,\gq)-2\varepsilon,
\end{align*}
%end
hence $\Lambda(\bsX,\gq)\ge \overline \Lambda(\bsX,\gq)$ as expected since $\varepsilon$ is arbitrary.\vspace{2mm}

{\bf Step 3.} Let $\gq$ belong to $\overline{\cbQ}$ and $\left(\gq^{(k)}\right)_{k\ge 1}$ be a sequence in $\cbQ$ which converges to $\gq$ in the topology $\sbT$. We want to show that, given $\varepsilon>0$,
%beg
\begin{equation}
\sup_{\gr\in \overline\cbQ}\left[\gT\left(\bsX,\gq^{(k)},\gr\right)-\gT(\bsX,\gq,\gr)\right]\le\eps
\quad\mbox{for $k$ large enough}.
\label{eq-Z1}
\end{equation}
%end
It is actually enough to show that, for all $i\in\{1,\ldots,n\}$ and $k$ large enough,
%beg
\begin{equation}
\sup_{\gr\in \overline\cbQ}\left[\psi\pa{\sqrt{r_{i}\over q^{(k)}_{i}}(X_{i})}-\psi\pa{\sqrt{r_{i}\over q_{i}}(X_{i})}\right]\le\frac{\eps}{n}.
\label{eq-Z2}
\end{equation}
%end
Let us fix $i$ and distinguish between two situations.\vspace{2mm}\\
--- {\bf Case 1:} $q_{i}(X_{i})=0$. If $r_{i}(X_{i})>0$, then $\psi\pa{\sqrt{r_{i}/q_{i}}(X_{i})}=1$ so that the result is trivially true since $\psi$ is bounded by one. If $r_{i}(X_{i})=0$ then $\psi\pa{\sqrt{r_{i}/q_{i}}(X_{i})}=0$. If $q^{(k)}_{i}(X_{i})=0$, $\psi\pa{\sqrt{r_{i}/q^{(k)}_{i}}(X_{i})}=0$ and  the inequality is true while if $q^{(k)}_{i}(X_{i})>0$, $\psi\pa{\sqrt{r_{i}/q^{(k)}_{i}}(X_{i})}=\psi(0)=-1$ and the inequality is also true.\vspace{2mm} \\
---  {\bf Case 2:} Let $I\subset\{1,\ldots,n\}$ be the set of indices such $q_{i}(X_{i})>0$. Since
$q^{(k)}_{i}(X_{i})\dfleche{k\rightarrow+\infty}q_{i}(X_{i})$ for all $i\in I$, one can find $a>0$ such that, for $k$ large enough
\[
\inf_{i\in I}\min\left\{q_{i}(X_{i}),q^{(k)}_{i}(X_{i})\right\}\ge a.
\]
We may now apply the following lemma with $\eta=\varepsilon/n$ and
\[
\sup_{i\in I}\left|q_{i}(X_{i})-q^{(k)}_{i}(X_{i})\right|\le\delta,
\] 
which again holds for $k$ large enough, to show that (\ref{eq-Z2}) and therefore (\ref{eq-Z1}) holds.
%
% LEMMA
\begin{lem}\label{lem-cont2}
For any $a>0$, $\eta>0$, there exists $\delta=\delta(a,\eta,\psi)>0$, depending only on $a$, $\eta$ and the Lipschitz function $\psi$, such that for all $x, x'\ge a$ and $y\ge0$,
\[
\psi\left(\sqrt{y/x}\right)\le\psi\left(\sqrt{y/x'}\right)+\eta\qquad\mbox{if}\quad|x'-x|\le\delta.
\]
\end{lem}
%
% PROOF
\begin{proof}
There is nothing to prove if $x\ge x'$ by monotonicity of $\psi$. Let us therefore assume that $a\le x<x'$. If $y/x'\ge A>0$, then 
\[
1\ge\psi\left(\sqrt{y/x}\right)\ge\psi\left(\sqrt{y/x'}\right)\ge\psi\left(\sqrt{A}\right)
\]
so that the result holds provided that $\psi\left(\sqrt{A}\right)\ge1-\eta$. Let us fix such a value of $A$ and consider the case of $y/x'<A$. Since $\psi$ is Lipschitz with some Lipschitz constant $L$, then
\begin{align*}
\psi\left(\sqrt{\frac{y}{x}}\right)-\psi\left(\sqrt{\frac{y}{x'}}\right)&\le L\sqrt{y}\left[\frac{1}{\sqrt{x}}-\frac{1}{\sqrt{x'}}\right]=L\sqrt{\frac{y}{x'}}\frac{x'-x}{\sqrt{x}\left(\sqrt{x}+\sqrt{x'}\right)}\\
&< \frac{L\sqrt{A}}{2a}(x'-x)
\end{align*}
and the result holds as soon as $\delta\le2a\eta/\left(L\sqrt{A}\right)$.
\end{proof}

{\bf Step 4}.
Let $\gQ\in \overline\sbE(\psi,\bsX)$ and $\varepsilon$ satisfy
\[
0<\varepsilon<\inf_{\gq'\in\overline{\cbQ}}\overline \gup(\bsX,\gq')+{\kappa\over 25}
-\overline \gup(\bsX,\gq).
\]
Then
%beg
\begin{equation}
\overline \gup(\bsX,\gq)=\overline \Lambda(\bsX,\gq)+\overline \gpen(\gQ)<\inf_{\gq'\in\overline{\cbQ}}\overline \gup(\bsX,\gq')+{\kappa\over 25}-\varepsilon.
\label{eq-Z3}
\end{equation}
%end
In such a case $\gQ\in\overline{\sbQ}_{m_{\gQ}}$ and there exists a sequence $\left(\gq^{(k)}\right)_{k\ge 1}$ in $\cbQ_{m_{\gQ}}$ which converges to $\gq$ in the topology $\sbT$ with $\overline \gpen\left(\gQ^{(k)}\right)\le \overline \gpen(\gQ)$ by Step~1. Since 
\[
\overline\Lambda\left(\bsX,\gq^{(k)}\right)-\overline\Lambda(\bsX,\gq)\le \sup_{\gr\in \overline \cbQ}\left[\gT\left(\bsX,\gq^{(k)},\gr\right)-\gT(\bsX,\gq,\gr)\right],
\]
it follows from 
(\ref{eq-Z1}) that, for $k$ large enough, $\overline \Lambda\left(\bsX,\gq^{(k)}\right)-\overline \Lambda(\bsX,\gq)\le \eps$ and, from Step~2, since $\gq^{(k)}\in\cbQ$, that
\begin{align*}
\gup\left(\bsX,\gq^{(k)}\right)=\overline \gup\left(\bsX,\gq^{(k)}\right)&=\overline \Lambda\left(\bsX,\gq^{(k)}\right)+\overline \gpen\left(\gQ^{(k)}\right)\\
&\le \overline\Lambda(\bsX,\gq)+\eps+\overline \gpen(\gQ).
\end{align*}
This, together with (\ref{eq-Z3}) and Step~2 shows that
\[
\gup\left(\bsX,\gq^{(k)}\right)\le \overline \gup(\bsX,\gq)+\eps<\inf_{\gq'\in\cbQ} \overline \gup(\bsX,\gq')+{\kappa\over 25}=\inf_{\gq'\in\cbQ}\gup(\bsX,\gq')+{\kappa\over 25}.
\]
This means that $\gQ^{(k)}\in\sbE(\psi,\bsX)$ for $k$ large enough. Since the sequence 
$\left(\gq^{(k)}\right)_{k\ge 1}$ converges to $\gq$ in the topology $\sbT$ and therefore, as already mentioned in Section~\ref{sect-UnivSep}, in the topology $\sbV$. This shows that $\gQ$ belongs to the closure of $\sbE(\psi,\bsX)$ as required and concludes our proof.
}
%\end{supplement}

%\end{document}

%%%%%%%%%%%%%%%%%%%%%%%%%%%%%%%%%%%%%%%%%%%%%
%\begin{supplement}
%  \sname{Supplement B}\label{sect-pprop-expsi}
%  \stitle{Proof of Proposition~\ref{prop-expsi}}
%  \slink[doi]{??}
%\sdescription

%
%\end{supplement}
%\begin{supplement}
%  \sname{Supplement C}\label{PR}
%  \stitle{Other proofs}
%  \slink[doi]{??}
%  \sdescription

\section{Universal separability}\label{Sect-UnivSep}
\paragraph{{\bf Histograms}}
$\!\!\!$Let $D$ be a positive integer and $\overline \cH_{D}$ the set of right-continuous histograms on $\R$ with at most $D$ pieces, that is the set of functions $f=\sum_{j=1}^{D}\alpha_{j}\1_{[a_{j},a_{j+1})}$ with
\[
\alpha_{j}\ge 0\;\mbox{ for }1\le j\le D,\quad a_{1}<a_{2}\ldots<a_{D+1}\quad\mbox{and}\quad\sum_{j=1}^{D}\alpha_{j}(a_{j+1}-a_{j})=1.
\]
Let $\cH_{D}$ be the subset of $\overline \cH_{D}$ consisting of those elements $f$ for which the $\alpha_{j}$ and the $a_{j}$ are rational numbers and $\alpha_{j}>0$ for all $j$. Clearly, $\cH_{D}$ is countable. Let us show that $\cH_{D}$ is $\sbT$-dense in $\overline \cH_{D}$. Given $f\in\overline \cH_{D}$, let us consider sequences of rational numbers $(\alpha_{j,n})_{n\ge 1}$, $(a_{j,n})_{n\ge 1}$ for $1\le j\le D$ and $(a_{D+1,n})_{n\ge 1}$ with the following properties: $(\alpha_{j,n})_{n\ge 1}$ is non-increasing and $(a_{j,n})_{n\ge 1}$ is non-decreasing, $a_{D+1,n}\dfleche{n\to+\infty}a_{D+1}$ and, for all $j\in\{1,\ldots,D\}$, $\alpha_{j,n}\dfleche{n\to+\infty}\alpha_{j}$ and $a_{j,n}\dfleche{n\to+\infty}a_{j}$. Then the sequence $(f_{n})_{n\ge 1}$ of histograms $f_{n}$ given by
\[
f_{n}=\sum_{j=1}^{D}{\alpha_{j,n}\over \sum_{j=1}^{D}\alpha_{j,n}(a_{j+1,n}-a_{j,n})}\,\1_{[a_{j,n},a_{j+1,n})}\in\cH_{D}
\]
converges pointwise toward $f$. The set $\overline \cH_{D}$ is therefore universally separable.
%
%PARAGRAPH
\paragraph{{\bf Non-decreasing densities}}
Let $\overline\cH_{\downarrow}$ be the set of non-decreasing and right-continuous densities on $I=(0,+\infty)$ and $(A_{n})_{n\ge 1}$ be an increasing (for the inclusion) sequence of finite 
subsets of $(I\cap \Q)^{2}$ such that $\cup_{n\ge 1}A_{n}=(I\cap \Q)^{2}$. We define, for $C$ a finite non empty subset 
\pagebreak
of $(I\cap \Q)^{2}$, $g_{C}=\sup_{(q,r)\in C}r\1_{(0,q)}$ and set, for $n\ge1$,
\[
\cG_{n}=\ac{g_{C}\;\mbox{ for }C\subset A_{n},\ C\ne\varnothing}\qquad \mbox{and}\qquad \cG=\cup_{n\ge 1}\cG_{n}.
\]
Any element $g_{C}$ of $\cG_{n}$ is not identically equal to 0 since $qr>0$; it is also non-increasing, right-continuous and integrable on $I$ (with respect to the Lebesgue measure) as the supremum of a finite number of non-increasing, right-continuous and integrable functions on $I$. This implies that $\int_{I}g_{C}(x)dx$ belongs to $(0,+\infty)$. Moreover $\cG$ is countable since each $\cG_{n}$ is finite. We may therefore define 
\[
\cH_{\downarrow}=\ac{f={g\over \int_{I}g(x)dx },\; g\in \cG}
\]
and, for $f\in\overline\cH_{\downarrow}$ and $n\ge 1$,
\[
C_{n}^{f}=\{(q,r)\in A_{n},\; r\le f(q)\}.
\]

Then $\cH_{\downarrow}$ is a countable subset of $\overline\cH_{\downarrow}$. It follows from the properties of the sets $A_{n}$ that the sequence $\left(C_{n}^{f}\right)_{n\ge 1}$ is non-decreasing for the inclusion, that $C_{n}^{f}\ne\varnothing$ for $n\ge N$ if $N$ is large enough and that
\[
\bigcup_{n\ge N}C_{n}^{f}=C^{f}=\{(q,r)\in (I\cap \Q)^{2},\; r\le f(q)\}.
\]
Let us consider the sequence $(g_{n})_{n\ge N}$ with $g_{n}=g_{C_{n}^{f}}\in\cG_{n}$. Since the function $f$ is non-increasing and right-continuous, for all $x\in I$, 
\[
f(x)=\sup_{q\in I\cap \Q}f(q)\1_{(0,q)}(x)=\sup_{(q,r)\in C^{f}}r\1_{(0,q)}(x)=\lim_{n\to +\infty}g_{n}(x).
\]
Furthermore, the sequence $(g_{n})_{n\ge N}$ is non-negative and non-decreasing so that by monotone convergence $\int_{I}g_{n}(x)dx\to \int_{I}f(x)dx=1$. Therefore the sequence $(f_{n})_{n\ge N}$, with $f_{n}=g_{n}/[\int_{I}g_{n}(x)dx]$, belongs to $\cH_{\downarrow}$ for all $n\ge N$ and
\[
f(x)=\lim_{n\to +\infty}f_{n}(x)\quad \mbox{for all }x\in I.
\]
This implies that $\cH_{\downarrow}$ is $\sbT$-dense in $\overline\cH_{\downarrow}$ which is therefore universally separable. 
%}
%\end{supplement}

%%%%
%%%%SUPPLEMENT
%SECTION
\section{VC-Dimensions}\label{sect-ppp-dimexp}

%SUBSECTION
\subsection{Proof of Proposition~\ref{prop-CPM}}\label{pr9}
Let $(x_{1},u_{1}),\ldots,(x_{n},u_{n})$ be $n=2k+1$ points of $\sX\times\R$ such that 
$x_{1}\le x_{2}\le \ldots\le x_{n}$. For all $j\in\{1,\ldots,k\}$, let $i_{j}$ be an index in $\{2j-1,2j\}$ such that $u_{i_{j}}=\max\{u_{2j-1},u_{2j}\}$ and $\cK=\{i_{j}, j\in\{1,\ldots,k\}\}\cup\{2k+1\}$. Let us prove that the subset $\{(x_{i},u_{i}),\ i\in  \cK\}$ cannot be picked up by the subgraphs of the functions $f\in\sF_{k}$. For any $f\in\sF_{k}$, there exists a partition $\sI=\sI(f)$ of $\R$ into at most $k$ intervals on which $f$ is based. Since $n>2k$, there exists at least one interval $I\in \sI$ such that $I$ contains three consecutive points $x_{i}$ and among these three points, there exist two points $x_{i},x_{i'}$ with indices $i\in \cK$ and $i'\not\in \cK$ such that either $(i,i')$ or $(i',i)$ is of the form $(2j-1,2j)$ for some $j\in\{1,\ldots,k\}$. Since $f$ is piecewise constant on the elements on $\sI$ and $u_{i'}\le u_{i}$ whenever the subgraph of $f$ picks up $(x_{i},u_{i})$ then it also picks the point $(x_{i'},u_{i'})$. Hence, no subgraph of $f\in\sF_{k}$ picks up the subset $\{(x_{i},u_{i}),\ i\in  \cK\}$.

\subsection{Proof of Proposition~\ref{pp-dimexp}}\label{Appendix-pp-dimexp}
Let us first prove~\ref{pp-eq1}. The linear span $V$ of $\{g_{1},\ldots,g_{J}\}$ is VC-subgraph with VC-index not larger than $J+2$ by Lemma~2.6.15 of van der Vaart and Wellner~\citeyearpar{MR1385671}. The function $u\mapsto e^{u}$ being increasing, it follows from Proposition~42, $ii)$ of Baraud {\em et al.}~\citeyearpar{MR3595933} that the class $\cF=\{e^{v}, v\in V\}$ is also VC-subgraph with index not larger than $J+2$ and $\cQ$ as well since it is a subset of $\cF$, which concludes the proof of~\ref{pp-eq1}. 
 
By~\ref{pp-eq1}, the families $\cQ_{I}$ with $I\in\sI$ are VC-subgraph on $I$ with indices not larger than $J+2$ and since $\sI$ is a partition of $\sX$ with cardinality not larger than $k$, we deduce from  Baraud and Birg\'e~\citeyearpar{MR3565484}[Lemma~5.3)] that $\cQ$ is VC-subgraph with index not larger than $k(J+2)$ which proves~\ref{pp-eq1b}. 

To prove~\ref{pp-eq2} we fix some $\overline p\in \cQ$. By assumption, an element $q\in\cQ$ consists of functions which are of the form~\eref{formexp} on a partition $\sI(q)$ of $\sX$ into at most $k$ intervals. Since 
\[
\sI(q)\vee \sI(\overline p)=\left\{I\cap I',\ I\in\sI(q),\ I'\in \sI'(\overline p)\right\}
\]  
is a partition of $\sX$ into at most $2k$ intervals and, since on each element of such partition $q/\overline p$ is of the form~\eref{formexp}, the class $(\cQ/\overline p)$ is a $(2k)$-piecewise exponential family based on $J$ functions. Consequently, to prove~\ref{pp-eq2} it suffices to show that a $K$-piecewise exponential family based on $J$ functions is weak-VC major with dimension not larger than $\lceil4.7K(J+2)\rceil$ and to apply the result with $K=2k$. This is precisely the aim of the following proposition.
%
% PROPOSITION
\begin{prop}\label{prop-VCexp}
If $\sF$ is a $K$-piecewise exponential family based on $J$ functions on a non-trivial interval $\sX\subset \R$, it is weak VC-major with dimension not larger than $d=\lceil4.7K(J+2)\rceil$.
\end{prop}
\begin{proof}
Let $u\in\R$. If $u\le 0$, $\sC_{u}(\sF)$ is reduced to the singleton $\{\sX\}$ so that it is VC on $\sX$ with dimension $0<d$. We may therefore assume from now on that $u>0$. Let  $x_{1},\ldots,x_{k}$ be $k>d$ arbitrary points in $\sX$. With no loss of generality we may assume that $x_{1}<\ldots<x_{k}$. We need to prove that $\sC_{u}(\sF)$ cannot shatter $\{x_{1},\ldots,x_{k}\}$. To do so, it suffices to prove that 
\[
\ab{\ac{\{i\in\cI, f(x_{i})>u\},\ f\in\sF}}<2^{k}\quad\mbox{with}\quad\cI=\{1,\ldots,k\}.
\]
Each $f\in\sF$ is associated to a partition $\sI(f)$ of $\sX$ with cardinality not larger than $K\ge 1$ which provides a partition of $\{x_{1},\ldots,x_{k}\}$ into $L\le K$ non-void subsets and this partition induces, via the correspondance $x_{i}\mapsto i$ a partition $\{\cI_{\ell},\ \ell=1,\ldots,L\}$ of $\cI$, each $\cI_{\ell}$ consisting of consecutive integers. Such a partition is therefore determined by a sequence $i_{1}=1<i_{2}<\ldots<i_{L}\le k$ where $i_{\ell}$ denotes the first element of $\cI_{\ell}$. For a given $L\in\{1,\ldots,K\}$, the number $N_{L}$ of possible partitions of $\cI$ into $L$ such subsets is therefore the number $N_{L}$ of choices for $\{i_{2},\ldots,i_{L}\}$: 
\begin{equation}\label{eq-bNL}
N_{L}=\binom{k-1}{L-1}.
\end{equation}
We have seen in~\ref{pp-eq1}) that the class 
\[
\sG=\ac{\exp\pa{\sum_{j=1}^{J}\beta_{j}g_{j}},\ \beta_{1},\ldots,\beta_{J}\in\R}
\] 
is VC-subgraph with dimension (VC-dimension = VC-index$-1$) not larger than $J+1$, it is therefore weak VC-major with dimension not larger than $J+1$ by Proposition~1 of Baraud~\citeyearpar{Bar2016}. Hence the class of subsets $\sC_{u}(\sG)$ is VC with dimension not larger than $J+1$. Given a partition $\{\cI_{\ell},\ \ell=1,\ldots,L\}$ of $\{1,\ldots,k\}$, it follows from Sauer's Lemma (see Sauer~\citeyearpar{MR0307902}) that 
\[
\ab{\ac{\{i\in \cI_{\ell},\ g(x_{i})>u\},\ g\in\sG}}\le \pa{e|\cI_{\ell}|\over J+1}^{J+1}\ \ \mbox{for all}\ \ \ell=1,\ldots,L.
\]
Hence, 
\begin{equation}\label{eq-MF}
\ab{\ac{\{i\in\cI,\ f(x_{i})>u\},\ f\in\sF}}\le \sum_{L=1}^{K}\sum_{\cI_{1},\ldots,\cI_{L}}\prod_{\ell=1}^{L}\pa{e|\cI_{\ell}|\over J+1}^{J+1}
\end{equation}
where the second sum varies among all possible partitions $\{\cI_{1},\ldots,\cI_{L}\}$ of size $L$ of $\cI$ into consecutive integers. Using the concavity of the logarithm, the fact that $\sum_{\ell=1}^{L}|\cI_{\ell}|=k$ and~\eref{eq-bNL} we get, 
\pagebreak
%beg
\begin{align*}
\lefteqn{\sum_{L=1}^{K}\,\sum_{\cI_{1},\ldots,\cI_{L}}\,\prod_{\ell=1}^{L}\pa{e|\cI_{\ell}|\over J+1}^{J+1}}\hspace{15mm}\\
&=\sum_{L=1}^{K}\,\sum_{\cI_{1},\ldots,\cI_{L}}\pa{e\over J+1}^{L(J+1)}\exp\cro{(J+1)L\times {1\over L}\sum_{\ell=1}^{L}\log |\cI_{\ell}|}\\
&\le\sum_{L=1}^{K}\,\sum_{\cI_{1},\ldots,\cI_{L}}\pa{e\over J+1}^{L(J+1)}\exp\cro{(J+1)L\log\left(\frac{k}{L}\right)}\\&=\sum_{L=1}^{K}N_{L}\pa{ek\over L(J+1)}^{L(J+1)}.
\end{align*}
%end
%
Since the function $x\mapsto(ek/x)^{x}$ is increasing on the interval $(0,k]$ and $k>d>K(J+1)$, then
\[
\pa{ek\over L(J+1)}^{L(J+1)}\le \pa{ek\over K(J+1)}^{K(J+1)}\quad\mbox{for all }L\in\{1,\ldots,K\},
\]
so that
%beg
\begin{equation}
\sum_{L=1}^{K}\,\sum_{\cI_{1},\ldots,\cI_{L}}\,\prod_{\ell=1}^{L}\pa{e|\cI_{\ell}|\over J+1}^{J+1}\le\left(\sum_{L=1}^{K}N_{L}\right)\pa{ek\over K(J+1)}^{K(J+1)}.
\label{eq-conv0}
\end{equation}
%end
Since $\sum_{j=0}^{m}\binom{p}{j}\le (ep/m)^{m}$ for $0\le m\le p$, it follows from \eref{eq-bNL} that
\[
\sum_{L=1}^{K}N_{L}=\sum_{L=1}^{K}\binom{k-1}{L-1}<\sum_{j=0}^{K}\binom{k}{j}\le \pa{ek\over K}^{K},\]
and (\ref{eq-conv0}) becomes, using the concavity of the logarithm again,
%beg
\begin{align*}
\lefteqn{\frac{1}{K(J+2)}\log\left(\sum_{L=1}^{K}\,\sum_{\cI_{1},\ldots,\cI_{L}}\,\prod_{\ell=1}^{L}\cro{e|\cI_{\ell}|\over J+1}^{J+1}\right)}\hspace{10mm}\\
&\le\frac{1}{(J+2)}\log\pa{ek\over K}+\frac{J+1}{J+2}\log\pa{ek\over K(J+1)}\le\log\left(\frac{2ek}{K(J+2)}\right).
\end{align*}
%end
Finally (\ref{eq-MF}) leads to
\[
\frac{1}{k}\log\left(\str{4}\ab{\ac{\{i\in\cI, f(x_{i})>u\},\ f\in\sF}}\right)\le\frac{K(J+2)}{k}\log\pa{{2ek\over K(J+2)}}.
\]
One can easily check that the function $x\mapsto x^{-1}\log(2ex)$ is decreasing for $x>1/2$ and smaller than $\log 2$ for $x\ge4.7$ which implies that
\[
\str{4}\ab{\ac{\{i\in\cI, f(x_{i})>u\},\ f\in\sF}}<2^{k}\quad\mbox{for }k\ge4.7K(J+2).
\]
The conclusion follows.
\end{proof}

%SECTION
\section{Other proofs}

%SUBSECTION
\subsection{Proof of Proposition \ref{CEX-MLE}}\label{proof-EXgaussien}
Let us denote by $X_{(1)},\ldots,X_{(n)}$ the order statistics, by $\overline{X}_{n}$ the mean of the observations  and let us work on the set $\Omega_{n}\subset\Omega$ on which the following properties are satisfied:
%beg
\begin{equation}
\overline{X}_{n}\not\in\left\{X_{(1)},\ldots,X_{(n)}\right\}\qquad\mbox{and}\qquad X_{(n)}\ge\sqrt{\log(4n)}>\left|\overline{X}_{n}\right|.
\label{eq-O_n}
\end{equation}
%end
Since $\P(\Omega_{n})\dfleche{n\rightarrow+\infty}1$ whatever $\theta\in\R$, it is enough to show that the MLE is attained at $X_{(n)}$ when the event $\Omega_{n}$ holds which we now assume. It implies in particular that 
%beg
\begin{equation}
n^{-1}\exp\cro{X_{(n)}^{2}}-1\ge3.
\label{eq-omega-n}
\end{equation}
%end
For all $\theta\in\R$, the log-likelihood writes $nL_{n}(\theta)$ with
\[
L_{n}(\theta)=\theta\overline X_{n}-{\theta^{2}\over 2}+{\theta^{2}\over 2n}\left(\sum_{k=1}^{n}\exp\cro{X_{(k)}^{2}}\1_{X_{(k)}}(\theta)\right)\1_{(0,+\infty)}(\theta).
\]
If either $\theta<0$ or $\theta>X_{(n)}$, by (\ref{eq-O_n}) and (\ref{eq-omega-n}),
\begin{align*}
L_{n}(\theta)&=\theta\overline X_{n}-\left(\theta^{2}/2\right)\le\overline X_{n}^{2}/2\\&<
X_{(n)}\overline X_{n}+{X_{(n)}^{2}\over 2}\left(n^{-1}\exp\cro{X_{(n)}^{2}}-1\right)=L_{n}\!\left(X_{(n)}\right),
\end{align*}
so that the MLE necessarily belongs to the interval $\left[0,X_{(n)}\right]$. Moreover,
\[
L_{n}(\theta)\le\overline L_{n}(\theta)=\theta 
\overline X_{n}+\left(\theta^{2}/2\right)\left(n^{-1}\exp\cro{X_{(n)}^{2}}-1\right)\quad\mbox{for all }\theta
\]
and, by (\ref{eq-omega-n}), $\overline L_{n}$ is strictly convex. Since by (\ref{eq-O_n}) and (\ref{eq-omega-n}),
\[
\overline L_{n}(0)=0<X_{(n)}\overline X_{n}+X_{(n)}^{2}\le\overline L_{n}\!\left(X_{(n)}\right),
\]
by convexity the unique maximum of the function $\overline L_{n}$ on the interval $[0,X_{(n)}]$ is reached at the point $X_{(n)}$. The conclusion follows since $L_{n}\!\left(X_{(n)}\right)=\overline L_{n}\!\left(X_{(n)}\right)$ and $L_{n}(\theta)\le\overline{L}_{n}(\theta)$ for all $\theta$.
%\end{supplement}

\subsection{Some comments on Assumptions~\ref{ass-psi} and \ref{H-debase}}\label{Hpsi1}

In view of~\eref{eq-esp} and~\eref{eq-var}, the conditions $a_{0}\ge 1$, $a_{1}\le 1$ and  $a_{2}^{2}\ge 1\vee(6a_{1})$ can always be satisfied by enlarging $a_0$ and $a_2$ and diminishing $a_1$ if necessary. The conditions $a_{0}\ge 1\ge a_{1}$ and  $a_{2}\ge 1$ turn out to be necessary when $\psi(+\infty)=1$ and there exist two probabilities $\gQ,\gQ'\in\sbQ$ such that $h(Q_{i},Q_{i}')=1$ for some $i\in\{1,\ldots,n\}$. In this case, for $R_{i}=Q_{i}'$ and any reference measure $\mu_{i}$, $\psi\pa{\sqrt{(q'_{i}/q_{i})(x)}}=\psi(+\infty)=1$ for $R_{i}$-almost all $x\in\sX_{i}$ so that the left-hand side of~\eref{eq-esp} equals 1 while the right-hand side equals $a_{0}h^{2}(R_{i},Q_{i})=a_{0}$ leading to the inequality $a_{0}\ge 1$. The same argument applies to~\eref{eq-var} and leads to the inequality $a_{2}\ge 1$. Taking now $R_{i}=Q_{i}$,  $\psi\pa{\sqrt{q'_{i}/q_{i}}(x)}=\psi(0)=-1$ for $R_{i}$-almost all $x\in\sX_{i}$ so that the left-hand side of~\eref{eq-esp} equals $-1$ while the right hand-side equals $-a_{1}h^{2}(R_{i},Q_{i}')=-a_{1}$ which leads to the inequality $a_{1}\le 1$. 
	
Inequality~\eref{eq-var} actually holds if $\psi$ is Lipschitz on $[0,+\infty)$, which is required by Assumption~\ref{ass-psi}. It indeed follows from Lemma~\ref{Cond-Lips} below that \eref{eq-var} holds under the weaker assumption that $\psi$ satisfies the following inequality for some $L>0$:
%beg
\begin{equation}
\ab{\psi(x)}=\ab{\psi(x)-\psi(1)}\le L\ab{x-1}\ \ \mbox{for all}\ \ x\in\R_{+}.
\label{eq-01}
\end{equation}
%end
%
% LEMMA
\begin{lem}\label{Cond-Lips} 
If $\psi$ satisfies Assumption~\ref{ass-psi} without the Lipschitz condition together with \eref{eq-01}, inequality~\eref{eq-var} is satisfied with $a_{2}=2L+1$.
\end{lem}
\begin{proof}
Let $(\gmu,\cbQ)$ be some representation of $\sbQ$, $\gq,\gq'$ two densities belonging to $\cbQ$, $\gR$ an arbitrary element of $\sbP$ and $i$ some index in $\{1,\ldots,n\}$. Decompose $R_{i}$ in the form $r_{i}\cdot\mu_{i}+R''_{i}$ with $R''_{i}$ orthogonal to $\mu_{i}$. For simplicity we drop the index $i$ and write $\sX,\mu,R,R'',Q,Q',r,q,q'$ for $\sX_{i}$, $\mu_{i}$, $R_{i}$, $R_{i}''$, $Q_{i}$, $Q_{i}'$, $r_{i}$, $q_{i}$, $q_{i}'$. Using the inequality, valid for all $\alpha>0$, 
\[
r=\pa{\sqrt{r}-\sqrt{q}+\sqrt{q}}^{2}\le(1+\alpha)\pa{\sqrt{r}-\sqrt{q}}^{2}+\left(1+\alpha^{-1}\right)q,
\]
the fact that $\psi^2$ is bounded by 1 and~\eref{eq-01}, we obtain that
%beg
\begin{align*}
\int_{\sX}\psi^{2}\pa{\sqrt{q'/ q}}r\,d\mu\le&\;(1+\alpha)\int_{\sX}\psi^{2}\pa{\sqrt{q'/ q}}\pa{\sqrt{r}-\sqrt{q}}^{2}d\mu\\
&\;+\left(1+\alpha^{-1}\right)\int_{\sX}\psi^{2}\pa{\sqrt{q'/q}}q\,d\mu\\
=&\;(1+\alpha)\int_{\sX}\psi^{2}\pa{\sqrt{q'/q}}\pa{\sqrt{r}-\sqrt{q}}^{2}d\mu\\
&\;+\left(1+\alpha^{-1}\right)
\int_{\sX\cap\{q>0\}}\psi^{2}\pa{\sqrt{q'/q}}q\,d\mu\\\le&\;
(1+\alpha)\int_{\sX}\pa{\sqrt{r}-\sqrt{q}}^{2}d\mu\\
&\;+\left(1+\alpha^{-1}\right)L^{2}\int_{\sX\cap\{q>0\}}\pa{\sqrt{q'/q}-1}^{2}q\,d\mu,
\end{align*}
%end
hence
\[
\int_{\sX}\psi^{2}\pa{\sqrt{q'/ q}}r\,d\mu\le(1+\alpha)\!\int_{\sX}\!\pa{\sqrt{r}-\sqrt{q}}^{2}d\mu+2\left(1+\alpha^{-1}\right)L^{2}h^{2}(Q,Q').
\]
Since
\[
2h^{2}(R,Q)=\int_{\sX}\pa{\sqrt{r}-\sqrt{q}}^{2}d\mu+\int_{\sX}dR'',
\]
we finally get
%beg
\begin{align*}
\int_{\sX}\psi^{2}\pa{\sqrt{q'/ q}}dR&\le\int_{\sX}\psi^{2}\pa{\sqrt{q'/ q}}r\,d\mu+\int_{\sX}dR''\\&\le2(1+\alpha)h^{2}(R,Q)+2\left(1+\alpha^{-1}\right)L^{2}h^{2}(Q,Q').
\end{align*}
%end
A similar bound holds for $\int_{\sX}\psi^{2}\pa{\sqrt{q/ q'}}dR$. Using (\ref{H-phi}) and averaging the two bounds, then using $h^{2}(Q,Q')\le 2\pa{h^{2}(R,Q)+h^{2}(R,Q')}$, we get
%beg
\begin{align*}
\int_{\sX}\psi^{2}\pa{\sqrt{q'/ q}}dR=&\;\int_{\sX}\psi^{2}\pa{\sqrt{q/ q'}}dR\\ \le&\;
(1+\alpha)\cro{h^{2}(R,Q)+h^{2}(R,Q')}\\
&\;+2\left(1+\alpha^{-1}\right)L^{2}h^{2}(Q,Q')\\\le&\;
\cro{h^{2}(R,Q)+h^{2}(R,Q')}\left[1+\alpha+4L^{2}(1+\alpha^{-1})\right].
\end{align*}
%end
The conclusion follows by choosing $\alpha=2L$.
\end{proof}

%SUBSECTION
\subsection{Proof of Proposition~\ref{prop-universel}}\label{sect-pprop-universel}
We start with the following result.
\begin{lem}\label{lem-ext}
If \eref{eq-esp} and \eref{eq-var} hold with $a_{0}\ge 2(1+a_{1})$ for a representation $(\gmu,\cbQ)$ and all $\gR\ll\gmu$, they hold for this representation and all $\gR\in\sbP$.
\end{lem}
\begin{proof}
Let $\gq,\gq'\in\cbQ$ and $\gR$ be some probability on $(\sbX,\sbB)$ which is not necessarily absolutely continuous with respect to $\gmu$. We stick here to the notations introduced in Section~\ref{ES3} and drop the index $i$. We may write $R=\delta^{2} R'+(1-\delta^{2})R''$ where $R'$ is a probability which is absolutely continuous with respect to $\mu$, $R''$ is a probability which is orthogonal to $\mu$ and  $\delta\in [0,1]$. Taking $
\overline \mu=R+Q$ which dominates both $R$ and $Q$ and using the fact that $(dR''/d\overline \mu)(dQ/d\overline \mu)=0$, $\overline \mu$-a.e. since $Q=q\cdot \mu$ is orthogonal to $R''$, we get
\begin{align*}
\rho(R,Q)&=\int_{\sX}\sqrt{\left[\delta^{2} {dR'\over d\overline \mu}+(1-\delta^{2}){dR''\over d\overline \mu}\right]{dQ\over d\overline \mu}}\,d\overline \mu\\
&=\int_{\sX}\delta\sqrt{{dR'\over d\overline \mu}{dQ\over d\overline \mu}}\,d\overline \mu=\delta\rho(R',Q)\le\delta.
\end{align*}
Hence $\delta h^{2}(R',Q)=\delta-\delta\rho(R',Q)=\delta-\rho(R,Q)$, so that
\begin{equation}
\delta h^{2}(R',Q)=\delta-1+h^{2}(R,Q)\qquad\mbox{and}\qquad h^{2}(R,Q)\ge1-\delta,
\label{eq-2et}
\end{equation}
with similar results for $Q'=q'\cdot \mu$. Then, applying~\eref{eq-var} to $R'\ll \mu$ and using the fact that $a_{2}\ge1\ge|\psi|$, we derive from (\ref{eq-2et}) that
%beg
\begin{align*}
\int_{\sX}\psi^{2}\pa{\sqrt{q'\over q}}dR&\le\delta^{2}\int_{\sX}\psi^{2}\pa{\sqrt{q'\over q}}dR'+1-\delta^{2}\\&\le\delta^{2} a_{2}^{2}\cro{h^{2}(R',Q)+h^{2}(R',Q')}+a_{2}^{2}(1-\delta^{2})\\&=
a_{2}^{2}\left[2\delta^{2} -2\delta+\delta\left[h^{2}(R,Q)+h^{2}(R,Q')\right]+1-\delta^{2}\right]\\
&=a_{2}^{2}\left[(1-\delta)^{2}+\delta A\right],
\end{align*}
%end
with $A=h^{2}(R,Q)+h^{2}(R,Q')\ge2(1-\delta)$, hence $(1-\delta)^{2}+\delta A\le (1-\delta)(A/2)+ \delta A\le A$ which leads to (\ref{eq-var}).

Let us now focus on (\ref{eq-esp}). The same computations, using \eref{eq-2et}, lead to
%beg
\begin{align*}
\lefteqn{\int_{\sX}\psi\pa{\sqrt{q'\over q}}dR}\hspace{25mm}\\\le&\;\delta^{2}\cro{a_{0}h^{2}(R',Q)-a_{1}h^{2}(R',Q')}+1-\delta^{2}\\
=&\;\delta(\delta-1)(a_{0}-a_{1})+\delta\cro{a_{0}h^{2}(R,Q)-a_{1}h^{2}(R,Q')}+1-\delta^{2}\\
=&\;a_{0}h^{2}(R,Q)-a_{1}h^{2}(R,Q')\\&-(1-\delta)\left[\delta(a_{0}-a_{1})+a_{0}h^{2}(R,Q)-a_{1}h^{2}(R,Q')-1-\delta\right]\\\le&\;a_{0}h^{2}(R,Q)-a_{1}h^{2}(R,Q')\\
&\;-(1-\delta)\left[\delta(a_{0}-a_{1})+a_{0}(1-\delta)-a_{1}-1-\delta\right]\\
=&\;a_{0}h^{2}(R,Q)-a_{1}h^{2}(R,Q')-(1-\delta)\left[a_{0}-(a_{1}+1)(1+\delta)\right].
\end{align*}
%end
Our conclusion follows since $a_{0}-(a_{1}+1)(1+\delta)\ge a_{0}-2(a_{1}+1)\ge 0$ in view of our assumption on $a_{0}$ and $a_{1}$.
\end{proof}
Let us now proceed with the proof of Proposition~\ref{prop-universel}.
For all $i\in\{1,\ldots,n\}$, let $\nu_{i}$ be a privileged probability measure for $\sQ_{i}=\{Q_{i},\ \gQ\in\sbQ\}$, which means that, for all $A\in\sB_{i}$,
\begin{equation}\label{eq-pvl}
\nu_{i}(A)=0\quad\mbox{if and only if}\quad Q_{i}(A)=0\ \ \mbox{for all }Q_{i}\in\sQ_{i}.
\end{equation}
Such a dominating probability measure exists as soon as  $\sQ_{i}$ is separable with respect to the Hellinger distance which is the case here since $\sQ_{i}$ is countable. We may therefore consider a representation $(\gnu,\cbT)$ of $\sbQ$ based on $\gnu=(\nu_{1},\ldots,\nu_{n})$. Let $(\gmu,\cbQ)$ be an alternative one. For each $i$, $\nu_{i}\ll\mu_{i}$ and $\mu_{i}$ decomposes (in a unique way) as $\mu_{i}=\mu_{i}'+\mu''_{i}$ where $\mu_{i}'\ll\nu_{i}$, hence  $\mu_{i}'=m_{i}\cdot\nu_{i}$ for some non-negative  function $m_{i}$ on $\sX_{i}$, and $\mu''_{i}$ is orthogonal to $\nu_{i}$. Consequently, for all $i\in\{1,\ldots,n\}$ and $Q_{i}\in\sQ_{i}$,
\[
Q_{i}=q_{i}\cdot \mu_{i}=q_{i}m_{i}\cdot \nu_{i}+q_{i}\cdot\mu_{i}''=t_{i}\cdot\nu_{i}\ \ \mbox{with}\ \ q_{i}\in\cQ_{i}\ \mbox{and}\ t_{i}\in\cT_{i}.
\]
Since $Q_{i}\ll\nu_{i}$, $q_{i}=0$ $\mu_{i}''$-a.e.\ and $t_{i}=q_{i}m_{i}$ $\nu_{i}$-a.e. Moreover $Q_{i}(\{m_{i}=0\})=0$ for all $Q_{i}\in\sQ_{i}$, hence $\nu_{i}(\{m_{i}=0\})=0$ and $m_{i}>0$ $\nu_{i}$-a.e. If $Q_{i}'=t_{i}'\cdot \nu_{i}=q_{i}'\cdot \mu_{i}\in\sQ_{i}$ then $t'_{i}=q'_{i}m_{i}$ $\nu_{i}$-a.e.\ and
\begin{equation}\label{eq-condnu2}
\psi\pa{\sqrt{q'_{i}\over q_{i}}}=\psi\pa{\sqrt{q'_{i}m_{i}\over q_{i}m_{i}}}=\psi\pa{\sqrt{t'_{i}\over t_{i}}}\;\;\nu_{i}\mbox{-a.e.}
\end{equation}
Since $q_{i}=q_{i}'=0$ $\mu_{i}''$-a.e., the first equality in (\ref{eq-condnu2}) is also true $\mu_{i}''$-a.e.\ for all $\gq,\gq'\in\cbQ$ and $i\in\{1,\ldots,n\}$. Consequently, for any probability measure $\gR\in\sP$ such that $\gR\ll \gmu$ and all $\gq,\gq'\in\cbQ$,
\begin{equation}\label{eq-condnu}
\psi\pa{\sqrt{q'_{i}\over q_{i}}}=\psi\pa{\sqrt{q'_{i}m_{i}\over q_{i}m_{i}}}\;\;R_{i}\mbox{-a.s.}\;\mbox{ for all }i\in\{1,\ldots,n\}.
\end{equation}
If for the representation $(\gmu,\cbQ)$ of $\sbQ$ the inequalities \eref{eq-esp} and~\eref{eq-var} hold for all $\gR\ll\gmu$, they hold for all $\gR\ll \gnu\ll \gmu$ and it follows from~\eref{eq-condnu2} that they also hold for the representation $(\gnu,\cbT)$ and all $\gR\ll\gnu$. Conversely, if the inequalities \eref{eq-esp} and~\eref{eq-var} are satisfied for the representation $(\gnu,\cbT)$ and all $\gR\ll\gnu$, they also hold for the representation 
\[
(\gnu,\cbT'),\quad\cbT'=\{(q_{1}m_{1},\ldots,q_{n}m_{n}),\ \gq\in\cbQ\},
\]
and all $\gR\ll\gnu$ because of~\eref{eq-condnu2}. Lemma~\ref{lem-ext} then shows that they also hold for the representation $(\gnu,\cbT')$ and all $\gR\in\sbP$, therefore also for the representation $(\gmu,\cbQ)$ and all $\gR\ll\gmu$ by~\eref{eq-condnu}. This completes the proof.

%%%% Subsection %%%
%%%%%

\subsection{Proof of Proposition~\ref{prop-expsi}}\label{pprop-expsi}
{
It is clear that both functions are monotone and satisfy~\eref{H-phi}. Let $\sbQ$ be an arbitrary $\rho$-model and $(\gmu,\cbQ)$ a representation of it. In view of Proposition~\ref{prop-universel}, it is enough to prove \eref{eq-esp} and~\eref{eq-var} when $\gR=\gS\ll\gmu$, which we shall assume in the sequel, denoting by $\gs$ the corresponding density. As in the proof of Lemma~\ref{Cond-Lips}, we fix some $i\in\{1\ldots,n\}$, $\gq,\gq'\in\cbQ$ and then drop the index $i$ in the notations to establish~\eref{eq-esp} and~\eref{eq-var}. Given two densities $t,t'$ on $(\sX,\sB,\mu)$ we shall write $h(t,t')$ and $\rho(t,t')$ for the Hellinger distance and the Hellinger affinity between the probabilities $t\cdot \mu$ and $t'\cdot\mu$.
The proof will repetedly use that $(a+b)^{2}\le (1+\alpha)a^{2}+\left(1+\alpha^{-1}\right)b^{2}$ for all $\alpha>0$.

%PARAGRAPH
\paragraph{{\bf Case of the function $\psi_{1}$}.}
Let $r=(q+q')/2$. Our conventions $0/0=1$ and $a/0=+\infty$ for $a>0$ imply that the equalities
%beg
\begin{equation}
\psi_{1}\left(\sqrt{q'\over q}\right)=\frac{\sqrt{q'}-\sqrt{q}}{\sqrt{q+q'}}\1_{r>0}={\sqrt{q'}-\sqrt{q}\over \sqrt{2r}}\1_{r>0}
\label{eq-psi1-0}
\end{equation}
%end
hold for all densities $q,q'$. Moreover, by the concavity of the square root,
%beg
\begin{align}
h^{2}(s,r)&=1-\rho(s,r)=1-\int\sqrt{\frac{sq+sq'}{2}}\,d\mu\nonumber\\&\le1-\frac{1}{2}
\left[\int\sqrt{sq}\,d\mu+\int\sqrt{sq'}\,d\mu\right]=\frac{1}{2}\left[h^{2}(s,q)+h^{2}(s,q')\right].
\label{eq-h}
\end{align}
%end
Squaring (\ref{eq-psi1-0}), integrating with respect to $S=s\cdot\mu$, using the bound 
$\left(\sqrt{q'}-\sqrt{q}\right)^{2}\le2r$ and then (\ref{eq-h}), we get, 
%beg
\begin{align*}
\int_{\sX}\psi_{1}^{2}\left(\sqrt{q'\over q}\right)s\,d\mu=&\;\int_{r>0}\frac{\left(\sqrt{q'}-\sqrt{q}\right)^{2}}{2r}\left(\sqrt{s}-\sqrt{r}+\sqrt{r}\right)^{2}d\mu\\ 
\le&\;\int_{r>0}\left[(1+\alpha)\frac{\left(\sqrt{q'}-\sqrt{q}\right)^{2}}{2r}\left(\sqrt{s}-\sqrt{r}\right)^{2}\right.\\&\;+\left.\left(1+\alpha^{-1}\right)\frac{\left(\sqrt{q'}-\sqrt{q}\right)^{2}}{2r}r\right]d\mu\\\le&\; 2(1+\alpha)h^{2}(s,r)+\left(1+\alpha^{-1}\right)h^{2}(q,q')\\ 
\le&\; (1+\alpha)\left[h^{2}(s,q)+h^{2}(s,q')\right]\\&\;+2\left(1+\alpha^{-1}\right)\left[h^{2}(s,q)+h^{2}(s,q')\right].
\end{align*}
%end
Setting $\alpha=\sqrt{2}$ leads to $a_{2}^{2}=3+2\sqrt{2}$ which proves (\ref{eq-var}).

The proof of~\eref{eq-esp} is based on (\ref{eq-h}) and the following inequalities:
%beg
\begin{equation}
0\le \sqrt{a+b\over 2}-{\sqrt{a}+\sqrt{b}\over 2}\le {\sqrt{2}-1\over 2}\left|\sqrt{a}-\sqrt{b}\right|
\quad\mbox{for all }a,b\ge0.
\label{eq-rac}
\end{equation}
%end
The concavity of the square root leads to the left-hand side of~\eref{eq-rac}. For the right-hand side, note that $z\mapsto \sqrt{(1+z^{2})/2}$ is convex and its graph being under any of its chords, for all $z\in [0,1]$,
\[
 \sqrt{1+z^{2}\over 2}\le {1\over \sqrt{2}}+z\left(1-{1\over \sqrt{2}}\right)={1+z\over 2}+{\sqrt{2}-1\over 2}(1-z).
\]
The result follows by applying this inequality to $z=\sqrt{(a\wedge b)/(a\vee b)}$ when $a\vee b\neq 0$, the case $a\vee b= 0$ being trivial.

Let us now turn to the proof of~\eref{eq-esp} and let 
\[
A=\int_{\sX}\psi_{1}\left(\sqrt{q'\over q}\right)s\,d\mu-\int_{r>0}{\sqrt{q'}-\sqrt{q}\over \sqrt{2r}}\left(\sqrt{s}-\sqrt{r}\right)^{2}d\mu.
\]
We derive from (\ref{eq-psi1-0}) that
%beg
\begin{align*}
A&=\int_{r>0}{\sqrt{q'}-\sqrt{q}\over \sqrt{2r}}\left(\sqrt{s}-\sqrt{r}+\sqrt{r}\right)^{2}d\mu-\int_{r>0}{\sqrt{q'}-\sqrt{q}\over \sqrt{2r}}\left(\sqrt{s}-\sqrt{r}\right)^{2}d\mu\\
&=\int_{r>0}\left(\sqrt{q'}-\sqrt{q}\right)\!\left(\sqrt{2s}-\sqrt{\frac{r}{2}}\right)d\mu=
\int\left(\sqrt{q'}-\sqrt{q}\right)\!\left(\sqrt{2s}-\sqrt{\frac{r}{2}}\right)d\mu\\
&=\sqrt{2}\left[\rho(s,q')-\rho(s,q)\right]+\int\frac{\sqrt{q}-\sqrt{q'}}{\sqrt{2}}\sqrt{r}\,d\mu-\int\frac{q-q'}{2\sqrt{2}}\,d\mu\\&=\sqrt{2}\left[h^{2}(s,q)-h^{2}(s,q')\right]+
\int\frac{\sqrt{q}-\sqrt{q'}}{\sqrt{2}}\left[\sqrt{\frac{q+q'}{2}}-{\sqrt{q}+\sqrt{q'}\over 2}\right]d\mu.
\end{align*}
%end
The inequality $|\sqrt{q}-\sqrt{q'}|\le \sqrt{2r}$ and~\eref{eq-h} imply that
\[
\int_{r>0}{\sqrt{q'}-\sqrt{q}\over \sqrt{2r}}\left(\sqrt{s}-\sqrt{r}\right)^{2}d\mu\le 2h^{2}(s,r)\le h^{2}(s,q)+h^{2}(s,q')
\]
and~\eref{eq-rac} yields
\[
\int\frac{\sqrt{q}-\sqrt{q'}}{\sqrt{2}}\left[\sqrt{\frac{q+q'}{2}}-{\sqrt{q}+\sqrt{q'}\over 2}\right]d\mu\le \frac{\sqrt{2}-1}{\sqrt{2}}h^{2}(q,q').
\]
Therefore 
\begin{align*}
\lefteqn{\int_{\sX}\psi_{1}\left(\sqrt{q'\over q}\right)s\,d\mu}\hspace{15mm}\\&
=A+\int_{r>0}{\sqrt{q'}-\sqrt{q}\over \sqrt{2r}}\left(\sqrt{s}-\sqrt{r}\right)^{2}d\mu\\&\le
\left(1+\sqrt{2}\right)h^{2}(s,q)-\left(\sqrt{2}-1\right)h^{2}(s,q')+\frac{\sqrt{2}-1}{\sqrt{2}}h^{2}(q,q'),
\end{align*}
hence
%beg
\begin{align*}
\sqrt{2}\int_{\sX}\psi_{1}\left(\sqrt{q'\over q}\right)s\,d\mu\le&\,\left[\left(2+\sqrt{2}\right)+\left(\sqrt{2}-1\right)(1+\alpha)\right]h^2(s,q)\\
&\,-\left(\sqrt{2}-1\right)\left[\sqrt{2}-\left(1+\alpha^{-1}\right)\right]h^2(s,q').
\end{align*}
%end
The choice $\alpha=7.7$ implies \eref{eq-esp}.

%PARAGRAPH
\paragraph{{\bf Case of the function $\psi_{2}$}}
Let us set $r=[(\sqrt{q}+\sqrt{q'})/\delta]^{2}$ with
\[
\delta^{2}=\int_{\sX}\!\pa{\sqrt{q}+\sqrt{q'}}^{2}d\mu=2\left[1+\rho(q,q')\right]=4\left[1-\frac{h^2(q,q')}{2}\right],
\]
so that $r$ is the density of a probability on $(\sX,\sB)$. Moreover,
%beg
\begin{equation}
\sqrt{2}\le\delta\le2\qquad\mbox{and}\qquad\frac{2}{\delta}=\frac{1}{\sqrt{1-(1/2)h^2(q,q')}}\ge1+\frac{h^{2}(q,q')}{4}
\label{eq-psi2-0}
\end{equation}
%end
by the convexity of the map $u\mapsto {1/\sqrt{1-u}}$. Consequently,
%beg
\begin{align}
h^{2}(s,r)&=1-\int_{\sX}\sqrt{sr}\,d\mu=1-{1\over \delta}\cro{\rho(s,q')+\rho(s,q)}
\label{eq-R0}
\\&=1-{2\over \delta}+{1\over \delta}\left[h^{2}(s,q)+h^{2}(s,q')\right]\nonumber\\&\le
\frac{h^{2}(s,q)+h^{2}(s,q')}{\delta}-\frac{h^{2}(q,q')}{4}.
\label{eq-R}
\end{align}
%end
The previous computations with this new value of $r$ lead to 
%beg
\begin{align*}
\int_{\sX}\psi_{2}^{2}\left(\sqrt{q'\over q}\right)s\,d\mu=\:&\int_{r>0}\left(\frac{\sqrt{q'}-\sqrt{q}}{\sqrt{q'}+\sqrt{q}}\right)^{2}\left(\sqrt{s}-\sqrt{r}+\sqrt{r}\right)^{2}d\mu\\\le\:&(1+\alpha)\int_{r>0}\left(\frac{\sqrt{q'}-\sqrt{q}}{\sqrt{q'}+\sqrt{q}}\right)^{2}\left(\sqrt{s}-\sqrt{r}\right)^{2}d\mu\\&+\left(1+\alpha^{-1}\right)\int_{r>0}\left(\frac{\sqrt{q'}-\sqrt{q}}{\sqrt{q'}+\sqrt{q}}\right)^{2}\left({\sqrt{q}+\sqrt{q'}\over\delta}\right)^2
d\mu\\\le\:&2(1+\alpha)h^{2}(s,r)+2\left(1+\alpha^{-1}\right)\delta^{-2}h^{2}(q,q')
\end{align*}
%end
and by (\ref{eq-R}),
%beg
\begin{align*}
\lefteqn{\int_{\sX}\psi_{2}^{2}\left(\sqrt{q'\over q}\right)s\,d\mu}\hspace{10mm}\\
&\le2(1+\alpha)\left[\frac{h^{2}(s,q)+h^{2}(s,q')}{\delta}-\frac{h^{2}(q,q')}{4}\right]+2\left(1+\alpha^{-1}\right)\frac{h^{2}(q,q')}{\delta^{2}}\\&=2(1+\alpha)\delta^{-1}\left[h^{2}(s,q)+h^{2}(s,q')\right]
\end{align*}
%end
provided that $(1+\alpha)/4=\delta^{-2}\left(1+\alpha^{-1}\right)$. Solving this equation with respect to $\alpha$ leads to $\alpha=4\delta^{-2}$ hence $2(1+\alpha)\delta^{-1}=2\delta^{-3}
\left(\delta^{2}+4\right)$ which is a decreasing function of $\delta$. We then conclude from (\ref{eq-psi2-0}) that $2(1+\alpha)\delta^{-1}\le3\sqrt{2}$ which gives $a_{2}^{2}=3\sqrt{2}$.

Let us now turn to the proof of~\eref{eq-esp}, setting
%beg
\begin{equation}
\rho_{r}(S,q)={1\over 2}\cro{\int_{r>0}\sqrt{qr}\,d\mu+\int_{r>0}\sqrt{q\over r}\,dS}.
\label{eq-rho-r}
\end{equation}
%end
Then
%beg
\begin{align*}
\int_{r>0}&\sqrt{q\over r}\,dS\\
&=\int_{r>0}\sqrt{q\over r}s\,d\mu=\int_{r>0}\sqrt{q\over r}\left(\sqrt{s}-\sqrt{r}+\sqrt{r}\right)^{2}d\mu\\
&=
\int_{r>0}\sqrt{q\over r}\left(\sqrt{s}-\sqrt{r}\right)^{2}d\mu+\int_{r>0}\sqrt{qr}\,d\mu+2\int_{r>0}\sqrt{q}
\left(\sqrt{s}-\sqrt{r}\right)d\mu\\&=\int_{r>0}\sqrt{q\over r}\left(\sqrt{s}-\sqrt{r}\right)^{2}d\mu-\int_{r>0}\sqrt{qr}\,d\mu+2\int_{r>0}\sqrt{qs}\,d\mu,
\end{align*}
%end
so that, since $r=0$ implies $q=0$,
\begin{align*}
\rho_{r}(S,q)&=\rho(s,q)+{1\over 2}\int_{r>0}\sqrt{q\over r}\pa{\sqrt{s}-\sqrt{r}}^{2}d\mu\\
&=\rho(s,q)+{\delta\over 2}\int_{r>0}{\sqrt{q}\over \sqrt{q'}+\sqrt{q}}\pa{\sqrt{s}-\sqrt{r}}^{2}d\mu.
\end{align*}
Moreover, since $q$ and $q'$ are densities, (\ref{eq-rho-r}) leads to
\begin{align*}
\lefteqn{\rho_{r}(S,q')-\rho_{r}(S,q)}\hspace{6mm}\\
&={1\over 2\delta}\int_{r>0}\pa{\sqrt{q'}-\sqrt{q}}\pa{\sqrt{q'}+\sqrt{q}}d\mu +{\delta\over 2}\int_{r>0}{\sqrt{q'}-\sqrt{q}\over \sqrt{q'}+\sqrt{q}}\,dS\\
&={1\over 2\delta}\int_{r>0}\pa{q'-q}d\mu+{\delta\over 2}\int_{r>0}\psi_{2}\pa{\sqrt{q'\over q}}dS={\delta\over 2}\int_{\sX}\psi_{2}\pa{\sqrt{q'\over q}}dS,
\end{align*}
since $q=q'=0$ when $r=0$ and, by convention, $\psi_{2}(0/0)=\psi_{2}(1)=0$. Putting everything together, we derive that
%beg
\begin{align*}
{\delta\over 2}\int_{\sX}\psi_{2}\pa{\sqrt{q'\over q}}dS&=\rho(s,q')- \rho(s,q)+{\delta\over 2}
\int_{r>0}{\sqrt{q'}-\sqrt{q}\over \sqrt{q'}+\sqrt{q}}\pa{\sqrt{s}-\sqrt{r}}^{2}d\mu\\&=\rho(s,q')- \rho(s,q)+{\delta\over 2}\int_{r>0}\psi_{2}\pa{\sqrt{q'\over q}}\!\pa{\sqrt{s}-\sqrt{r}}^{2}d\mu.
\end{align*}
%end
Since $|\psi_{2}|$ is bounded by 1 we derive from \eref{eq-R0} that
\[
{\delta\over 2}\int_{\sX}\psi_{2}\pa{\sqrt{q'\over q}}dS\le \rho(s,q')- \rho(s,q)+\delta h^{2}(s,r)=\delta-2\rho(s,q),
\]
hence, by (\ref{eq-psi2-0}), 
\begin{align}
\int_{\sX}\psi_{2}\pa{\sqrt{q'\over q}}dS&\le2\cro{1-{2\rho(s,q)\over \delta}}\le2\cro{1-\rho(s,q)\pa{1+{h^{2}(q,q')\over 4}}}\nonumber\\&=\nonumber
2\cro{h^{2}(s,q)\pa{1+{h^{2}(q,q')\over 4}}-{h^{2}(q,q')\over 4}}\\&\le{1\over 2}\cro{5h^{2}(s,q)-h^{2}(q,q')}.
\label{eq-psi21}
\end{align}
Since $h(q,q')\ge \ab{h(s,q)-h(s,q')}$, we deduce from~\eref{eq-psi21} with $\alpha=4$ that,
as claimed,
\begin{align*}
\lefteqn{\int_{\sX}\psi_{2}\pa{\sqrt{q'\over q}}dS}\hspace{10mm}\\
&\le{1\over 2}\cro{5h^{2}(s,q)-\pa{h^{2}(s,q)+h^{2}(s,q')-2\alpha^{1/2} h(s,q)\alpha^{-1/2}h(s,q')}}\\
&\le{1\over 2}
\cro{(4+\alpha)h^{2}(s,q)-(1-\alpha^{-1})h^{2}(s,q')}=4h^{2}(s,q)-{3\over 8}h^{2}(s,q').
\end{align*}

%SUBSECTION
\subsection{Proof of Proposition~\ref{prop-gw}}\label{sect-pprop-gw}
Let us fix $\gP$ and $\overline \gP$ in $\sbP$.
Since 
\[
w(\cR,\sbQ,\gP,\overline \gP,y)=w(\cR,\sbQ\cup\{\overline \gP\},\gP,\overline \gP,y)
\] 
we may assume with no loss of generality that $\overline \gP\in \sbQ$ and that $\cR_{1}=\cR$ is a representation of $\sbQ$. Let us now consider another representation $\cR_{2}=(\gnu,{\cbU})$ of $\sbQ$ that we may therefore alternatively write as  $\sbQ=\{\gu\cdot\gnu,\,\gu\in\cbU\}$. It is suffices to prove that
\begin{equation}\label{eq-amontrer}
w(\cR_{1},\sbQ,\gP,\overline \gP,y)\le w(\cR_{2},\sbQ,\gP,\overline \gP,y)+8\gh^{2}(\gP,\overline \gP)\quad \mbox{for all $y>0$.}
\end{equation}
 
For all $i\in\{1,\ldots,n\}$, $\mu_{i}$ can be decomposed (in a unique way) into a part which is absolutely continuous with respect to $\nu_{i}$, denoted $\mu_{i}'=m_{i}\cdot \nu_{i}$ and a part which is orthogonal to $\nu_{i}$, denoted $\mu_{i}''$. If $\gQ\in\sbQ$ with $\gQ=\gq\cdot\gmu=\gu\cdot\gnu$, then
\[
Q_{i}=q_{i}\cdot \mu_{i}=q_{i}m_{i}\cdot \nu_{i}+ q_{i}\cdot\mu_{i}''=u_{i}\cdot \nu_{i}\quad\mbox{for }i\in\{1,\ldots,n\}.
\]
Let 
\[
A_{i}(\gQ)=\{x\in\sX_{i}\,|\, u_{i}(x)=q_{i}(x)m_{i}(x)\}\;\mbox{ for }\gQ\in\sbQ\quad\mbox{and}\quad B_{i}=\{m_{i}>0\}. 
\]
For all $\gQ\in\sbQ$, $Q_{i}\ll\nu_{i}$ so that $u_{i}=q_{i}m_{i}$ $\nu_{i}$-a.e., hence $\overline{P}_{i}(A_{i}(\gQ)^{c})=\nu_{i}(A_{i}(\gQ)^{c})=0$. Moreover $\overline{P}_{i}(B_{i}^{c})=0$. Finally, since $\sbQ$ is countable,
\[
\overline{P}_{i}(A_{i}^{c})=0\quad\mbox{with}\quad A_{i}=\left(\bigcap_{\gQ\in\sbQ}A_{i}(\gQ)\right)\bigcap B_{i}\quad\mbox{for all }i\in\{1,\ldots,n\}.
\]
If $\overline{\gP}=\overline{\gp}\cdot \gmu=\overline{\gs}\cdot \gnu$, it then follows from the definition of $A_{i}$ that
%beg
\begin{equation}
\psi\pa{\sqrt{q_{i}\over\overline{p}_{i}}(x_{i})}\1_{A_{i}}(x_{i})=\psi\pa{\sqrt{u_{i}\over\overline{s}_{i}}(x_{i})}\1_{A_{i}}(x_{i})\:\mbox{ for all }x_{i}\in\sX_{i}\mbox{ and }Q_{i}\in \sQ_{i}.
\label{eq-nupp}
\end{equation}
%end
For $\gq\in \cbQ$ and $\gx\in\sbX$, let
\[
\left\{
\begin{array}{lll}  
\gT_{\gA}(\gx,\overline \gp,\gq)&=&\dps{\sum_{i=1}^{n}\psi\pa{\sqrt{q_{i}\over \overline p_{i}}(x_{i})}\1_{A_{i}}(x_{i})},\\
\gT_{\gA^{\!c}}(\gx,\overline \gp,\gq)&=&\dps{\sum_{i=1}^{n}\psi\pa{\sqrt{q_{i}\over \overline p_{i}}(x_{i})}\1_{A_{i}^{c}}(x_{i})},\\
\gZ_{\gA}(\bsX,\overline{\gp},\gq)&=&\gT_{\gA}(\bsX,\overline \gp,\gq)-\E\cro{\gT_{\gA}(\bsX,\overline \gp,\gq)}.
\end{array}\right.
\]
Define $\gT_{\gA}(\gx,\overline \gs ,\gu)$, $\gT_{\gA^{\!c}}(\gx,\overline \gs,\gu)$ and $\gZ_{\gA}(\gx,\overline \gs ,\gu)$ in the same way for $\gu\in\cbU$. Then
\[
\gZ(\bsX,\overline{\gp},\gq)=\gZ_{\gA}(\bsX,\overline{\gp},\gq)+
\left(\st\gT_{\gA^{\!c}}(\bsX,\overline \gp,\gq)-\E\cro{\gT_{\gA^{\!c}}(\bsX,\overline \gp,\gq)}\right),
\]
with a similar decomposition for $\gZ(\bsX,\overline \gs,\gu)$. Since $|\psi|\le 1$, 
%beg
\begin{equation}
\ab{\st\gT_{\gA^{\!c}}(\bsX,\overline \gp,\gq)}\le \sum_{i=1}^{n}\1_{A_{i}^{c}}(X_{i})\quad\mbox{and}\quad
\E\cro{\sum_{i=1}^{n}\1_{A_{i}^{c}}(X_{i})}=\sum_{i=1}^{n}P_{i}(A_{i}^{c}).
\label{eq-TAc}
\end{equation}
%end
Besides, since $\overline{P}_{i}(A_{i}^{c})=0$, if the measure $\lambda_{i}$ dominates both $P_{i}$ and $\overline{P}_{i}$,
%beg
\begin{equation}
P_{i}(A_{i}^{c})=\int_{A_{i}^{c}}\left(\sqrt{\frac{d\overline{P}_{i}}{d\lambda_{i}}}-\sqrt{\frac{dP_{i}}{d\lambda_{i}}}\right)^{2}d\lambda_{i}\le2h^{2}(\overline P_{i},P_{i}).
\label{eq-Aic}
\end{equation}
%end
This implies that
%beg
\begin{equation}
\left|\st\gT_{\gA^{\!c}}(\bsX,\overline \gp,\gq)-\E\cro{\gT_{\gA^{\!c}}(\bsX,\overline \gp,\gq)}\right|
\le \sum_{i=1}^{n}\1_{A_{i}^{c}}(X_{i})+2\gh^{2}(\overline \gP,\gP),
\label{eq-Aicb}
\end{equation}
%end
with the same bound for $\left|\st\gT_{\gA^{\!c}}(\bsX,\overline \gs,\gu)-\E\cro{\gT_{\gA^{\!c}}(\bsX,\overline \gs,\gu)}\right|$ when $\gu\in\cbU$. For all $\gQ\in\sbQ$, by (\ref{eq-nupp}),
$\gZ_{\gA}(\bsX,\overline \gp,\gq)=\gZ_{\gA}(\bsX,\overline \gs ,\gu)$, which, together with \eref{eq-Aicb} implies that, for all $\gQ=\gq\cdot \gmu=\gu\cdot \gnu$,
%beg
\begin{align*}
\ab{\gZ(\bsX,\overline{\gp},\gq)}\le&\; \ab{\gZ_{\gA}(\bsX,\overline \gp,\gq)}+\left|\st\gT_{\gA^{\!c}}(\bsX,\overline \gp,\gq)-\E\cro{\gT_{\gA^{\!c}}(\bsX,\overline \gp,\gq)}\right|\\
=&\; \ab{\gZ_{\gA}(\bsX,\overline \gs,\gu)}+\left|\st\gT_{\gA^{\!c}}(\bsX,\overline \gp,\gq)-\E\cro{\gT_{\gA^{\!c}}(\bsX,\overline \gp,\gq)}\right|\\
\le&\; \ab{\gZ(\bsX,\overline \gs,\gu)}+\left|\st\gT_{\gA^{\!c}}(\bsX,\overline \gp,\gq)-\E\cro{\gT_{\gA^{\!c}}(\bsX,\overline \gp,\gq)}\right|\\
&\; +\left|\st\gT_{\gA^{\!c}}(\bsX,\overline \gs,\gu)-\E\cro{\gT_{\gA^{\!c}}(\bsX,\overline \gs,\gu)}\right|\\
\le&\; \ab{\gZ(\bsX,\overline \gs,\gu)}+ 2\cro{\sum_{i=1}^{n}\1_{A_{i}^{c}}(X_{i})+2\gh^{2}
(\overline \gP,\gP)}.
\end{align*}
%end
Taking the supremum with respect to $\gQ\in \sbB^{\sbQ}(\gP,\overline \gP,y)$ then the expectation on both sides of the last inequality, we get 
\begin{align*}
w(\cR_{1},\sbQ,\gP,\overline \gP,y)& \le w(\cR_{2},\sbQ,\gP,\overline \gP,y)+2\cro{\E\cro{\sum_{i=1}^{n}\1_{A_{i}^{c}}(X_{i})}+2\gh^{2}(\overline \gP,\gP)}
\end{align*}
and we conclude by \eref{eq-TAc} and \eref{eq-Aic}.

%SUBSECTION
\subsection{Proof of Proposition~\ref{cas-fini}}\label{sect-pcas-fini}
Let $(\gmu, \cbQ)$ be an arbitrary representation of the finite set $\sbQ\cup\{\overline \gP\}$ so that we may write $\gQ=\gq\cdot \gmu$ and $\overline \gP=\overline \gp\cdot \gmu$ with $\gq,\overline \gp\in\cbQ$. Applying Proposition~50 of Baraud {\em et al.}~\citeyearpar{MR3595933} with $T=\sbB^{\sbQ}(\gP,\overline \gP,y)\subset \sbQ\cap \sbB(\gP,y)$, so that $\log_{+}(2|T|)\le \sbH(\sbQ,y)$ and
$U_{i,t}=\psi\pa{\sqrt{\left(q_{i}/\overline p_{i}\right)(X_{i})}}\in [-1,1]$, for which one may take $b=1$ and  $v^{2}=a_{2}^{2}y^{2}$ by~\eref{eq-var}, we  obtain that for all $y>0$ and $\gP,\overline \gP\in\sbP$
%beg
\begin{align*}
\gw^{\sbQ}(\gP,\overline \gP,y)&\le\sbH(\sbQ,y)+a_{2}y\sqrt{2\sbH(\sbQ,y)}
\\&=y^{2}\cro{{\sbH(\sbQ,y)\over y^{2}}+a_{2}\sqrt{{2\sbH(\sbQ,y)\over y^{2}}}}.
\end{align*}
%end
The inequality $x^{2}+\sqrt{2}a_{2}x\le a_{1}/8$ is satisfied for 
\[
0\le x\le \left(a_{2}/\sqrt{2}\right)\left[\sqrt{1+(\beta/a_{2})}-1\right]=\beta x_{0}^{-1}.
\]
It follows from the definition of $\overline \eta$ that, if $\beta y>\overline \eta$, 
\[
\sqrt{\sbH(\sbQ,y)}\le \beta x_{0}^{-1}y\qquad\mbox{hence}\qquad\gw^{\sbQ}\left(\gP,\overline \gP,y\right)\le(a_{1}/8)y^{2}.
\]
Consequently $D^{\sbQ}(\gP,\overline \gP)\le\overline \eta^{2}\vee 1$ by (\ref{eq-bound-D}). The second bound derives from (\ref{Eq-eta}).

%SUBSECTION
\subsection{Proof of Proposition~\ref{cas-VC}}\label{sect-pcas-VC}
By (\ref{eq-RB}), we may restrict ourselves to the case of $\overline{V}\le n/6$.
The proof being similar to that of Theorem~12 in Baraud {\em et al.}~\citeyearpar{MR3595933}, we only provide here a sketch of proof of the result. Let $\overline \gP\in\sbP^{\gmu}$. We may write $\overline \gP=\overline \gp\cdot \gmu$ with $\overline \gp\in\cbL(\gmu)$, and represent $\sbQ$ by $(\gmu,\cbQ)$ for some subset $\cbQ\subset \overline \cbQ$. Let us consider the pair $(\gmu,\cbQ\cup\{\overline \gp\})$ for representing $\sbQ\cup\{\overline \gP\}$.

Since $\psi$ is monotone and $\overline \cbQ$ is VC-subgraph on $\overline \sX$ with index not larger than $\overline V$ so are the set $\ac{\psi\pa{\sqrt{\gq/\overline \gp}},\ \gq\in \cbQ}$ and its subset 
\begin{equation}\label{eq-sF}
\sF^{{\cbQ}}(\gP,\overline \gP,y)=\ac{\left.\psi\pa{\sqrt{\gq/\overline \gp}}\,\right|\,\gQ\in\sbB^{\sbQ}(\gP,\overline \gP,y)}\quad \mbox{for all $y>0$.}
\end{equation}
Since the elements of $\sF^{{\cbQ}}(\gP,\overline \gP,y)$ are bounded by 1, it follows from Theorem~2.6.7 in van der Vaart and Wellner~\citeyearpar{MR1385671} that, for some numerical constant $K$ and any probability measure $R$ on $\overline \sX$, the $\L_{2}(R)$-entropy $\sH(\sF^{\cbQ},R,\cdot)$ of $\sF^{\cbQ}(\gP,\overline \gP,y)$ satisfies for all $\eps\in (0,1)$
\[
\sH(\sF^{\cbQ},R,\eps)\le\log\pa{K\overline V(16e)^{\overline V}}+2(\overline V-1)\log(1/\eps)\le2\overline V\log(A/\eps)
\]
for some numerical constant $A\ge 2e$. Applying Lemma~49 in Baraud {\em et al.}~\citeyearpar{MR3595933} to the random variables $(i,X_{i})\in\overline \sX$, $\sF=\sF^{\cbQ}(\gP,\overline \gP,y)$, $v^{2}=a_{2}^{2}y^{2}$ by~\eref{varT} and $\overline \sH(z)=2\overline V\log_{+}(Az)$ with $z\ge1/2$ (so that $L\le3/2$, according to the proof of Theorem~12 in Baraud {\em et al.}~\citeyearpar{MR3595933}), we get, for some numerical constant $C_{0}>0$,
\[
\gw^{\sbQ}(\gP,\overline \gP,y)\le C_0\cro{a_{2}y\sqrt{H}+H}\quad\mbox{with}\quad H=\overline\sH\pa{{\sqrt{n}\over2a_{2}y}\bigvee\frac{1}{2}}.
\]
Let $D\ge \overline V$ to be chosen later and $y\ge \beta^{-1}\sqrt{D}$. Since $\beta\le 1$, $a_{2}\ge 1$, $\overline V\le n$ and $A\ge 2e$ we deduce that $y\ge \sqrt{\overline V}$, 
\[
H\le\overline{H}=\overline\sH\pa{{\sqrt{n}\over2\sqrt{\overline{V}}}}=2\overline V\log\left(\frac{A\sqrt{n}}{2\sqrt{\overline{V}}}\right)\quad\mbox{and}\quad \overline{H}\ge2\overline{V}.
\]
For all $y\ge \beta^{-1}\sqrt{D}$,
%beg
\begin{align*}
\gw^{\sbQ}(\gP,\overline \gP,y)&\le C_0\cro{a_{2}y\sqrt{\overline{H}}+\overline{H}}={ a_{1}\over 8}y^{2}C_{0}\cro{{8a_{2}\over a_{1}}{\sqrt{\overline{H}}\over y}+{8\over a_{1}}{\overline{H}\over y^{2}}}\\
&\le{ a_{1}\over 8}y^{2}C_{0}\cro{{8a_{2}\beta\over a_{1}}{\sqrt{\overline{H}\over D}}+{8\beta^{2}\over a_{1}}{\overline{H}\over D}}.
\end{align*}
%end
Since $\beta=a_{1}/(4a_{2})$ and $a_{2}^{2}\ge 6 a_{1}$, we derive that
\[
\gw^{\sbQ}(\gP,\overline \gP,y)\le { a_{1}\over 8}y^{2}C_{0}\cro{2{\sqrt{\overline{H}\over D}}+{a_{1}\over 2a_{2}^{2}}{\overline{H}\over D}}\le { a_{1}\over 8}y^{2}C_{0}\cro{2{\sqrt{\overline{H}\over D}}+{\overline{H}\over 12 D}}.
\]
The inequality $2u+u^{2}/12\le C_{0}^{-1}$ being satisfied for $u\in [0, \overline u]$ with 
\[
\frac{1}{\overline u}=C_{0}\pa{\sqrt{1+{1\over 12 C_{0}}}+1},
\]
we deduce that, for 
\[
D=\pa{\frac{1}{\overline u^{2}}\vee {1\over 2}}\overline H\ge \max\left\{\frac{\overline H}{\overline u^{2}}, \overline V\right\}
\]
and all $y\ge \beta^{-1}\sqrt{D}$, $\gw^{\sbQ}(\gP,\overline \gP,y)\le a_{1}y^{2}/8$.  We conclude by (\ref{eq-bound-D}).

%SUBSECTION
\subsection{Proof of Proposition~\ref{cas-extreme}}\label{sect-pcas-extreme}
The us fix some $d\in \cD$ and $\overline \gP\in\overline \sbQ_{d}$. The $\rho$-model $\sbQ\cup\{\overline \gP\}$ can be represented by $(\gmu, \cbQ)$ where $\cbQ\subset \overline \cbQ$ contains a density $\overline \gp$ which is extremal in $\overline \cbQ$ with degree $d$ and satisfies $\overline \gP=\overline \gp\cdot \gmu$. The function $\psi$ being monotone, it follows from Proposition~3 of Baraud~\citeyearpar{Bar2016} that the class of functions $\sF^{{\cbQ}}(\gP,\overline \gP,y)$ defined by~\eref{eq-sF} which satisfies 
\[
\sF^{{\cbQ}}(\gP,\overline \gP,y)\subset \ac{\psi(g),\ g\in (\overline{\cbQ}/\overline \gp)}
\]
is weak VC-major on $\overline \sX=\bigcup_{i=1}^{n}\left(\{i\}\times\sX_{i}\right)$ with dimension not larger than $d\ge 1$. Besides, since $\psi$ takes its values in $[-1,1]$, $\sF^{{\cbQ}}(\gP,\overline \gP,y)$ is uniformly bounded by 1. Applying Corollary~1 of Baraud~\citeyearpar{Bar2016} to $\sF^{{\cbQ}}(\gP,\overline \gP,y)$ with $b=1$ and $\sigma^{2}=(a_{2}^{2}y^{2}/n)\wedge 1$ (because of~\eref{varT}) and setting
\[
\overline \Gamma(d)=\log\pa{2\sum_{j=0}^{d}\binom{n}{j}}\le
\log 2+d\log\pa{en\over d}\le dL
\]
with  $L=\log\pa{e^{2}n\over d}\ge2$, we get
%beg
\begin{align*}
\gw^{\sbQ}(\gP,\overline \gP,y)&\le\left[4\sqrt{2n\overline \Gamma(d)}\times \sigma\log\pa{e\over\sigma}\right]+16\overline \Gamma(d)\\
&\le\left[4\sqrt{2ndL}\times \sigma\log\pa{e\over\sigma}\right]+16dL\\
&\le4a_{2}y\sqrt{2dL}\log\pa{e\sqrt{{n\over a_{2}^{2}y^{2}}\vee 1}}+16dL.
\end{align*}
%end
Let $D\ge d$ to be chosen later on. Since $a_{2}\ge 1$, $\beta=a_{1}/(4a_{2})\le 1$ and $d\le n$, for all $y\ge \beta^{-1}\sqrt{D}\ge \sqrt{d}$, 
\[
\log\pa{e\sqrt{{n\over a_{2}^{2}y^{2}}\vee 1}}\le \log\pa{e\sqrt{n\over d}}={L\over 2},
\]
hence, since $L\ge2$,
%beg
\begin{align*}
\lefteqn{\gw^{\sbQ}(\gP,\overline \gP,y)}\hspace{5mm}\\
&\le2a_{2}y\sqrt{2dL^{3}}+16dL\le 2a_{2}y\sqrt{2dL^{3}}+4dL^{3}\\
&={a_{1}\over 8}y^{2}\cro{{16a_{2}\over a_{1}}\sqrt{dL^{3}\over y^{2}}+{32\over a_{1}}{2dL^{3}\over y^{2}}}\le{a_{1}\over 8}y^{2}\cro{{16a_{2}\beta\over a_{1}}\sqrt{2dL^{3}\over D}+{32\beta^{2}\over a_{1}}{dL^{3}\over D}}\\
&={a_{1}\over 8}y^{2}\cro{4\sqrt{2dL^{3}\over D}+{2a_{1}\over a_{2}^{2}}{dL^{3}\over D}}\le{a_{1}\over 8}y^{2}\cro{4\sqrt{2dL^{3}\over D}+{dL^{3}\over 3D}},
\end{align*}
%end
%
since $a_{2}^{2}\ge 6 a_{1}$. The inequality $4\sqrt{2}u+u^{2}/3\le 1$ is satisfied for $u\in [0,\overline u]$ with $\overline u=1/\sqrt{33}$ and consequently, for $D=33dL^{3}\ge d$ and all $y\ge \beta^{-1}\sqrt{D}$, $\gw^{\sbQ}(\gP,\overline \gP,y)\le a_{1}y^{2}/8$. The conclusion then follows from (\ref{eq-bound-D}).

%\end{supplement}
%\begin{supplement}
%  \sname{Supplement D}\label{sect-CEX-MLE}
%  \stitle{Influence of the choices of the densities on the performance of the MLE}
%  \slink[doi]{??}
%  \sdescription

%\section{Proofs and comments relative to Section~\ref{F}}\label{sect-CEX-MLE}

%\begin{supplement}
%  \sname{Supplement E}\label{Hpsi1}
%  \stitle{Some comments on Assumptions~\ref{ass-psi} and \ref{H-debase}}
%  \slink[doi]{??}
%  \sdescription{

%
%\end{supplement}
%\begin{supplement}
%\sname{Supplement F}\label{Hpsi2}
%\stitle{Connection with robust tests}
%\slink[url]{??}
%\sdescription{

%}
%\end{supplement}

%\begin{supplement}
%\sname{Supplement G}\label{Sect-UnivSep}
%\stitle{Universal separability}
%\slink[url]{??}
%\sdescription{%PARAGRAPH

%SUBSECTION
\subsection{Proof of Corollary~\ref{cor-reg}}\label{sect-pcor-reg}
Let us fix some arbitrary $m=(r,F)\in\cM$ and $g\in F$.
Under Assumption~\ref{ass-2} and~\eref{eq-hel}, for all $r'\in\cD$, $g'\in\sF$ and $w\in\sW$,
\[
h^{2}(r,r')= h^{2}(R,R')\le A^{2}h^{2}\!\left(R_{g(w)-g'(w)},R'\right)=A^{2}h^{2}\!\left(R_{g(w)},R'_{g'(w)}\right).
\]
Integrating this inequality with respect to $P_{W}$ gives
\begin{equation}\label{eq-125}
h(r,r')\le Ah(Q_{r,g},Q_{r',g'})\quad\mbox{for all }g'\in\sF.
\end{equation}
For all $r'\in\cD$, we deduce from~\eref{eq-125},~\eref{eq-hel} and~\eref{eq-hel2} that
%beg
\begin{align*}
d_{s}(f,g')&=h(Q_{s,f},Q_{s,g'})\le h(Q_{s,f},Q_{r',g'})+ h(Q_{r',g'},Q_{r,g'})+h(Q_{r,g'},Q_{s,g'})\\
&\le h(Q_{s,f},Q_{r',g'})+h(Q_{r',0},Q_{r,0})+h(Q_{r,0},Q_{s,0})\\
&=h(Q_{s,f},Q_{r',g'})+h(r,r')+h(s,r)\\
&\le h(Q_{s,f},Q_{r',g'})+Ah(Q_{r',g'},Q_{r,f})+h(s,r)\\
&\le h(Q_{s,f},Q_{r',g'})+A\cro{h(Q_{r',g'},Q_{s,f})+h(Q_{s,f},Q_{r,f})}+h(s,r)\\
&= h(Q_{s,f},Q_{r',g'})+A\cro{h(Q_{r',g'},Q_{s,f})+h(s,r)}+h(s,r)\\
&\le(1+A)\!\cro{h(Q_{s,f},Q_{r',g'})+h(s,r)}\\&=(1+A)\!\cro{h(P,Q_{r',g'})+h(s,r)}\!.
\end{align*}
%end
Taking $r'=\widehat s$ and $g'=\widehat f$, we get
\begin{equation}\label{eq-120}
d_{s}(f,\widehat f)\le (1+A)\cro{h(P,\widehat P)+h(s,r)}\quad\mbox{for all }r\in\cD.
\end{equation}
Besides, 
\[
h(P,\sQ_{m})\le h(Q_{s,f},Q_{r,g})\le h(Q_{s,f},Q_{s,g})+h(Q_{s,g},Q_{r,g})=d_{s}(f,g)+h(s,r),
\]
and it follows from Theorem~\ref{thm-main2} that, on a set $\Omega_{\xi}$ of probability at least $1-e^{-\xi}$,
\[
Ch^{2}(P,\widehat P)\le d_{s}^{2}(f,g)+h^{2}(s,r)+n^{-1}[D_{n}(m)+\Delta(m)+\xi],
\]
with $D_{n}(m)$ given by (\ref{def-Dm}). Hence, by \eref{eq-120}, on $\Omega_{\xi}$, 
%beg
\begin{align}
d_{s}^{2}(f,\widehat f)\le&\;  2(1+A)^{2}\cro{h^{2}(P,\widehat P)+h^{2}(s,r)}\nonumber\\
\le&\;\frac{2(1+A)^{2}}{C}\cro{d_{s}^{2}(f,g)+h^{2}(s,r)+\frac{D_{n}(m)+\Delta(m)+\xi}{n}}\nonumber\\
&\;+2(1+A)^{2}h^{2}(s,r)\nonumber\\
\le&\; C'\left[d^{2}_{s}(f,g)+h^{2}(s,r)+n^{-1}(D_{n}(m)+\Delta(m)+\xi)\right],
\label{eq-127}
\end{align}
%end
where $C'$ depends on $A$ and $\psi$. Using now~\eref{eq-125} with $r'=\widehat s$ and $g'=\widehat f$, we deduce that
%beg
\begin{align*}
h(s,\widehat s)&\le h\pa{s,r}+h(r,\widehat s)\le h\pa{s,r}
+Ah(Q_{r,f},Q_{\widehat s,\widehat f})\\&\le h\pa{s,r}+A\cro{h(Q_{r,f},Q_{s,f})+h(Q_{s,f},Q_{\widehat s,\widehat f})}\\
&\le(1+A)\cro{h(s,r)+h(P,\widehat P)}.
\end{align*}
%end
This bound is the same as the one established in~\eref{eq-120} for $d_{s}(f,\widehat f)$. Arguing as before we derive that, on the same event $\Omega_{\xi}$, $h^{2}(s,\widehat s)$ is also not larger than the right-hand side of~\eref{eq-127}. The conclusion follows since $m=(r,F)$ is arbitrary in $\cM$. 

\subsection{Connection with robust tests}\label{Hpsi2}
As already mentioned in Section~\ref{ES2}, for all $\gQ$ and $\gQ'$ in ${\sbQ}$,
\[
a_{1}\gh^{2}(\gP,\gQ)-a_{0}\gh^{2}(\gP,\gQ')\le \E\cro{\gT(\bsX,\gq,\gq')}\le a_{0}\gh^{2}(\gP,\gQ)-a_{1}\gh^{2}(\gP,\gQ').
\]
The right-hand side of the inequality shows that $\E\cro{\gT(\bsX,\gq,\gq')}<0$ if $\gh^{2}(\gP,\gQ)<a_{1}a_{0}^{-1}\gh^{2}(\gP,\gQ')$ while the left-hand side shows that it is positive when $\gh^{2}(\gP,\gQ')<a_{1}a_{0}^{-1}\gh^{2}(\gP,\gQ)$. In particular if $\gT(\bsX,\gq,\gq')$ is close enough to its expectation, its sign may be used as a test statistic to decide which of the two probabilities $\gQ$ and $\gQ'$ is closer to $\gP$. Bounds for the probabilities of errors of this test are provided by the following lemma.
\begin{lem}\label{lem-test}
Let $\psi$ satisfy Assumption~\ref{H-debase}, $\gP\in\sbP$, $\gQ,\gQ'\in\sbQ$ and $x\ge 0$. Whatever the representation $\cR(\sbQ)$ of the $\rho$-model $\sbQ$,
if $\gh^{2}(\gP,\gQ)<a_{1}a_{0}^{-1}\gh^{2}(\gP,\gQ')$, then
\begin{align*}
\lefteqn{\P\cro{\gT(\bsX,\gq,\gq')\ge x}}\hspace{10mm}\\
&\le \exp\cro{-\pa{a_{1}\gh^{2}(\gP,\gQ')-a_{0}\gh^{2}(\gP,\gQ)+x}^{2}\over 2\cro{(a_{2}^{2}+a_{1}/3)\gh^{2}(\gP,\gQ')+(a_{2}^{2}-a_{0}/3)\gh^{2}(\gP,\gQ)+x/3}}
\end{align*}
while if  $\gh^{2}(\gP,\gQ')<a_{1}a_{0}^{-1}\gh^{2}(\gP,\gQ)$, then
\begin{align*}
\lefteqn{\P\cro{\gT(\bsX,\gq,\gq')\le -x}}\hspace{10mm}\\
&\le \exp\cro{-\pa{a_{1}\gh^{2}(\gP,\gQ)-a_{0}\gh^{2}(\gP,\gQ')+x}^{2}\over 2\cro{(a_{2}^{2}+a_{1}/3)\gh^{2}(\gP,\gQ)+(a_{2}^{2}-a_{0}/3)\gh^{2}(\gP,\gQ')+x/3}}.
\end{align*}
\end{lem}
\begin{proof}
Using the right-hand side of~\eref{eq-ET}, we may write
\begin{align*}
\lefteqn{\P\cro{\gT(\bsX,\gq,\gq')\ge x}}\hspace{5mm}\\
&=\P\cro{\gT(\bsX,\gq,\gq')-\E\cro{\gT(\bsX,\gq,\gq')}\ge x-\E\cro{\gT(\bsX,\gq,\gq')}\str{3.7}}\\
&\le\P\cro{\gT(\bsX,\gq,\gq')-\E\cro{\gT(\bsX,\gq,\gq')}\ge a_{1}\gh^{2}(\gP,\gQ')-a_{0}\gh^{2}(\gP,\gQ)+x\str{3.7}}.
\end{align*}
We obtain the first inequality by applying the Bernstein deviation inequality (see Massart~\citeyearpar{MR2319879}, inequality~(2.16)) to $\gT(\bsX,\gq,\gq')-\E\cro{\gT(\bsX,\gq,\gq')}$. Indeed, $\gT(\bsX,\gq,\gq')$ is a sum of independent random variables with absolute value not larger than 1 (since $|\psi|$ is bounded by 1) and by~\eref{eq-var}, 
\begin{equation}\label{varT}
\sum_{i=1}^{n}\int_{\sX_{i}}\psi^{2}\pa{\sqrt{q'_{i}\over q_{i}}}dP_{i}\le a_{2}^{2}\cro{\gh^{2}(\gP,\gQ)+\gh^{2}(\gP,\gQ')}.
\end{equation}
To complete the proof, note that the second inequality is a consequence of the first one by exchanging the roles of $\gQ$ and $\gQ'$ and using~\eref{eq-antisym}. 
\end{proof}

\end{document}